\documentclass[12pt,pdftex,noinfoline]{imsart}

\RequirePackage[OT1]{fontenc}
\usepackage{amsthm,amsbsy,amsmath,amsfonts,natbib,mathtools,amssymb}
\usepackage{mathrsfs}
\RequirePackage{hypernat}
\usepackage[ruled, vlined, lined, commentsnumbered]{algorithm2e}
\usepackage{graphicx}
\usepackage{verbatim}
\usepackage{times}
\usepackage{hyperref}
\usepackage[usenames,dvipsnames,svgnames,table]{xcolor}
\hypersetup{citecolor=MidnightBlue}
\hypersetup{linkcolor=MidnightBlue}
\hypersetup{urlcolor=MidnightBlue}
\usepackage{enumerate}
\usepackage{fullpage}
\usepackage[margin=1in,footskip=.40in]{geometry}
\usepackage{dsfont}

\newcommand{\E}{\operatorname{\mathbb E}}

\numberwithin{equation}{section}
\newtheorem{thm}{Theorem}[section]
\newtheorem{lemma}{Lemma}[section]
\newtheorem{corollary}{Corollary}[section]

\theoremstyle{remark}

\newtheorem{remark}{Remark}[section]

\newtheorem*{thm1}{Claim A}
\newtheorem*{thm2}{Claim B}
\newtheorem*{thm3}{Claim C}
\newtheorem*{thm4}{Claim D}
\def\conditionA{Condition $A$}
\def\conditionAp{Condition $\tilde A$}
\def\conditionB{Condition $B$}
\newtheorem*{con0}{\conditionA}
\newtheorem*{con1}{\conditionAp}
\newtheorem*{con2}{\conditionB}

\def\given{\,|\,}

\def\E{\mathbb{E}}
\def\reals{\mathbb{R}}
\def\ones{\mathds{1}}

\let\hat\widehat

\let\tilde\widetilde
\def\Var{\textsf{Var}}

\newcommand{\argmin}{\mathop{\rm argmin}}

\newcommand{\wh}{\widehat}
\newcommand{\wt}{\widetilde}

\newcommand{\fnorm}[1]{\|#1\|_{\rm F}}
\newcommand{\opnorm}[1]{\|#1\|_{\rm op}}

\newcommand{\sgn}{\mathop{\rm sign}}

\newcommand{\iprod}[2]{\left \langle #1, #2 \right\rangle}

\newcommand{\TV}{{\sf TV}}
\newcommand{\blue}{\color{blue}}
\newcommand{\nb}[1]{\texttt{\blue[#1]}}

\begin{document}

\setlength{\parskip}{0.5em}
\begin{frontmatter}
\runtitle{Model Repair}
\title{Model Repair: Robust Recovery of Over-Parameterized Statistical Models}
%\affil[**]{Department of Statistics and Data Science, Yale University}
\begin{aug}
\vskip15pt
\address{
\begin{tabular}{ccc}
{\normalsize\rm\bfseries Chao Gao} && {\normalsize\rm\bfseries John Lafferty}\\[5pt]
Department of Statistics && Department of Statistics and Data Science\\
University of Chicago && Yale University\\[10pt]
\end{tabular}
\\[10pt]
\today
\vskip10pt
}
\end{aug}
% !TEX root = ./repair.tex

\begin{abstract}
A new type of robust estimation problem is introduced where the goal is to recover a statistical model that has been corrupted after it has been estimated from data. Methods are proposed for ``repairing'' the model using only the design and not the response values used to fit the model in a supervised learning setting. Theory is developed which reveals that two important ingredients are necessary for model repair---the statistical model must be over-parameterized, and the estimator must incorporate redundancy. In particular, estimators based on stochastic gradient descent are seen to be well suited to model repair, but sparse estimators are not in general repairable. After formulating the problem and establishing a key technical lemma related to robust estimation, a series of results are presented for repair of over-parameterized linear models, random feature models, and artificial neural networks. Simulation studies are presented that corroborate and illustrate the theoretical findings.
\end{abstract}

\end{frontmatter}

% !TEX root = ./repair.tex

\section{Introduction}
\label{sec:intro}

In this paper we introduce a new type of robust estimation problem---how to recover a statistical
model that has been corrupted after estimation. Traditional robust estimation assumes that the data are corrupted, and studies methods of estimation that are immune to these corruptions or outliers in the data.
In contrast, we explore the setting where the data are ``clean'' but a statistical model is corrupted after it has been estimated using the data. We study methods for recovering the model that do not require re-estimation from scratch, using only the design and not the original response values.

The problem of model repair is motivated from several different perspectives. First, it can be formulated as a well-defined statistical problem that is closely related to, but different from, traditional robust estimation, and that deserves study in its own right. From a more practical perspective, modern machine learning practice is increasingly working with very large statistical models. For example, artificial neural networks having several million parameters are now routinely estimated. It is anticipated that neural networks having trillions of parameters will be built in the coming years, and that large models will be increasingly embedded in systems, where they may be subject to errors and corruption of the parameter values. In this setting, the maintenance of models in a fault tolerant manner becomes a concern.  A different perspective takes inspiration from plasticity in brain function, with the human brain in particular having a remarkable ability to repair itself after trauma. The framework for model repair that we introduce in this paper can be viewed as a simple but mathematically rigorous formulation of this ability in neural networks.

At a high level, our findings reveal that two important ingredients are necessary for model repair. First, the statistical model must be over-parameterized, meaning that there should be many more parameters than observations.
While over-parameterization leads to issues of identifiability from traditional perspectives, here it is seen as a necessary property of the model. Second, the estimator must incorporate redundancy in some form; for instance, sparse estimators of over-parameterized models will not in general be repairable. Notably, we show that estimators based on gradient descent and stochastic gradient descent are well suited to model repair.

At its core, our formulation and analysis of model repair rests upon representing an estimator in terms of the row space of functions of the data design matrix. This leads to a view of model repair as a form of robust estimation. The recovery algorithms that we propose are based on solving a linear program that is equivalent to median regression. Our key technical lemma, which may be of independent interest, gives sharp bounds on the probability that this linear program successfully recovers the model, which in turn determines the level of over-parameterization that is required. An interesting facet of this formulation is that the response vector is not required by the repair process. Because the model is over-parameterized, the estimator effectively encodes the response. This phenomenon can be viewed from the perspective of communication theory, where the corruption process is seen as a noisy channel, and the design matrix is seen as a linear error-correcting code for communication over this channel.

After formulating the problem  and establishing the key technical lemma, we present a series of results for
repair of over-parameterized linear models, random feature models, and artificial neural networks. These form the main technical contributions of this paper. A series of simulation experiments are presented that corroborate and illustrate our theoretical results. In the following section we give a more detailed overview of our results, including the precise formulation of the model repair problem, its connection to robust estimation and error correcting codes, and an example of the repair algorithm in simulation. We then present the key lemma, followed by detailed analysis of model repair for specific model classes.
We present the proof of the key lemma in Section \ref{sec:regression}, and the proofs of the neural network repair results with hyperbolic tangent activation are given in Section \ref{sec:pf-nn-repair}. Proofs of technical lemmas and the results for  neural networks with ReLU activation are presented in the appendix, to make the presentation more readable.
Section~\ref{sec:discuss} gives a discussion of directions for further research and potential implications of our findings for applications.

% !TEX root = ./repair.tex
\let\epsilon\varepsilon
\def\point#1{\vskip10pt\noindent{\it\bfseries #1.}\enspace}
\def\L{{\mathcal L}}

\section{Problem formulation and overview of results}
\label{sec:overview}

In this section we formulate the problem of model repair, and give an overview of our results. Suppose that $\hat\theta \in \reals^p$ is a model with $p$ parameters estimated on $n$ data points $\{(x_i, y_i)\}_{i=1}^n$ as a classification or regression model. The model $\hat\theta$ is then corrupted by noise. The primary noise model we study in this paper is
\begin{equation}
  \eta = \hat\theta + z
\end{equation}
where $z_j \sim (1-\epsilon) \delta_0 + \epsilon Q_j$ and $Q_j$ is an arbitrary distribution. In other words, each component $\hat\theta_j$ of $\hat\theta$ is corrupted by additive noise from an arbitrary distribution $Q_j$ with probability $\epsilon$, where $0\leq \epsilon \leq 1$, and is uncorrupted with probability $1-\epsilon$. The noise vector $z$ is assumed to be independent of the design $\{x_i\}_{i=1}^n$. We discuss alternative error models later in the paper. The goal is to recover $\hat\theta$ from $\eta$, without reestimating the model from scratch; in particular, without using the response values $\{y_i\}$ when the model is estimated in a supervised learning setting.

\point{Overparameterized linear models}
To explain the main ideas, let us first consider the setting of under-determined linear regression. Let $X\in\reals^{n\times p}$ be the design matrix and $y\in\reals^n$ a vector of response values, and suppose that we wish to minimize the squared error  $\|y-X\theta\|_2^2$. If $n > p$ then this is an under-determined optimization problem. Among all solutions to the linear system $y = X\theta$, the solution of minimal norm $\|\theta\|_2$ is given by
\begin{equation}
  \hat \theta = X^T (X X^T)^{-1} y \label{eq:min-norm-solution}
\end{equation}
assuming that $X$ has full rank $n$ \citep{boyd:04}. Thus, $\hat\theta$ lies in the row space of the $n\times p$ design matrix $X$. The risk behavior of (\ref{eq:min-norm-solution}) has been well studied in the recent literature \citep{belkin2019two,hastie2019surprises,bartlett2020benign}.

Now suppose that $\eta = \hat\theta + z$ where $z_j \sim (1-\epsilon) \delta_0 + \epsilon Q_j$. The method we propose to recover $\hat\theta$ from $\eta$ is to let $\tilde u\in\reals^n$ be the solution to the optimization
\begin{equation}
  \tilde u = \argmin_u \| \eta  - X^T u\|_1
  \label{eq:keylp}
\end{equation}
and define the repaired model as $\tilde\theta = X^T\tilde u$.
The linear program defined in \eqref{eq:keylp} can be thought of as performing median regression of $\eta$ onto the rows of $X$.
Our analysis shows that, under appropriate assumptions, the estimated model is exactly recovered with high probability, so that $\tilde \theta = \hat\theta$, as long as $n/p \leq c(1-\epsilon)^2$ for some sufficiently small constant $c$.

Figure~\ref{fig:exp} shows the performance of
the repair algorithm in simulation. The design is sampled as $x_{ij} \sim N(0,1)$ and the corruption distribution is $Q_j = N(1,1)$ for each $j$. With the sample size fixed at $n=50$, the dimension $p$ is varied according to $p_k/n=200/k^2$ with
$k$ ranging from 1 to 6. The plots show the empirical probability of exact repair $\tilde\theta = \hat\theta$ as a function of $\epsilon$. The roughly equal spacing of the curves agrees with our theory, which indicates that $\sqrt{n/p}/(1-\epsilon)$ should be sufficiently small for successful repair. The theory indicates that the repair probability for dimension $p_k$ as a function of the adjusted value $\epsilon_k = \epsilon + c'\cdot k - \frac{1}{2}$ should exhibit a threshold at
$\epsilon_k = 1/2$ for the constant $c' = \frac{\sqrt{2}}{20 c}$; this is seen in the right plot of Figure~\ref{fig:exp}.

\point{Robust regression} This procedure can be viewed in terms of robust regression. Specifically, $\eta$ can be viewed as a corrupted response vector, and $A = X^T \in \reals^{p\times n}$ can be viewed as a design matrix that is \textit{not corrupted}.  Our result makes precise conditions under which this robust regression problem can be successfully carried out.
In particular, we show that model repair is possible even if $\epsilon \to 1$, so that the proportion of corrupted model components approaches one. This is in contrast to the traditional Huber model where the design is also corrupted \citep{huber:64}, under which consistent estimation is only possible if $\epsilon$ is below some small constant \citep{chen2016,gao2020}. The problem of robust regression with uncorrupted design has a rich literature; we review some of the relevant work in Section \ref{sec:regression}.

\begin{figure}[t]
  \begin{tabular}{cc}
    \hskip-3pt
    \includegraphics[width=.48\textwidth]{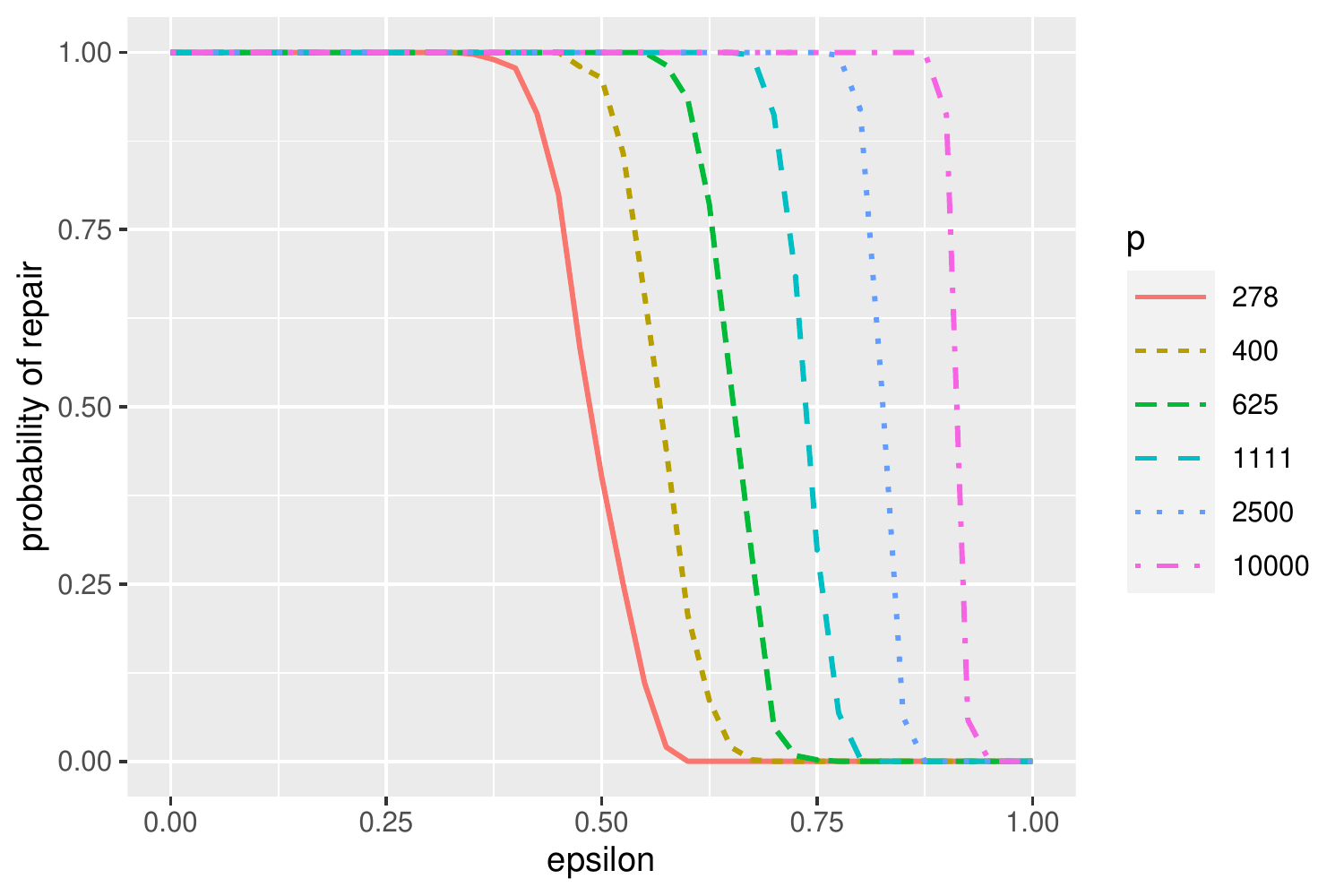} &
    \hskip-3pt
    \includegraphics[width=.48\textwidth]{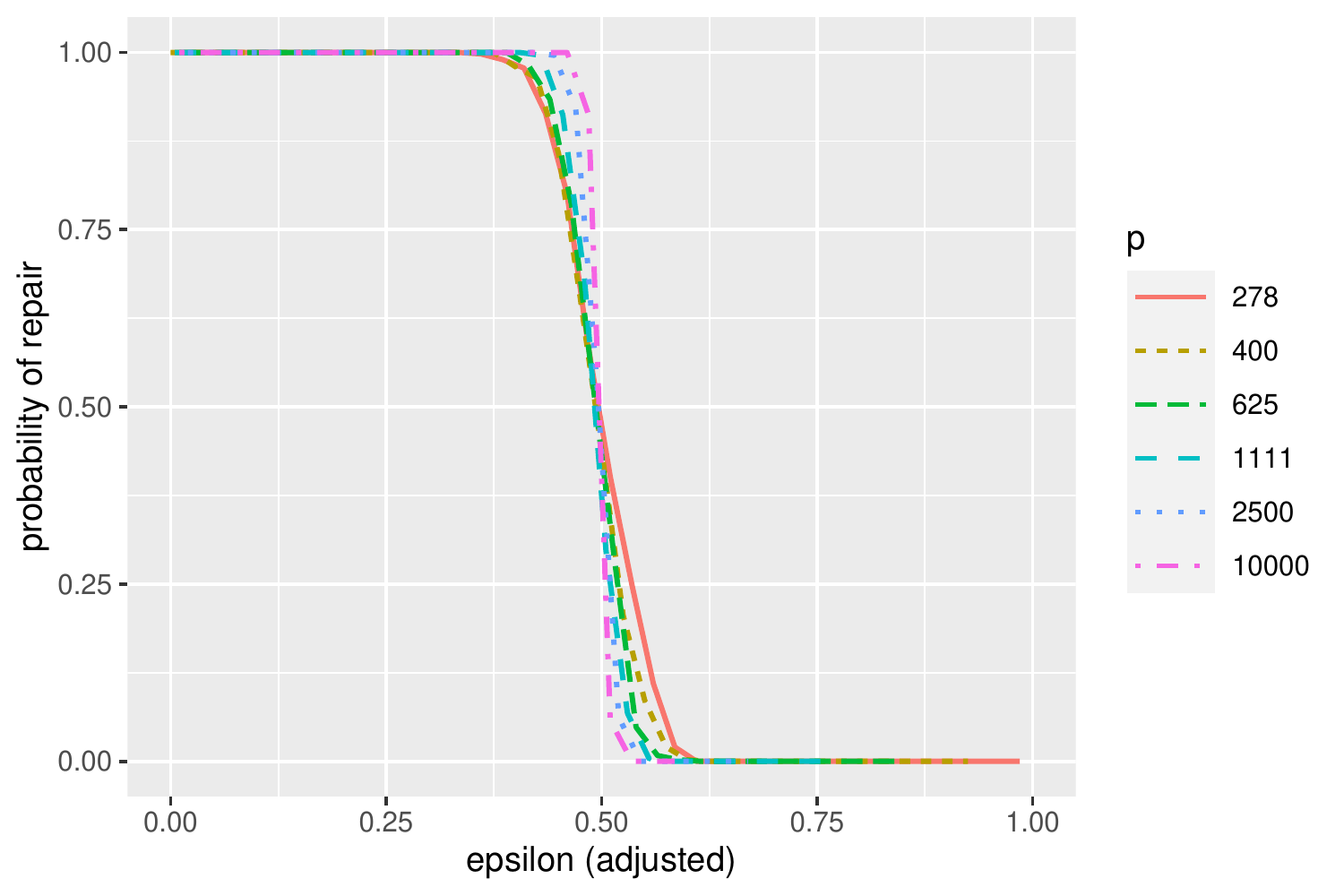}
  \end{tabular}
\caption{Left: Empirical probability of exact repair as a function of $\epsilon$.
The sample size is $n=50$ and the model dimension $p$
varies as $p_k/n = 200 /k^2$, for $k=1,\ldots, 6$; each point is an average over 500 trials. The plot on the right
shows the repair probability as a function of the adjusted value $\epsilon_k = \epsilon + c' \cdot k - \frac{1}{2}$
for dimension $p_k$, where the constant is $c'=\frac{\sqrt{2}}{20 c}=0.085$.}
\label{fig:exp}
\end{figure}

\point{Error-correcting codes}
Model repair can also be viewed in terms of error-correcting codes. Specifically, viewing the response vector $y\in\reals^n$  as a message to be communicated over a noisy channel, the minimum norm model $\hat\theta = X^T u = X^T (X X^T)^{-1} y$ redundantly
encodes $y$ since $p > n$ (see Figure~\ref{fig:code}). The decoding algorithm $\tilde u = \argmin_u \|\eta - X^T u\|$ then recovers the data $y$ according to $y = (XX^T) \tilde u$. The inequality $n/p < c(1-\epsilon)^2$ gives a condition on the rate of the code, that is, the level of redundancy that is sufficient for this decoding procedure to recover the message with high probability.

When $X$ is a random Gaussian matrix, the mapping $u \to X^T u = \sum_{i=1}^n u_i X_i^T$ can be viewed as a superposition of random codewords in $\reals^p$ \citep{joseph2012,rush2017}. The fundamental difference with channel coding is that in our regression setting the design matrix $X$ is fixed, and is not chosen for optimal channel coding. Indeed, the noise model $w \to w + z$ that we consider, with $z_j \sim (1-\epsilon)\delta_0 +\epsilon Q_j$ corresponds to a channel having infinite capacity, and a simple repetition code would suffice for identifying components that are uncorrupted \citep{CoverThomas}.

\begin{figure}[t]
  \includegraphics[width=1.0\textwidth]{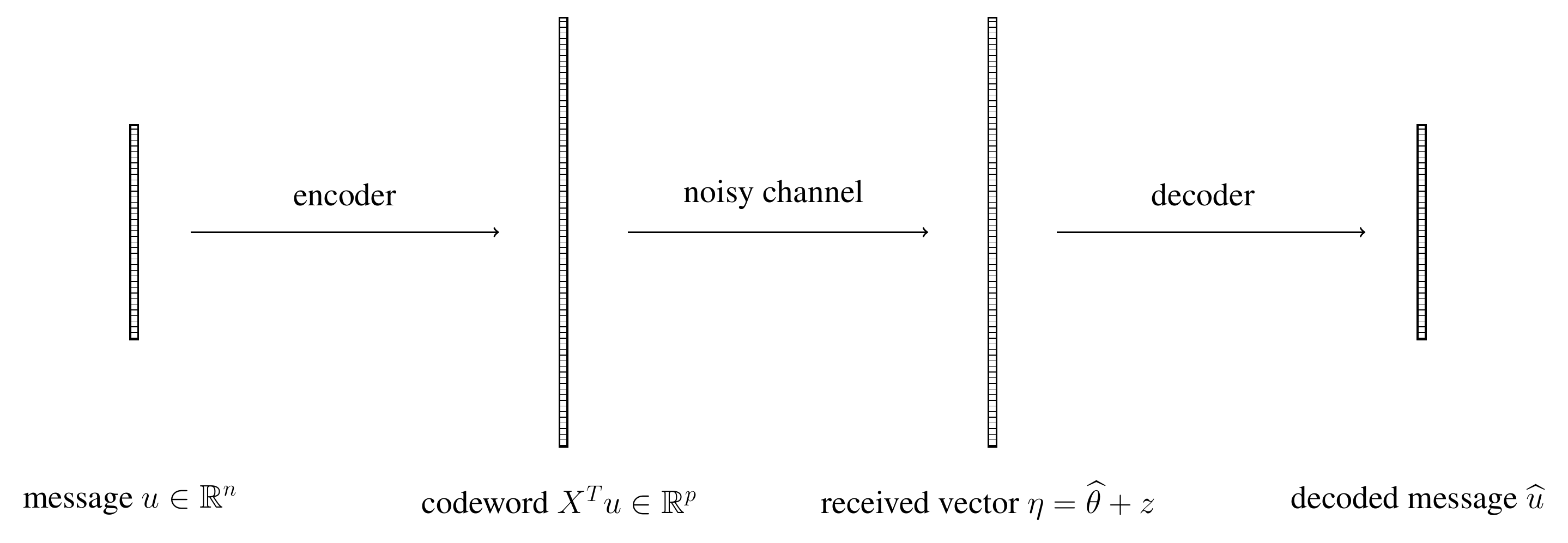}
  \caption{Model repair viewed in terms of error-correcting codes. The model $\hat\theta = X^T u\in\reals^p$ is in the row-space of the design matrix, which gives a redundant representation of the ``message'' $u\in\reals^n$ for $n < p$. The model is received as a noisy version $\eta$ with each entry corrupted with probability $\epsilon$. The received vector is decoded by solving a linear program.}
  \label{fig:code}
\end{figure}

\point{Estimators based on gradient descent}
The comments made above carry over to estimators of linear models based on gradient descent and stochastic gradient descent for arbitrary loss functions. Consider objective functions of the form
\begin{equation}
  \ell(\theta) = \frac{1}{n} \sum_{i=1}^n \L(y_i, x_i^T \theta)\label{eq:gen-obj}
\end{equation}
where $\L(y,f)$ is a general loss function; this includes
a broad range of estimators for problems such as linear least squares and logistic regression, robust regression, support vector machines, and others. The gradient descent update rule is
\begin{align}
  \theta^{(t+1)} &= \theta^{(t)} - \gamma_t \frac{1}{n} \sum_{i=1}^n \nabla_\theta \L(y_i, x_i^T \theta^{(t-1)}) \\
  &= \theta^{(t)} - \gamma_t \frac{1}{n} \sum_{i=1}^n \frac{\partial}{\partial f} \L(y_i, x_i^T \theta^{(t-1)}) x_i \\
  &= \theta^{(t)} - \sum_{i=1}^n w_i^{(t)} x_i,
\end{align}
where $\gamma_t$ is a step size parameter. If the model is initialized at $\theta^{(0)}= 0\in\reals^p$ then the estimate at time $t$ can thus be written as
\begin{equation}
  \theta^{(t)} = X^T u^{(t)}
\end{equation}
for some $u^{(t)}\in \reals^n$. After contamination, we have $\eta = X^T u^{(t)} + z$. Therefore, we can recover the model
by computing  $\tilde\theta = X^T \tilde u$ where $\tilde u$ is the solution to \eqref{eq:keylp}.

The same conclusion holds for estimators based on stochastic gradient descent. With $B_t$ denoting the set of samples used in the mini-batch of the $t$th iteration, we can write
\begin{equation}
  \theta^{(t)} = \sum_{i\in B_1\cup\cdots\cup B_t} w_i^{(t)} x_i.
\end{equation}
We repair the corrupted model $\eta$ by computing $\tilde \theta = X^T_{B_1\cup \cdots B_t} \tilde u$ with
\begin{equation}
  \tilde u = \argmin_u \| \eta - X^T_{B_1\cup \cdots B_t} u\|_1
\end{equation}
where the submatrix $X_{B_1\cup\cdots B_t}$ only includes rows of $X$ for indices that were visited during some stochastic
gradient descent step. Our theory then establishes that the model is recovered with high probability in case
\begin{equation}
  \frac{\sqrt{{|B_1 \cup \cdots B_t|}/{p}}}{1-\epsilon} < c.
\end{equation}
Typically the training takes place in ``epochs'' where all $n$ data points are visited in each epoch.

\point{Random features and neural networks} Our theory extends to random features models \citep{rahimi2008}, where the covariates are $\tilde X = \psi(X W)\in \reals^{n\times p}$ where $X\in\reals^{n\times d}$, the matrix $W\in\reals^{d\times p}$ is a random Gaussian matrix that is not trained, and $\psi$ is a threshold function such as the hyperbolic tangent function or rectified linear unit. In particular, when the model is trained using gradient descent, the parameters $\hat\theta$ lie in the row
space of the matrix $\tilde X$. We also show how the ideas can be extended to neural networks, where the weights $W$ are trained. This requires modifications to the training and recovery algorithms that we detail below.

In the following section we present the key technical lemma that explains how these results are possible; this is a result in robust regression that may be of independent interest. In Sections~\ref{sec:linear} and \ref{sec:neural} we state the theoretical results for over-complete linear models and neural networks trained with gradient descent. The proofs of the neural network results with hyperbolic tangent activation are given in Section \ref{sec:pf-nn-repair}.

% !TEX root = ./repair.tex

\section{Key lemma: Robust regression with uncorrupted design}
\label{sec:regression}

Consider a regression model $\eta=Au^*+z\in\mathbb{R}^m$, where $A^T=(a_1,a_2,...,a_m)^T\in\mathbb{R}^{m\times k}$ is a design matrix and $u^*\in\mathbb{R}^k$ is a vector of regression coefficients to be recovered. We consider a random design setting, and the distribution of $A$ will be specified later. For the noise vector $z\in\mathbb{R}^m$, we assume it is independent of the design matrix $A$, and
\begin{equation}
z_i\sim (1-\epsilon)\delta_0 + \epsilon Q_i, \label{eq:noise-add-con}
\end{equation}
independently for all $i\in[m]$. In other words, a fraction $\epsilon$ of the components $\eta_i$ are contaminated by
additive noise $z_i$ that is drawn from an arbitrary and unknown distribution. To robustly recover $u^*$, we propose the estimator
$$\wh{u}=\argmin_{u\in\mathbb{R}^k}\|\eta-Au\|_1.$$
The estimator $\hat u$ can be computed using standard linear programming.
In order that $\wh{u}$ successfully recovers the true regression coefficients $u^*$, we need to impose the following conditions on the design matrix $A$.

\begin{con0}
There exists some $\sigma^2$, such that for any fixed (not random) $c_1,...,c_m$ satisfying $\max_i|c_i|\leq 1$,
\begin{equation}
  \left\|\frac{1}{m}\sum_{i=1}^mc_ia_i\right\|^2\leq \frac{\sigma^2k}{m},
\end{equation}
with high probability.
\end{con0}

\begin{con2}
There exist $\underline{\lambda}$ and $\overline{\lambda}$, such that
\begin{eqnarray}
\label{eq:l1-upper-A} \inf_{\|\Delta\|=1}\frac{1}{m}\sum_{i=1}^m|a_i^T\Delta| &\geq& \underline{\lambda}, \\
\label{eq:l2-upper-A} \sup_{\|\Delta\|=1}\frac{1}{m}\sum_{i=1}^m|a_i^T\Delta|^2 &\leq& \overline{\lambda}^2,
\end{eqnarray}
with high probability.
\end{con2}

\begin{thm}\label{thm:main-improved}
Assume the design matrix $A$ satisfies \conditionA{} and \conditionB. Then if
\begin{equation}
\frac{\overline{\lambda}\sqrt{\frac{k}{m}\log\left(\frac{em}{k}\right)}+\epsilon\sigma\sqrt{\frac{k}{m}}}{\underline{\lambda}(1-\epsilon)}
\end{equation}
is sufficiently small, we have $\wh{u}=u^*$ with high probability.
\end{thm}

\begin{proof}
Define $L_m(u)=\frac{1}{m}\sum_{i=1}^m(|a_i^T(u^*-u)+z_i|-|z_i|)$, and $L(u)=\mathbb{E}(L_m(u) \given A)$.
%If $\|\wh{u}-u^*\|\geq t$,
Letting $t\geq 0$ be arbitrary, we must have $\inf_{\|u-u^*\|\geq t}L_m(u)\leq L_m(u^*)=0$. By the convexity of $L_m(u)$, this leads to $\inf_{\|u-u^*\|= t}L_m(u)\leq 0$, and thus
\begin{align*}
\inf_{\|u-u^*\|=t}L(u) &\leq \inf_{\|u-u^*\|=t}L_m(u) + \sup_{\|u-u^*\|=t} \left(L(u) - L_m(u)\right) \\
& \leq \sup_{\|u-u^*\|=t}|L_m(u)-L(u)|.
\end{align*}
Introducing the notation $f_i(x)=\mathbb{E}_{z_i\sim Q_i}(|x+z_i|-|z_i|)$ and $Q_i(x)=Q_i(z_i\leq x)$,
it is easy to see that $f_i(0)=0$ and $f'_i(x)=1-2Q_i(-x)$. Observe that we can write
\begin{equation}
L(u)=(1-\epsilon)\frac{1}{m}\sum_{i=1}^m|a_i^T(u-u^*)| + \epsilon\frac{1}{m}\sum_{i=1}^mf_i(a_i^T(u^*-u)). \label{eq:L-decomp}
\end{equation}
For any $u$ such that $\|u-u^*\|=t$, the first term of (\ref{eq:L-decomp}) can be lower bounded by
$$(1-\epsilon)\frac{1}{m}\sum_{i=1}^m|a_i^T(u-u^*)| \geq \underline{\lambda}(1-\epsilon)t,$$
by \conditionB. To analyze the second term of (\ref{eq:L-decomp}), we note that $f_i$ is a convex function, and therefore
for any $u$ such that $\|u-u^*\|=t$,
\begin{eqnarray*}
\epsilon\frac{1}{m}\sum_{i=1}^mf_i(a_i^T(u^*-u)) &\geq& \epsilon\frac{1}{m}\sum_{i=1}^mf_i(0) + \epsilon\frac{1}{m}\sum_{i=1}^mf_i'(0)a_i^T(u^*-u) \\
&=& \epsilon\frac{1}{m}\sum_{i=1}^m\left(1-2Q_i(0)\right)a_i^T(u^*-u) \\
&\geq& -\epsilon t\left\|\frac{1}{m}\sum_{i=1}^m\left(1-2Q_i(0)\right)a_i\right\| \\
&\geq& -\epsilon t\sigma\sqrt{\frac{k}{m}},
\end{eqnarray*}
where the first inequality uses Cauchy-Schwarz, and the second inequality uses \conditionA.
By \conditionB{} and an empirical process result proved as Lemma \ref{lem:EP} in Appendix \ref{sec:pf-robust-reg}, we have
\begin{equation}
\sup_{\|u-u^*\|= t}|L_m(u)-L(u)| \lesssim t\overline{\lambda}\sqrt{\frac{k}{m}\log\left(\frac{em}{k}\right)}, \label{eq:upper-EP}
\end{equation}
with high probability.
Therefore, we have shown that $\|\wh{u}-u^*\|\geq t$ implies
$$\underline{\lambda}(1-\epsilon)t - \epsilon t\sigma\sqrt{\frac{k}{m}} \lesssim t\overline{\lambda}\sqrt{\frac{k}{m}\log\left(\frac{em}{k}\right)},$$
which is impossible when $\frac{\overline{\lambda}\sqrt{\frac{k}{m}\log\left(\frac{em}{k}\right)}+\epsilon\sigma\sqrt{\frac{k}{m}}}{\underline{\lambda}(1-\epsilon)}$ is sufficiently small, and thus $\|\wh{u}-u^*\|< t$ with high probability. Since $t$ is arbitrary, we must have $\wh{u}=u^*$.
\end{proof}

The theorem gives a sufficient condition for the exact recovery of the regression coefficients. When both $(\sigma+\overline{\lambda})/\underline{\lambda}$ and $1-\epsilon$ are constants, the condition becomes that $k/m$ is sufficiently small.

A notable feature of this theorem is that it allows for $\epsilon \rightarrow 1$; that is, an arbitrarily large fraction of the components of the response $Au^*$ can be corrupted. This is in contrast to robust regression where both the response and design are contaminated. To be specific, consider independent observations $(a_i,\eta_i)\sim (1-\epsilon)P_{u^*}+\epsilon Q_i$, where the probability distribution $P_{u^*}$ encodes the linear model $\eta_i=a_i^Tu_i$, and for each $i\in[m]$, there is probability $\epsilon$ that the pair $(a_i,\eta_i)$ is drawn from some arbitrary distribution $Q_i$. In this setting, consistent or exact recovery of the regression coefficient is only possible when $\epsilon<c$ for some small constant $c>0$ \citep{gao2020}. The reason Theorem \ref{thm:main-improved} allows $\epsilon\rightarrow 1$ is that there is no contamination for the design matrix $A$.

%We note here previous work by that has also observed the possibility $\epsilon\to 1$ in a regression setting with corrupted response using $\ell_1$ minimization, although this literature has primarily focused on the case of sparse $u^*$ \citep{wright,nguyen1,nguyen2}. The work of \cite{bhatia} establishes consistency of an iterative hard thresholding method.

Another distinguishing feature of Theorem \ref{thm:main-improved} is that there is no assumption imposed on the contamination distribution $Q_i$, even though the median regression procedure apparently requires the noise to be symmetric around zero. To understand this phenomenon, consider a population objective function
$$\ell(u)=\mathbb{E}|\eta - a^Tu|,$$
where $\eta=a^Tu^*+z$, and the expectation is over both $a$ and $z$. In order for the minimizer of $\ell(u)$ to recover $u^*$ in the population, a criterion usually called Fisher consistency, it is required that $\nabla\ell(u^*)=0$. Under the assumption that $a$ and $z$ are independent, this gives
\begin{equation}
\nabla\ell(u^*)=\mathbb{E}[\sgn(z)a]=\mathbb{E}\sgn(z)\mathbb{E}a=0. \label{eq:pop-insight}
\end{equation}
This means we should be able to achieve consistency without any assumption on the noise variable $z$ as long as we assume $\mathbb{E}a_i=0$.

But \conditionA{} can be viewed as a general assumption that covers $\mathbb{E}a_i=0$ as a special case. As a concrete example, let us suppose the design matrix $A$ has $m$ uncorrelated rows and its entries all have mean zero and variance at most one. Then,
$$\mathbb{E}\left\|\frac{1}{m}\sum_{i=1}^mc_ia_i\right\|^2=\sum_{j=1}^k\mathbb{E}\left(\frac{1}{m}\sum_{i=1}^mc_ia_{ij}\right)^2=\sum_{j=1}^k\frac{1}{m^2}\sum_{i=1}^mc_i^2\mathbb{E}a_{ij}^2\leq \frac{k}{m},$$
and thus \conditionA{} holds with some constant $\sigma^2$, by an additional argument using Markov's inequality.

More generally, \conditionA{} also allows a design matrix with entries whose means are not necessarily zero. This will in general lead to a term $\sigma^2$ that may not be of constant order. However, since the condition of Theorem \ref{thm:main-improved} involves an additional $\epsilon$ factor in front of $\sigma$, the robust estimator $\wh{u}$ can still recover $u^*$ as long as the contamination proportion is vanishing at an appropriate rate. As an important application, the result of Theorem \ref{thm:main-improved} also applies to design matrices with an intercept.

We also introduce an alternative of \conditionA. By (\ref{eq:pop-insight}), we observe that Fisher consistency also follows if $\mathbb{E}\sgn(z_i)=0$. However, this does not mean that we have to assume the distribution of $z_i$ is symmetric. It turns out we only need the distribution of $a_i$ to be symmetric by applying a symmetrization argument.
Note that with the help of independent Rademacher random variables $\delta_i\sim\text{Uniform}\{\pm 1\}$, we can write the data generating process as $\delta_i\eta_i=\delta_ia_i^Tu^*+\delta_iz_i$. With this new representation, we can also view $\delta_i\eta_i$, $\delta_ia_i$ and $\delta_iz_i$ as the response, covariate, and noise. Now the noise $\delta_iz_i$ is symmetric around zero, and it can be shown that $\delta_ia_i$ and $\delta_iz_i$ are still independent as long as the distribution of $a_i$ is symmetric. Since for any $u\in\mathbb{R}^k$,
$$\sum_{i=1}^m|\delta_i\eta_i- \delta_ia_i^Tu|=\sum_{i=1}^m|\eta_i-a_i^Tu|,$$
we obtain an equivalent median regression after symmetrization.
This alternative condition is stated as follows.

\begin{con1}
Given i.i.d. Rademacher random variables $\delta_1,...,\delta_m$, the distribution of
$$\wt{A}^T=(\delta_1a_1,\delta_2a_2,...,\delta_ma_m)^T$$
is identical to that of $A^T$.
\end{con1}

\begin{thm}\label{thm:robust-reg}
Assume the design matrix $A$ satisfies \conditionAp{} and \conditionB. Then if $$\frac{\overline{\lambda}\sqrt{\frac{k}{m}\log\left(\frac{em}{k}\right)}}{\underline{\lambda}(1-\epsilon)}$$
is sufficiently small, we have $\wh{u}=u^*$ with high probability.
\end{thm}

To close this section, we note that the problem of robust regression with uncorrupted design is also recognized as outlier-robust regression in the literature.
This problem has been studied previously by \cite{tsakonas2014convergence,wright,nguyen1,nguyen2,karmalkar2018compressed}. In particular, \cite{bhatia} proposed a hard-thresholding algorithm that consistently recovers the regression coefficients as long as $\epsilon$ is below some small constant. The recent work \cite{suggala2019adaptive} has established consistent recovery while allowing $\epsilon\rightarrow 1$. Compared with their algorithm, our method based on $\ell_1$ minimization is much simpler. Moreover, we allow $\epsilon=1-\Theta\left(\sqrt{\frac{k}{m}\log\left(\frac{em}{k}\right)}\right)$, compared with the requirement $\epsilon\leq 1-\Theta\left(\frac{1}{\log\log m}\right)$ in \cite{suggala2019adaptive}.

% !TEX root = ./repair.tex

\section{Repair of linear and random feature models}
\label{sec:linear}

Consider a linear model with $X\in\mathbb{R}^{n\times p}$ being the design matrix and $y\in\mathbb{R}^n$ being a vector of response values. We assume that each entry of the design matrix is i.i.d. $N(0,1)$ and do not impose any assumption on the response $y$. A machine learning algorithm learns a linear model $X\wh{\theta}$ with some $\wh{\theta}\in\mathbb{R}^{p}$. The vector $\wh{\theta}$ is either computed via the formula (\ref{eq:min-norm-solution}) or through a gradient-based algorithm with the objective (\ref{eq:gen-obj}) initialized from $0$. Either case implies $\wh{\theta}$ belongs to the row space of $X$. Suppose we observe a contaminated version of $\wh{\theta}$ through $
\eta=\wh{\theta}+z$, where $z$ is independent of $\wh{\theta}$ and $z_j\sim (1-\epsilon)\delta_0+\epsilon Q_j$ independently for all $j\in[p]$. We then propose to recover $\wh{\theta}$ via
$$\wt{u}=\argmin_{u\in\mathbb{R}^n}\|\eta-X^Tu\|_1,$$
and define the repaired model as $\wt{\theta}=X^T\wt{u}$. This turns out to be the same robust regression problem studied in Section \ref{sec:regression}, and thus we only need to check the design matrix $A=X^T$ satisfies \conditionA{} and \conditionB.
\begin{lemma}\label{lem:design-linear}
Assume $n/p$ is sufficiently small. Then, \conditionA{} and \conditionB{} hold for $A=X^T$, $m=p$ and $k=n$ with some constants $\sigma^2$, $\underline{\lambda}$ and $\overline{\lambda}$.
\end{lemma}
Combine Lemma \ref{lem:design-linear} and Theorem \ref{thm:main-improved}, and we obtain the following guarantee for model repair.
\begin{corollary}\label{cor:repair-linear}
Assume $\frac{\sqrt{\frac{n}{p}}\log\left(\frac{ep}{n}\right)}{1-\epsilon}$ is sufficiently small. We then have $\wt{\theta}=\wh{\theta}$ with high probability.
\end{corollary}

We note that compared with the robust regression setting, the roles of the sample size and dimension are switched in model repair. Corollary \ref{cor:repair-linear} requires that the linear model to be overparametrized in the sense of $p\gg n(1-\epsilon)^2$ (with logarithmic factors ignored) in order that repair is successful. 

Besides an overparametrized model, we also require that the estimator $\wh{\theta}$ lies in the row space of the design matrix $X$, so that the redundancy of a overparametrized model is preserved in the estimator.
\begin{remark}
To understand the requirement on the estimator $\wh{\theta}$, let us consider a simple toy example. We assume that $X$ has $p$ identical columns, which is clearly an overparametrized model. Consider two estimators:
\begin{eqnarray*}
\wh{\theta}_{\sf min-norm} &\in& \argmin\left\{\|\theta\|: y=X\theta\right\}, \\
\wh{\theta}_{\sf sparse} &\in& \argmin\left\{\|\theta\|_0: y=X\theta\right\}.
\end{eqnarray*}
It is clear that $\wh{\theta}_{\sf min-norm}$ has identical entries and $\wh{\theta}_{\sf sparse}$ has one nonzero entry. Since the contamination will change an $\epsilon$-proportion of the entries, $\wh{\theta}_{\sf sparse}$ cannot be repaired if its only nonzero entry is changed. On the other hand, $\wh{\theta}_{\sf min-norm}$ is resilient to the contamination, and its redundant structure leads to consistent model repair. It is known that gradient based algorithms lead to implicit $\ell_2$ norm regularizations \citep{neyshabur2014search}, which then explains the result of Corollary \ref{cor:repair-linear}.
\end{remark}

We also study a random feature model with design $\{\psi(W_j^Tx_i)\}_{i\in[n],j\in[p]}$, where $x_i\sim N(0,I_d)$ and $W_j\sim N(0,d^{-1}I_d)$ independently for all $i\in[n]$ and $j\in[p]$. We choose the nonlinear activation function to be $\psi(t)=\tanh(t)=\frac{e^t-e^{-t}}{e^t+e^{-t}}$, the hyperbolic tangent unit. The design matrix can thus be written as $\wt{X}=\psi(XW)\in\mathbb{R}^{n\times p}$ with $X\in\mathbb{R}^{n\times d}$ and $W\in\mathbb{R}^{d\times p}$. This is an important model, and its asymptotic risk behavior under overparametrization has recently been studied by \cite{mei2019generalization}. We show that the design matrix $\wt{X}^T=\psi(W^TX^T)$ satisfies \conditionA{} and \conditionB{} so that model repair is possible.

\begin{lemma}\label{lem:design-rf}
Assume $n/p^2$ and $n/d$ are sufficiently small. Then, \conditionA{} and \conditionB{} hold for $A=\wt{X}^T$, $m=p$ and $k=n$ with some constants  $\sigma^2$, $\underline{\lambda}$ and $\overline{\lambda}$.
\end{lemma}

Now consider a model $\wh{\theta}$ that lies in the row space of $\wt{X}$. We observe a contaminated version $\eta=\wh{\theta}+z$.
We can then compute the procedure $\wt{u}=\argmin_{u\in\mathbb{R}^n}\|\eta-\wt{X}^Tu\|_1$ and use $\wt{\theta}=\wt{X}^T\wt{u}$ for model repair.

\begin{corollary}\label{cor:repair-rf}
Assume $\frac{\sqrt{\frac{n}{p}}\log\left(\frac{ep}{n}\right)}{1-\epsilon}$, $n/p^2$ and $n/d$ are sufficiently small. We then have $\wt{\theta}=\wh{\theta}$ with high probability.
\end{corollary}

%Since $\psi(t)\geq 0$ for all $t\in\mathbb{R}$, the design matrix $\psi(W^TX^T)$ clearly does not have zero mean, which leads to a strong condition on the contamination proportion $\epsilon$. In comparison, other choices of activation functions can lead to weaker conditions. For example, random feature model with hyperbolic tangent activation can be repaired with $\epsilon\rightarrow 1$. The results are stated in Appendix \ref{app:results}.
The results in this section are stated for the hyperbolic tangent nonlinear activation. They can be extended to other activation functions. In practice, the most popular choice is the rectified linear unit (ReLU) $\psi(t)=\max(0,t)$. The results for ReLU will be given in the appendix.
% !TEX root = ./repair.tex

\def\L{{\mathcal L}}
\section{Repair of neural networks}
\label{sec:neural}

In this section we show how to use robust regression to repair neural networks.
We consider a neural network with one hidden layer,
$$f(x)=\frac{1}{\sqrt{p}}\sum_{j=1}^p\beta_j\psi(W_j^Tx),$$
where $\psi$ is either the rectified linear unit (ReLU) function $\psi(t)=\max(t,0)$, or the
hyperbolic tangent $\psi(t) = \tanh(t) = \frac{e^t - e^{-t}}{e^t + e^{-t}}$.
The factor $p^{-1/2}$ in the definition above is convenient for our theoretical analysis.
In this section we present the analysis for the hyperbolic tangent activation function, with the ReLU
deferred to the appendix.

With the squared error loss function
$$\L(\beta,W)=\frac{1}{2}\sum_{i=1}^n\left(y_i - \frac{1}{\sqrt{p}}\sum_{j=1}^p\beta_j\psi(W_j^Tx_i)\right)^2,$$
we consider training the model using a standard gradient descent algorithm (Algorithm \ref{alg:GD}).

\vskip10pt
\begin{algorithm}[H]
\DontPrintSemicolon
%\SetKwInOut{Input}{Input}\SetKwInOut{Output}{Output}
\vskip5pt
\nl Input: Data $(y, X)$ and the number of iterations $t_{\max}$. \\[3pt]
\nl Initialization: $W_j(0)\sim N(0,d^{-1}I_d)$ and $\beta_j(0)\sim N(0,1)$ independently for all $j\in[p]$. \\[3pt]
\nl Iterate: For $t$ in $1:t_{\max}$, compute
\begin{align*}
\beta_j(t) &= \beta_j(t-1) - \left.\gamma\frac{\partial \L(\beta,W)}{\partial \beta_j}\right|_{(\beta,W)=(\beta(t-1),W(t-1))} \quad j\in[p], \\
W_j(t) &= W_j(t-1) - \left.\frac{\gamma}{d}\frac{\partial \L(\beta,W)}{\partial W_j}\right|_{(\beta,W)=(\beta(t),W(t-1))} \quad j\in[p].
\end{align*}
\nl Output: Trained parameters $\beta(t_{\max})$ and $W(t_{\max})$.\\[5pt]
\caption{Gradient descent for neural nets}\label{alg:GD}
\end{algorithm}
\vskip10pt

Based on this standard gradient descent algorithm, we consider two estimators of the parameters $(\wh{\beta},\wh{W})$.
The first is simply to set $\wh{\beta}=\beta(t_{\max})$ and $\wh{W}=W(t_{\max})$; this is the usual estimator.

In the second approach, one fixes $\wh{W} = W(t_{\max})$ and then retrains $\beta$
using gradient descent for the objective $\|y-\psi(X\wh{W})\beta\|^2$ after initializing at zero.
In this way, $\wh{\beta}$ is an approximation to the minimal $\ell_2$ norm solution of $\|y-\psi(X\wh{W})\beta\|^2$.
This estimator can be viewed as a linear model that uses features extracted from the data by the neural network.

Now consider the contaminated model $\eta=\wh{\beta}+z$ and $\Theta_j=\wh{W}_j+Z_j$, where each entry of $z$ and $Z_j$ is zero with probability $1-\epsilon$ and follows an arbitrary distribution with the complementary probability $\epsilon$. We analyze the following repair algorithm.

\vskip10pt
\begin{algorithm}[H]
\DontPrintSemicolon
\vskip5pt
\nl Input: Contaminated model $(\eta,\Theta)$, design matrix $X$, and initializations $\beta(0)$, $W(0)$. \\[3pt]
\nl Repair of the hidden layer: For $j\in[p]$, compute
$$\wt{v}_j=\argmin_v\|\Theta_j-W_j(0)-X^Tv_j\|_1,$$
and set $\wt{W}_j=W_j(0)+X^T\wt{v}_j$. \\[8pt]
\nl Repair of the output layer: Compute
$$\wt{u}=\argmin_u\|\eta-\beta(0)-\psi(\wt{W}^TX^T)u\|_1,$$
and set $\wt{\beta}=\beta(0)+\psi(\wt{W}^TX^T)\wt{u}$. \\[3pt]
\nl Output: The repaired parameters $\wt{\beta}$ and $\wt{W}$. \\[5pt]
\caption{Model repair for neural networks}\label{alg:MR}
\end{algorithm}
\vskip10pt

\begin{remark}
Algorithm \ref{alg:MR} adopts a layerwise repair strategy. This algorithm extends naturally to multilayer networks, repairing the parameters in stages with a forward pass through the layers. We leave the multilayer extension of our analysis to future work.
\end{remark}

\begin{remark}
It is important to note that the repair of neural networks not only requires $X$, but it also
requires the initializations $\beta(0)$ and $W(0)$. From a practical
perspective, this can be easily achieved by setting a seed using a pseudorandom number generator to initialize the
parameters, and making the seed available to the repair algorithm. We also note that when $\hat\beta$ is
trained after fixing $\hat W$, one can replace $\beta(0)$ by 0 in Algorithm \ref{alg:MR}.
\end{remark}

Since the gradient $\frac{\partial \L(\beta,W)}{\partial W_j}$ lies in the row space of $X$, the vector $\wh{W}_j-W_j(0)$ also lies in the row space of $X$. Thus, the theoretical guarantee of the hidden layer repair directly follows
Corollary \ref{cor:repair-linear}.
The repair of the output layer is more complicated, because the gradient $\frac{\partial \L(\beta,W)}{\partial \beta_j}|_{W=W(t-1)}$ lies in the row space of $\psi(XW(t-1))$, which changes over time. Thus, we cannot directly apply the result of Corollary \ref{cor:repair-rf} for the random feature model. However, when the neural network is overparametrized, it can be shown that the gradient descent algorithm (Algorithm \ref{alg:GD}) leads to $W(t)$ that is close to the initialization $W(0)$ for all $t\geq 0$. We establish this result in the following theorem by assuming that $x_i$ is i.i.d. $N(0,I_d)$ and $|y_i|\leq 1$ for all $i\in[n]$. Define $u(t)\in\mathbb{R}^n$ with its $i$th entry given by
$u_i(t)=\frac{1}{\sqrt{p}}\sum_{j=1}^p\beta_j(t)\psi(W_j(t)^Tx_i)$, the function value of $x_i$ at time $t$.
\vskip10pt
\begin{thm}\label{thm:nn-grad}
Assume $\frac{n}{d}$, $\frac{n^3(\log p)^2}{p}$, and $\gamma\left(1+\frac{n^4(\log p)^2}{p}\right)$ are all sufficiently small. Then, we have
\begin{equation}
\|y-u(t)\|^2 \leq \left(1-\frac{\gamma}{8}\right)^t\|y-u(0)\|^2, \label{eq:iter-function}
\end{equation}
and
\begin{eqnarray}
\label{eq:iter-parameter} \max_{1\leq j\leq p}\|W_j(t)-W_j(0)\| &\leq& R_1, \\
\label{eq:iter-parameter-beta} \max_{1\leq j\leq p}|\beta_j(t)-\beta_j(0)| &\leq& R_2,
\end{eqnarray}
for all $t\geq 1$ with high probability, where $R_1=\frac{100n\log p}{\sqrt{pd}}$ and $R_2=32\sqrt{\frac{n^2\log p}{p}}$.
\end{thm}

Theorem \ref{thm:nn-grad} assumes that the width of the neural network is large compared with the sample size in the sense that $\frac{p}{(\log p)^4}\gg n^3$. For fixed $n$, the limit of the neural network as $p\rightarrow\infty$ is known as the neural tangent kernel (NTK) regime, and the behavior of gradient descent under this limit has been studied by \cite{jacot2018neural}. The result of Theorem \ref{thm:nn-grad} follows the explicit calculation in \cite{du2018gradient}, and we are able to sharpen some of the asymptotic conditions in \cite{du2018gradient}.

The theorem has two conclusions. The first conclusion shows the gradient descent algorithm has global convergence in the sense of (\ref{eq:iter-function}) even though the loss $\L(\beta,W)$ is nonconvex. The second conclusion shows that the trajectory of the algorithm $(W(t),\beta(t))$ is bounded within some radius of the initialization. This allows us to characterize the repaired model $\wt{\beta}$ for the output layer.

Let us first consider the case $\wh{\beta}=\beta(t_{\max})$ and $\wh{W}=W(t_{\max})$. Since the vector $\beta(t)-\beta(t-1)$ lies in the row space of $\psi(XW(t-1))$ for every $t$, one can show that $\wh{\beta}-\beta(0)$ approximately lies in the row space of $\psi(XW(0))$ by Theorem \ref{thm:nn-grad}. Therefore, by extending the result of Corollary \ref{cor:repair-rf} that includes the bias induced by the row space approximation, we are able to obtain the following guarantee for the model repair.

\vskip10pt
\begin{thm}\label{thm:repair-nn-1}
Under the conditions of Theorem \ref{thm:nn-grad}, additionally assume that $\frac{\log p}{d}$, $\frac{\sqrt{\frac{n}{d}\log\left(\frac{ed}{n}\right)}}{1-\epsilon}$ and $\frac{n^2\log p}{p(1-\epsilon)}$ are sufficiently small. We then have $\wt{W}=\wh{W}$ and $\frac{1}{p}\|\wt{\beta}-\wh{\beta}\|^2 \lesssim \frac{n^2\log p}{p(1-\epsilon)}$ with high probability.
\end{thm}

We also consider the case where $\wh{W}=W(t_{\max})$ and $\wh{\beta}$ is obtained by retraining $\beta$ using the features $\psi(X\wh{W})$. Recall that in this case we shall replace $\beta(0)$ by $0$ in Algorithm \ref{alg:MR}. Note that the vector $\wh{\beta}$ exactly lies in the row space of $\psi(X\wh{W})$. This allows us to extend the result of Lemma \ref{lem:design-rf} to the matrix $\psi(\wh{W}^TX^T)$ with the help of Theorem \ref{thm:nn-grad}. Then, we can directly apply Theorem \ref{thm:robust-reg}. We are able to obtain exact recovery of both $\wh{\beta}$ and $\wh{W}$ in this case.

\vskip10pt
\begin{thm}\label{thm:repair-nn-2}
Under the conditions of Theorem \ref{thm:nn-grad}, additionally assume that $\frac{\log p}{d}$, $\frac{\sqrt{\frac{n}{d}\log\left(\frac{ed}{n}\right)}}{1-\epsilon}$, $\frac{n\log p}{p(1-\epsilon)}$ and $\frac{n}{p}\left(\frac{\log p}{1-\epsilon}\right)^{4/3}$ are sufficiently small. We then have $\wt{W}=\wh{W}$ and $\wt{\beta}=\wh{\beta}$ with high probability.
\end{thm}
\vskip10pt

\begin{remark}
As long as the rate that $\epsilon$ tends to $1$ is not so fast, the conditions of Theorem \ref{thm:repair-nn-1} and Theorem \ref{thm:repair-nn-2} can be simplified to $p\gg n^3$ and $d\gg n$ by ignoring the logarithmic factors. The condition $p\gg n^3$ ensures the good property of gradient descent in the NTK regime, but our experimental results show that it can potentially be weakened by an improved analysis.
\end{remark}

\iffalse
\begin{remark}
When the nonlinear activation is replaced by ReLU $\psi(t)=\max(t,0)$, \nb{add text}
\end{remark}
\fi

\section{Proofs of Theorem \ref{thm:repair-nn-1} and Theorem \ref{thm:repair-nn-2}}\label{sec:pf-nn-repair}

We give proofs of Theorem \ref{thm:repair-nn-1} and Theorem \ref{thm:repair-nn-2} in this section.
To prove Theorem \ref{thm:repair-nn-1}, we need to extend Theorem \ref{thm:main-improved}. Consider $\eta=b+Au^*+z\in\mathbb{R}^m$, where the noise vector $z$ satisfies (\ref{eq:noise-add-con}), and $b\in\mathbb{R}^m$ is an arbitrary bias vector. Then, the estimator $\wh{u}=\argmin_{u\in\mathbb{R}^k}\|\eta-Au\|_1$ satisfies the following theoretical guarantee.
\begin{thm}\label{thm:robust-reg-b}
Assume the design matrix $A$ satisfies \conditionA{} and \conditionB. Then, as long as $\frac{\overline{\lambda}\sqrt{\frac{k}{m}\log\left(\frac{em}{k}\right)}+\epsilon\sigma\sqrt{\frac{k}{m}}}{\underline{\lambda}(1-\epsilon)}$ is sufficiently small and $\frac{8\frac{1}{m}\sum_{i=1}^m|b_i|}{\underline{\lambda}(1-\epsilon)}<1$, we have
$$\|\wh{u}-u^*\|\leq \frac{4\frac{1}{m}\sum_{i=1}^m|b_i|}{\underline{\lambda}(1-\epsilon)},$$
with high probability.
\end{thm}
It is easy to see that Theorem \ref{thm:main-improved} is a special case when $b=0$.
Now we are ready to prove Theorem \ref{thm:repair-nn-1}.
\begin{proof}[Proof of Theorem \ref{thm:repair-nn-1}]
We first analyze $\wh{v}_1,...,\wh{v}_p$. The idea is to apply the result of Theorem \ref{thm:main-improved} to each of the $p$ robust regression problems. Thus, it suffices to check if the conditions of Theorem \ref{thm:main-improved} hold for the $p$ regression problems simultaneously. Since the $p$ regression problems share the same Gaussian design matrix, Lemma \ref{lem:design-linear} implies that Conditions $A$ and $B$ hold for all the $p$ regression problems. Next, by scrutinizing the proof of Theorem \ref{thm:main-improved}, the randomness of the conclusion is from the noise vector $Z_j$ through the empirical process bound given by Lemma \ref{lem:EP}. With an additional union bound argument applied to (\ref{eq:double-ub}) in its proof, Lemma \ref{lem:EP} can be extended to $Z_j$ simultaneously for all $j\in[p]$ with an additional assumption that $\frac{\log p}{d}$ is sufficiently small. Then, by the same argument that leads to Corollary \ref{cor:repair-linear}, we have $\wt{W}_j=\wh{W}_j$ for all $j\in[p]$ with high probability.

To analyze $\wh{u}$, we apply Theorem \ref{thm:robust-reg-b}. Note that
\begin{eqnarray*}
\eta_j - \beta_j(0) &=& \beta_j(t_{\max}) - \beta_j(0) +z_j \\
&=& \sum_{t=0}^{t_{\max}-1}\left(\beta_j(t+1)-\beta_j(t)\right) + z_j \\
&=& \frac{\gamma}{\sqrt{p}}\sum_{t=0}^{t_{\max}-1}\sum_{i=1}^n(y_i-u_i(t))\psi(W_j(t)^Tx_i) + z_j \\
&=& \frac{\gamma}{\sqrt{p}}\sum_{t=0}^{t_{\max}-1}\sum_{i=1}^n(y_i-u_i(t))(\psi(W_j(t)^Tx_i)-\psi(W_j(0)^Tx_i)) \\
&& + \frac{\gamma}{\sqrt{p}}\sum_{t=0}^{t_{\max}-1}\sum_{i=1}^n(y_i-u_i(t))\psi(W_j(0)^Tx_i) + z_j.
\end{eqnarray*}
Thus, in the framework of Theorem \ref{thm:robust-reg-b}, we can view $\eta-\beta(0)$ as the response, $\psi(X^TW(0)^T)$ as the design, $z$ as the noise, and $b_j=\frac{\gamma}{\sqrt{p}}\sum_{t=0}^{t_{\max}-1}\sum_{i=1}^n(y_i-u_i(t))(\psi(W_j(t)^Tx_i)-\psi(W_j(0)^Tx_i))$ as the $j$th entry of the bias vector. By Lemma \ref{lem:design-rf}, we know that the design matrix $\psi(X^TW(0)^T)$ satisfies \conditionA{} and \conditionB. So it suffices to bound $\frac{1}{p}\sum_{j=1}^p|b_j|$. With the help of Theorem \ref{thm:nn-grad}, we have
\begin{eqnarray*}
\frac{1}{p}\sum_{j=1}^p|b_j| &\leq& \frac{\gamma}{p^{3/2}}\sum_{j=1}^p\sum_{t=0}^{t_{\max}-1}\sum_{i=1}^n|y_i-u_i(t)||(W_j(t)-W_j(0))^Tx_i| \\
&\leq& \frac{R_1\gamma}{p^{1/2}}\sum_{t=0}^{t_{\max}-1}\sum_{i=1}^n|y_i-u_i(t)|\|x_i\| \\
&\leq& \frac{R_1\gamma}{p^{1/2}}\sum_{t=0}^{t_{\max}-1}\|y-u(t)\|\sqrt{\sum_{i=1}^n\|x_i\|^2} \\
&\lesssim& \frac{R_1}{p^{1/2}}\|y-u(0)\|\sqrt{\sum_{i=1}^n\|x_i\|^2} \\
&\lesssim& \frac{n^2\log p}{p},
\end{eqnarray*}
where the last inequality is by $\sum_{i=1}^n\|x_i\|^2\lesssim nd$ due to a standard chi-squared bound (Lemma \ref{lem:chi-squared}), and $\|u(0)\|^2\lesssim n$ is due to Markov's inequality and $\mathbb{E}|u_i(0)|^2 = \mathbb{E}\Var(u_i(0)|X) \leq 1$. By Theorem \ref{thm:robust-reg-b} and Lemma \ref{lem:lim-G}, we have $\frac{1}{p}\|\wt{\beta}-\wh{\beta}\|^2 \lesssim \frac{n^3\log p}{p}$, which is the desired conclusion.
\end{proof}

\begin{proof}[Proof of Theorem \ref{thm:repair-nn-2}]
The analysis of $\wh{v}_1,...,\wh{v}_p$ is the same as that in the proof of Theorem \ref{thm:repair-nn-1}, and we have $\wt{W}_j=\wh{W}_j$ for all $j\in[p]$ with high probability.

To analyze $\wh{u}$, we apply Theorem \ref{thm:robust-reg}. It suffices to check \conditionAp{} and \conditionB{} for the design matrix $\psi(X^T\wt{W}^T)=\psi(X^T\wh{W}^T)$. To check \conditionAp, we consider i.i.d. Rademacher random variables $\delta_1,...,\delta_m$. Then, we define a different gradient update with initialization $\check{W}_j(0)=\delta_jW_j(0)$ and $\check{\beta}_j(0)=\delta_j\beta_j(0)$, and
\begin{eqnarray*}
\check{\beta}_j(t) &=& \check{\beta}_j(t-1) - \gamma\frac{\partial L(\beta,W)}{\partial \beta_j}|_{(\beta,W)=(\check{\beta}(t-1),\check{W}(t-1))}, \\
\check{W}_j(t) &=& \check{W}_j(t-1) - \frac{\gamma}{d}\frac{\partial L(\beta,W)}{\partial W_j}|_{(\beta,W)=(\check{\beta}(t),\check{W}(t-1))}.
\end{eqnarray*}
In other words, $(W(t),\beta(t))$ and $(\check{W}(t),\check{\beta}(t))$ only differ in terms of the initialization. Recall that $u_i(t)=\frac{1}{\sqrt{p}}\sum_{j=1}^p\beta_j(t)\psi(W_j(t)^Tx_i)$. We also define 
\begin{eqnarray*}
\check{u}_i(t) &=& \frac{1}{\sqrt{p}}\sum_{j=1}^p\check{\beta}_j(t)\psi(\check{W}_j(t)^Tx_i), \\
v_i(t) &=& \frac{1}{\sqrt{p}}\sum_{j=1}^p\beta_j(t)\psi(W_j(t-1)^Tx_i), \\
\check{v}_i(t) &=& \frac{1}{\sqrt{p}}\sum_{j=1}^p\check{\beta}_j(t)\psi(\check{W}_j(t-1)^Tx_i).
\end{eqnarray*}
It is easy to see that
$$\check{u}_i(t)=\frac{1}{\sqrt{p}}\sum_{j=1}^p\delta_j\beta_j(t)\psi(\delta_jW_j(t)^Tx_i)=\frac{1}{\sqrt{p}}\sum_{j=1}^p\beta_j(t)\psi(W_j(t)^Tx_i)=u_i(t).$$
Similarly, we also have
$$\check{v}_i(t)=\frac{1}{\sqrt{p}}\sum_{j=1}^p\delta_j\beta_j(t)\psi(\delta_jW_j(t-1)^Tx_i)=\frac{1}{\sqrt{p}}\sum_{j=1}^p\beta_j(t)\psi(W_j(t-1)^Tx_i)=v_i(t).$$
Suppose $\check{W}_j(k)=\delta_jW_j(k)$ and $\check{\beta}_j(k)=\delta_j\beta_j(k)$ are true. Since
\begin{eqnarray*}
\frac{\partial L(\beta,W)}{\partial \beta_j}|_{(\beta,W)=(\check{\beta}(k),\check{W}(k))} &=& \frac{1}{\sqrt{p}}\sum_{i=1}^n(\check{u}_i(k)-y_i)\psi(\check{W}_j(k)^Tx_i) \\
&=& \frac{1}{\sqrt{p}}\sum_{i=1}^n({u}_i(k)-y_i)\psi(\delta_j{W}_j(k)^Tx_i) \\
&=& \delta_j\frac{1}{\sqrt{p}}\sum_{i=1}^n({u}_i(k)-y_i)\psi({W}_j(k)^Tx_i) \\
&=& \delta_j\frac{\partial L(\beta,W)}{\partial \beta_j}|_{(\beta,W)=({\beta}(k),{W}(k))},
\end{eqnarray*}
we have $\check{\beta}_j(k+1)=\delta_j\beta_j(k+1)$.
Then,
\begin{eqnarray*}
\frac{\partial L(\beta,W)}{\partial W_j}|_{(\beta,W)=(\check{\beta}(k+1),\check{W}(k))} &=& \frac{1}{\sqrt{p}}\check{\beta}_j(k+1)\sum_{i=1}^n(\check{v}_i(k+1)-y_i)\psi'(\check{W}_j(k)^Tx_i)x_i \\
&=& \frac{1}{\sqrt{p}}\delta_j{\beta}_j(k+1)\sum_{i=1}^n({v}_i(k+1)-y_i)\psi'(\delta_j{W}_j(k)^Tx_i)x_i \\
&=& \frac{1}{\sqrt{p}}\delta_j{\beta}_j(k+1)\sum_{i=1}^n({v}_i(k+1)-y_i)\psi'({W}_j(k)^Tx_i)x_i \\
&=& \delta_j\frac{\partial L(\beta,W)}{\partial W_j}|_{(\beta,W)=({\beta}(k+1),{W}(k))},
\end{eqnarray*}
and thus we also have $\check{W}_j(k+1)=\delta_jW_j(k+1)$. A mathematical induction argument leads to $\check{W}_j(t)=\delta_jW_j(t)$ and $\check{\beta}_j(t)=\delta_j\beta_j(t)$ for all $t\geq 1$. Since $(\check{W}(0),\check{\beta}(0))$ and $(W(0),\beta(0))$ have the same distribution, we can conclude that $(\check{W}(t),\check{\beta}(t))$ and $(W(t),\beta(t))$ also have the same distribution. Therefore, \conditionAp{} holds for the design matrix $\psi(X^T\wh{W}^T)=\psi(X^TW(t_{\max})^T)$.

We also need to check \conditionB. By Theorem \ref{thm:nn-grad}, we have
\begin{eqnarray*}
&& \left|\frac{1}{p}\sum_{j=1}^p\left|\sum_{i=1}^n\psi(\wh{W}_j^Tx_i)\Delta_i\right| - \frac{1}{p}\sum_{j=1}^p\left|\sum_{i=1}^n\psi(W_j(0)^Tx_i)\Delta_i\right|\right| \\
&\leq& \frac{1}{p}\sum_{j=1}^p\sum_{i=1}^n|\wh{W}_j^Tx_i-W_j(0)^Tx_i||\Delta_i| \\
&\leq& R_1\sum_{i=1}^n\|x_i\||\Delta_i| \leq R_1\sqrt{\sum_{i=1}^n\|x_i\|^2} \lesssim \frac{n^{3/2}\log p}{\sqrt{p}},
\end{eqnarray*}
where $\sum_{i=1}^n\|x_i\|^2\lesssim nd$ is by Lemma \ref{lem:chi-squared}. By Lemma \ref{lem:design-rf}, we can deduce that
$$\inf_{\|\Delta\|=1}\frac{1}{p}\sum_{j=1}^p\left|\sum_{i=1}^n\psi(\wh{W}_j^Tx_i)\Delta_i\right|\gtrsim 1,$$
as long as $\frac{n^{3/2}\log p}{\sqrt{p}}$ is sufficiently small. 
According to Lemma \ref{lem:lim-G}, we also have
$$\sup_{\|\Delta\|=1}\frac{1}{p}\sum_{j=1}^p\left|\sum_{i=1}^n\psi(\wh{W}_j^Tx_i)\Delta_i\right|^2\lesssim 1+\frac{n^2\log p}{\sqrt{p}}.$$
See (\ref{eq:Gk-spec-arctan}) in the appendix for details of derivation.
Therefore, \conditionB{} holds with $\overline{\lambda}^2\asymp 1+\frac{n^2\log p}{\sqrt{p}}$ and $\underline{\lambda}\asymp 1$. Apply Theorem \ref{thm:robust-reg}, we have $\wt{\beta}=\wh{\beta}$ with high probability as desired.
\end{proof}

% !TEX root = ./repair.tex

\section{Simulation studies}
\label{sec:experiments}

In this section we discuss experimental results for over-paramaterized linear models, random feature models, and neural networks, illustrating and confirming the theoretical results presented in the previous sections.
In all of our experiments, the \texttt{quantreg} package in $R$ is used to carry out the $\ell_1$ optimization of \eqref{eq:keylp} as quantile regression for quantile level $\tau = \frac{1}{2}$ using the Frisch-Newton interior point algorithm to solve the linear program (method \texttt{fn} in this package).

\subsection{Over-parameterized linear models}

We begin by giving further details of the simulation briefly discussed
in Section~\ref{sec:overview}. In this experiment we simulate underdetermined linear models where $p > n$.
We generate $n$ data points $(x_i, y_i)$ where
$y_i = x_i^T \theta^* + w_i$ with $w_i$ an additive noise term. We then compute the minimum norm estimator
\begin{equation*}
  \hat\theta = X^T (X X^T)^{-1} y.
\end{equation*}
The estimated model is corrupted to
\begin{equation*}
  \eta = \hat\theta + z
\end{equation*}
where $z_j \sim (1-\epsilon) \delta_0 +\epsilon Q$. The corrupted estimator is then repaired by performing median regression:
\begin{align*}
  \tilde u &= \argmin \|\eta - X^T u\|_1, \\
  \tilde \theta &= X^T \tilde u.
\end{align*}

The design is sampled as $X_{ij} \sim N(0,1)$ and we take $\theta_j^* \sim N(0,1)$ and $Q = N(1,1)$. In the plots shown
in Figure~\ref{fig:exp1} the sample size is fixed at $n=100$ and the dimension $p$ is varied according to $p/n=200/j^2$
for a range of values of $j$. The plots show the empirical probability of exact repair $\tilde\theta = \hat\theta$ as a function of $\epsilon$. Each point on the curves is the average repair success over $500$ random trials.
The roughly equal spacing of the curves agrees with the theory, which indicates that $\sqrt{n/p}/(1-\epsilon)$ should be sufficiently small for successful repair. The right plot in Figure~\ref{fig:exp1b} shows the per-coefficient repair probability, and the left plot shows the probability that the entire model is repaired; in this plot the sample size is $n=50$. The per-coefficient repair probability is the empirical probability that $\tilde \theta_j = \hat\theta_j$, averaged over $j=1,\ldots, p$.

\begin{figure}[ht]
  \begin{center}
    \begin{tabular}{cc}
      \includegraphics[width=.47\textwidth]{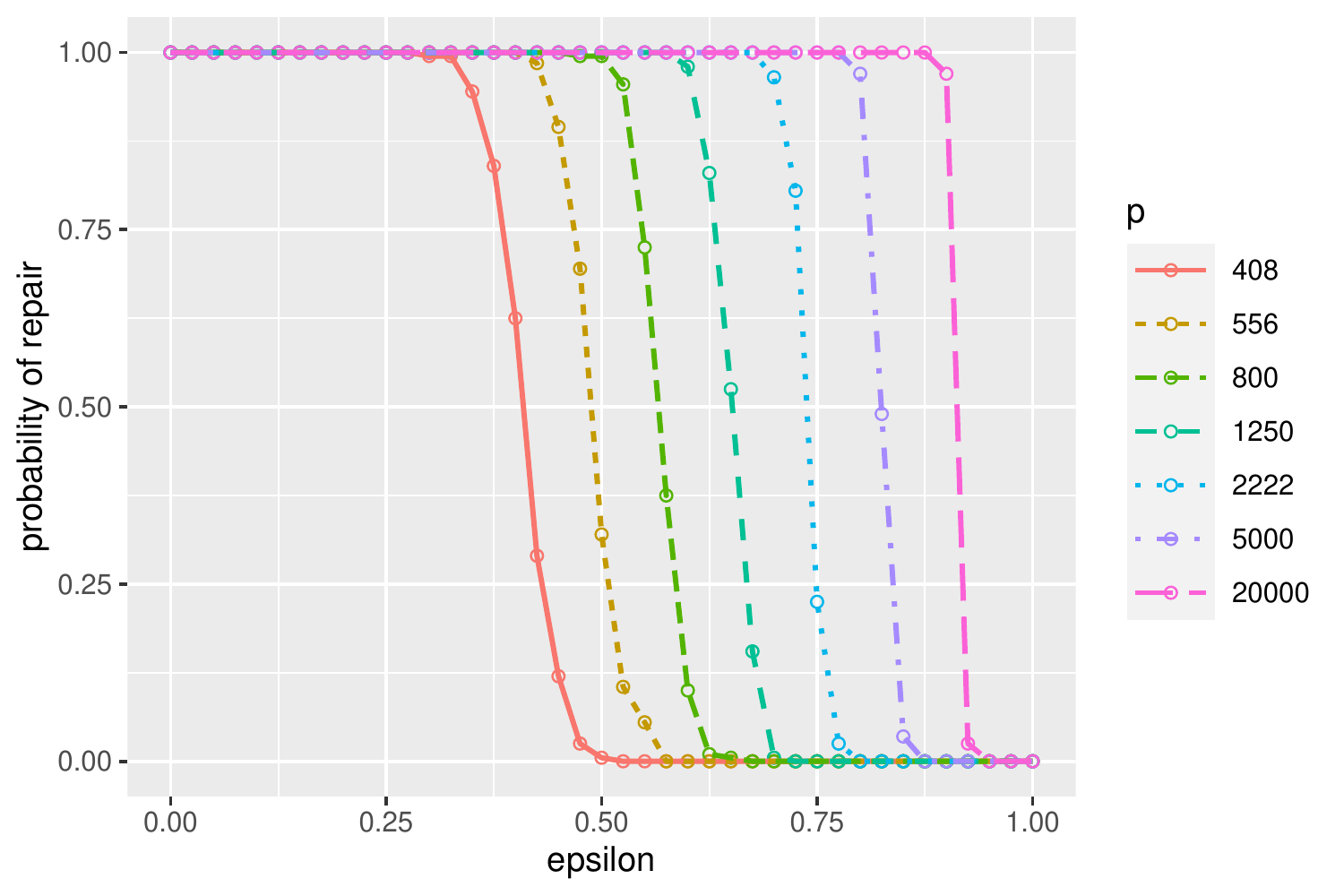} &
      \includegraphics[width=.47\textwidth]{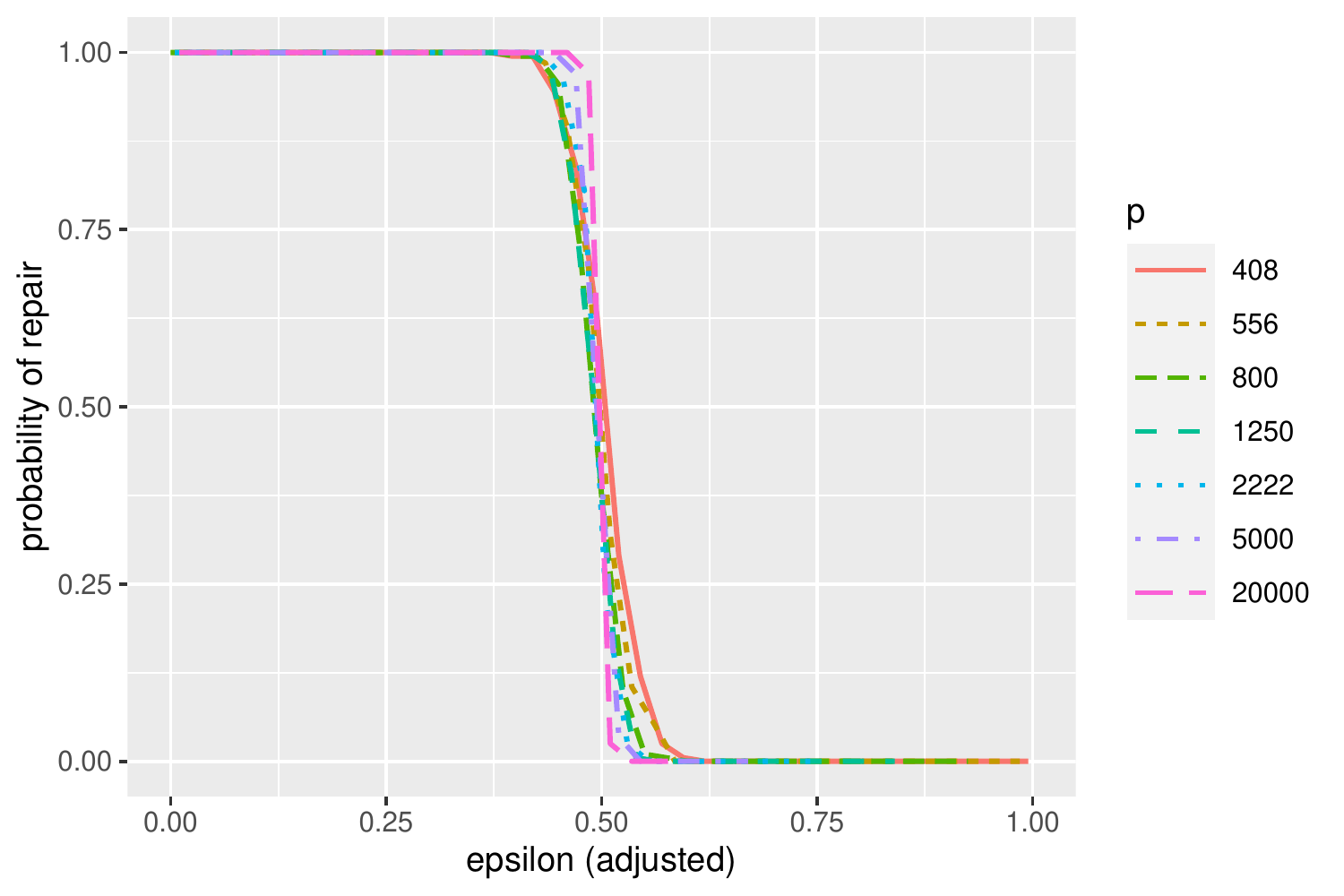}\\[-10pt]
    \end{tabular}
    \caption{Model repair for underdetermined linear models $y=X^T\theta + w$ with $p>n$. The left plot shows the empirical probability of successful model repair for $n=100$ with the model dimension $p$
    varying as $p/n = 200 /j^2$, for $j=1,\ldots, 7$. Each point is an average over 500 random trials. The covariates are sampled as $N(0,1)$ and the corruption distribution is $Q=N(1,1)$. The right plot shows the repair probablity as a function
    of the adjusted corruption probability $\tilde\epsilon_j = \epsilon + c'\cdot j - \frac{1}{2}$ for $c'=0.085$.}
    \label{fig:exp1}
    \vskip20pt
    \begin{tabular}{cc}
      \includegraphics[width=.47\textwidth]{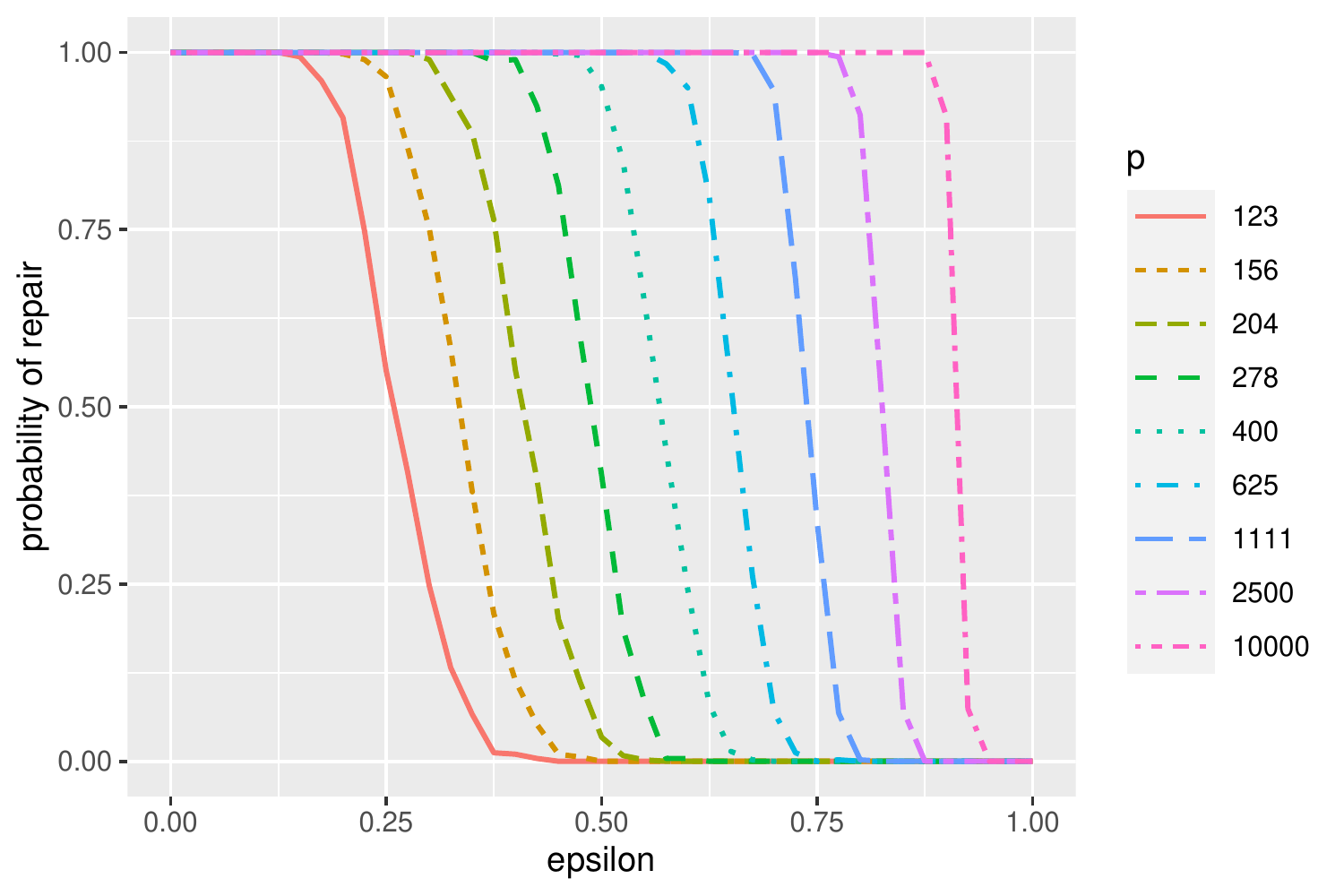} &
      \includegraphics[width=.47\textwidth]{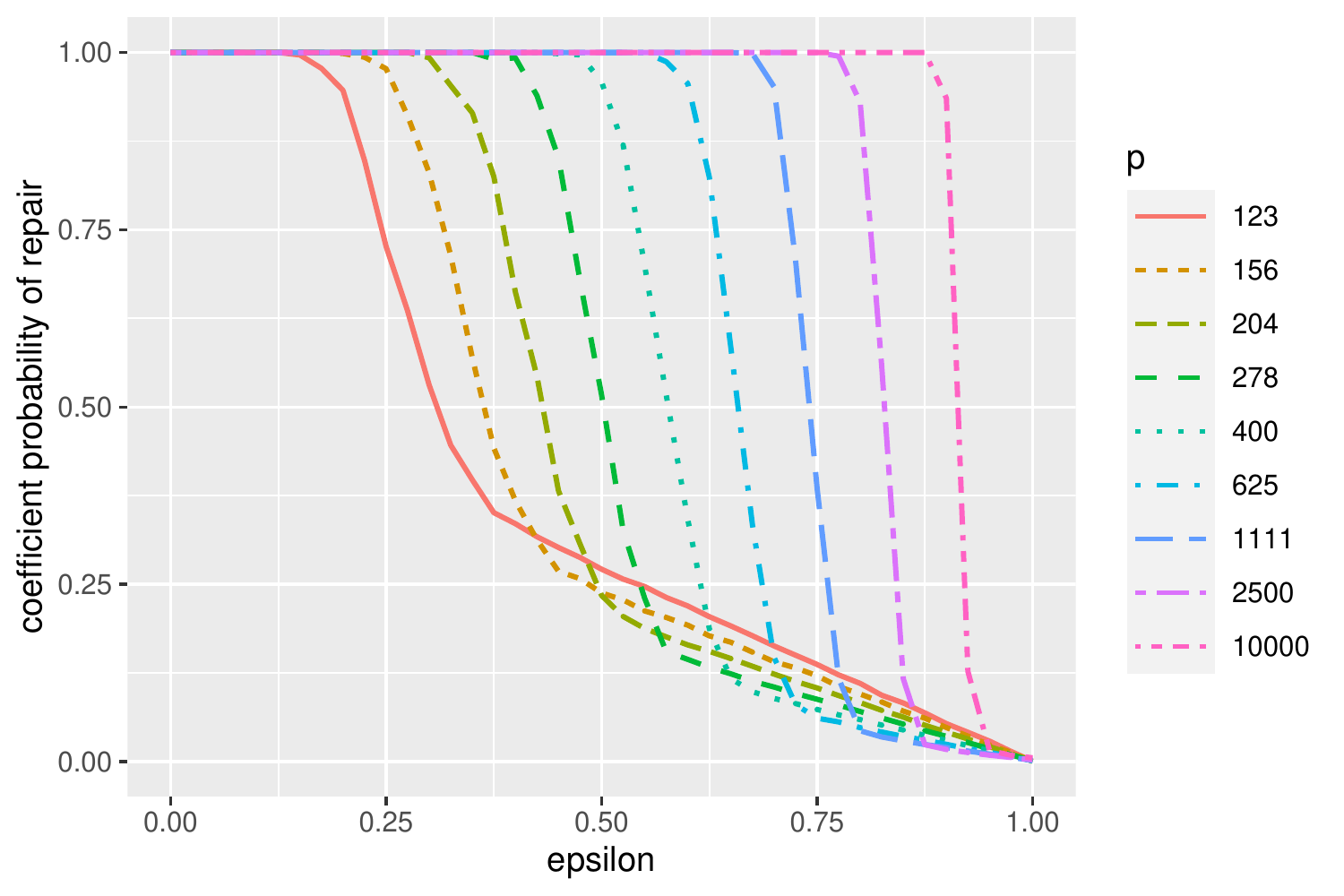}\\[-10pt]
    \end{tabular}
  \end{center}
\caption{Left: Empirical probability of successful model repair for $n=50$ with the model dimension $p$
varying as $p/n = 200 /j^2$. Right: Per-coefficient probability of successful repair.}
\label{fig:exp1b}
\end{figure}

\begin{figure}[t]
  \begin{center}
    \begin{tabular}{c}
      $n=50$ and $p=500$ fixed, varying mean $\mu$\\
      \includegraphics[width=.45\textwidth]{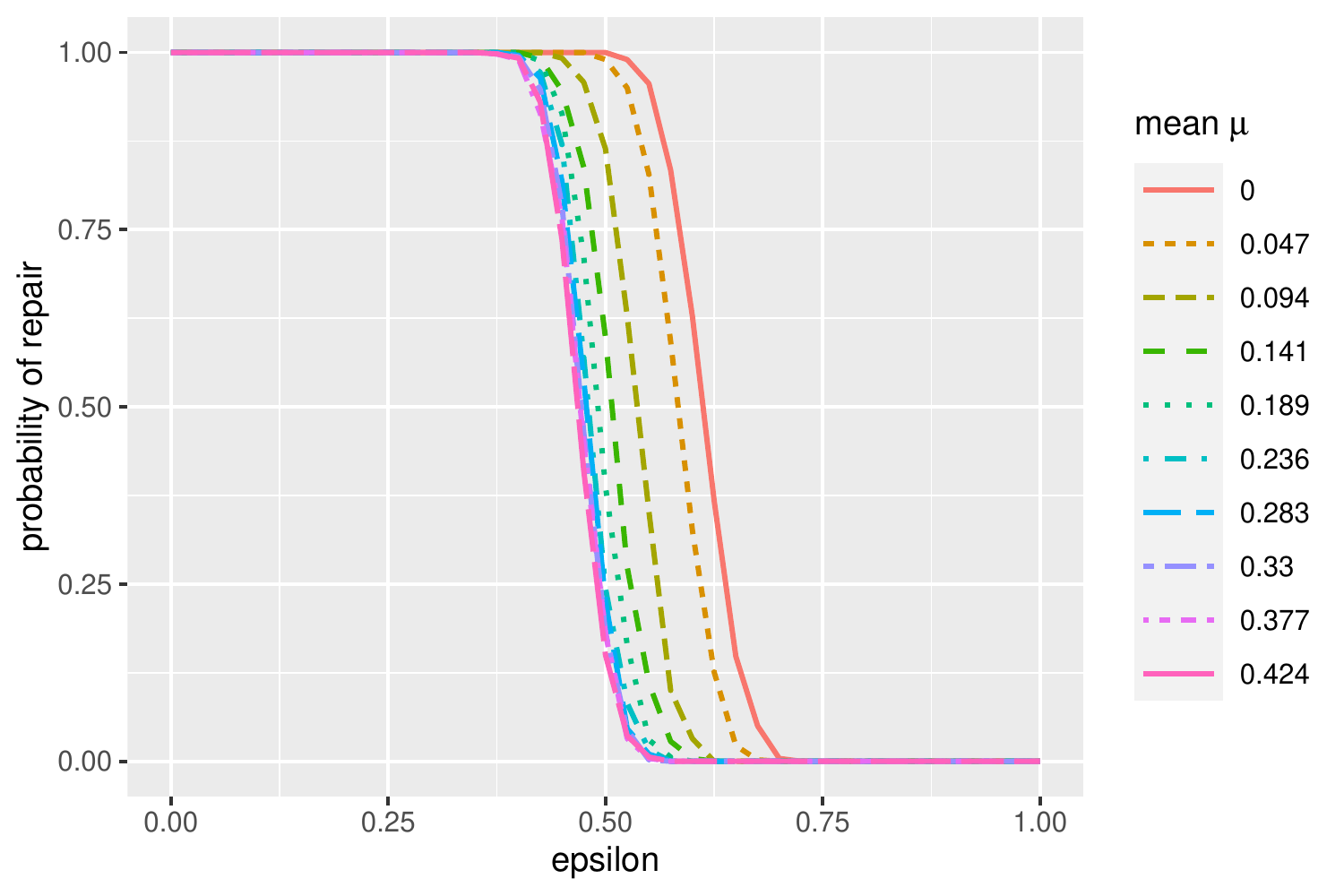}
    \end{tabular}
  \end{center}
\caption{Model repair for linear models with design entries $X_{i,j}\sim N(\mu, 1)$,
where the mean $\mu$ is varied and the sample size and dimension are fixed at $n=50$ and $p=500$. Consistent with Theorem~\ref{thm:main-improved}, a smaller corruption fraction $\epsilon$ is tolerated as the mean $\mu$ increases. In the plot above, the means are chosen as $\mu_j = c_j/\sqrt{n}$ for $c_j$ varying between zero and two.}
\label{fig:mean}
\vskip5pt
\begin{center}
  \begin{tabular}{cc}
    {\scriptsize $n=50$} & {\scriptsize $n=100$} \\
    \includegraphics[width=.45\textwidth]{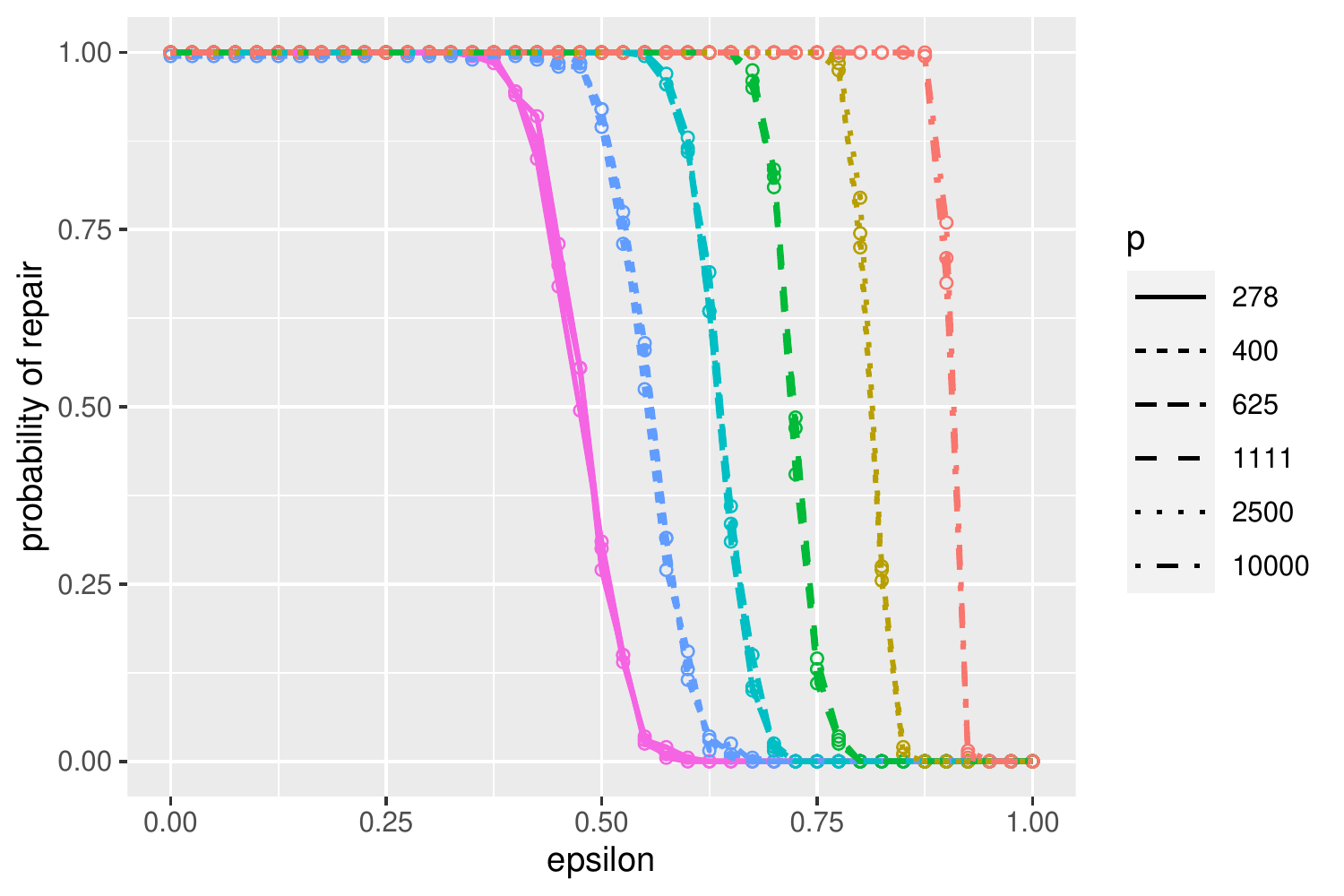} &
    \includegraphics[width=.45\textwidth]{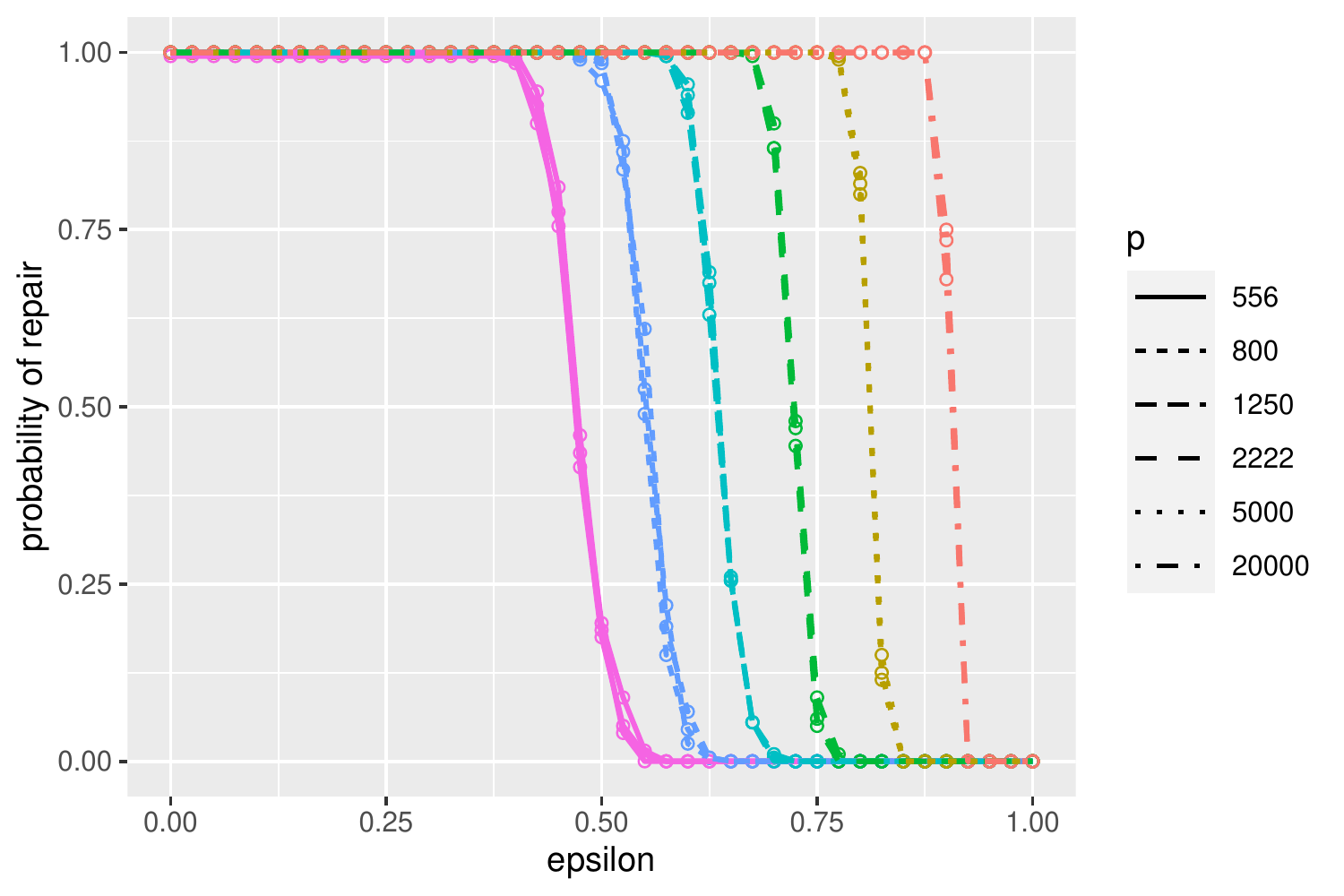}\\[-10pt]
  \end{tabular}
\end{center}
\caption{Model repair for random feature models $y=\psi(XW)\theta + w$ with $p>n$, where $\psi(\cdot) = \tanh(\cdot)$
for $n=50$ (left) and $n=100$ (right). For each value of $p$, three values of $d$ are evaluated, $d=p$, $d=\lceil 2p/3\rceil$,
and $d=\lceil p/2\rceil$; the results are effectively the same for each $d$. The curves are very similar when $\tanh$ is replaced by ReLU, as long as the population mean is subtracted from the features.}
\label{fig:rf}a
\end{figure}

\subsection{Varying the mean}

In this experiment we simulate over-parameterized linear models
with nonzero mean. The data are generated as $X_{ij} \sim N(\mu,1)$
independently, where we vary the mean $\mu$ and fix the dimension $p=500$ and sample size $n=50$.
The probability of successful repair is shown in Figure~\ref{fig:mean}.

As expected, the fraction $\epsilon$ that allows successful repair decreases; it appears to saturate at some fixed value $\epsilon_{\min}$. The mean $\mu$ cannot be made too large because it causes the design $X$ to become ill-conditioned, and the median regression fails.

The inequality in Condition $A$ in this case takes the form
\begin{align*}
  \mathbb{E}\left\|\frac{1}{m}\sum_{i=1}^m c_i a_i\right\|^2 &\leq \frac{k}{m} + \mu^2 k
   \equiv \frac{n}{p} + \mu^2 n.
   %  = \frac{(1+\mu^2)n}{p} + c_j.
\end{align*}
The means $\mu$ in Figure~\ref{fig:mean} are taken to be $\mu_j = c_j / \sqrt{n}$ as $c_j$ varies between zero and two.

\subsection{Random features models trained with gradient descent}

In this experiment we simulate over-parameterized random features models.
We generate $n$ data points $(\wt{x}_i, y_i)$ where
$y_i = \wt{x}_i^T \theta^* + w_i$ with $w_i$ an additive noise term. The covariates are generated
as a layer of a random neural network, with $\wt{x}_i = \tanh(W^Tx_i)$ where $x_i \in\reals^d$ with $x_{ij} \sim N(0,1)$
and $W\in\reals^{d\times p}$ with $W_{ij} \sim N(0, 1/d)$. We then approximate the least squares
solution using gradient descent intialized at zero, with updates
\begin{equation*}
  \hat\theta^{(t)} = \hat\theta^{(t-1)} + \frac{\eta}{n} \wt{X}^T R^{(t-1)}
\end{equation*}
where the residual vector $R^{(t-1)}\in\reals^n$ is given by $R_i^{(t-1)} = (y_i - \wt{x}_i^ T\hat\theta^{(t-1)})$.
The step size $\eta$ is selected empirically to insure convergence in under $1{,}000$ iterations.
Figure~\ref{fig:rf} shows two sets of results, for $n=50$ and $n=100$. For each value of the final dimension $p$,
three values of the original data dimension $d$ are selected: $d=p$, $d=\lceil 2p/3\rceil$,
and $d=\lceil p/2\rceil$. The recovery success curves for gradient descent are similar to those obtained for the minimal norm solution. The results are also similar if the ReLU activation function is used, as long as the population mean of the features is subtracted.

\begin{figure}[t]
  \begin{center}
    \vskip15pt
    \begin{tabular}{ccc}
    \hskip-6pt\includegraphics[width=.32\textwidth]{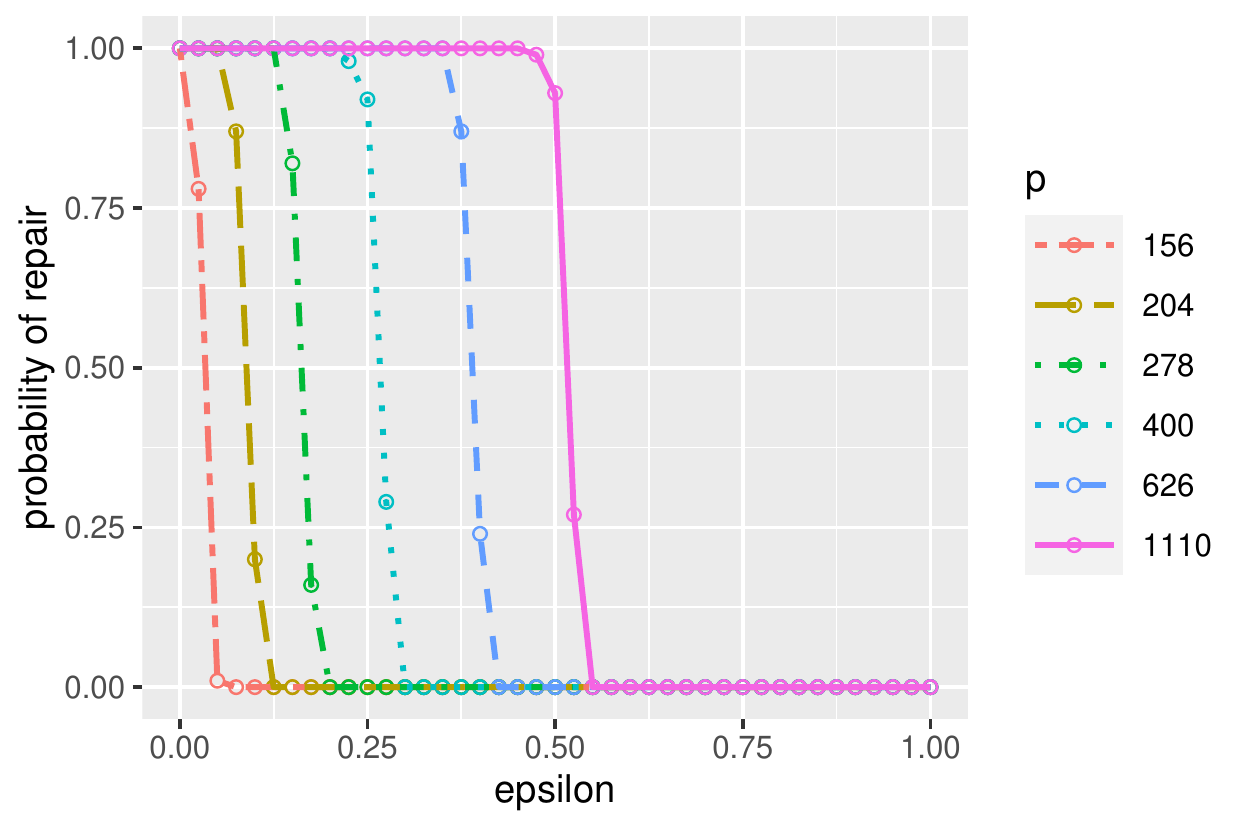} &
    \hskip-6pt\includegraphics[width=.32\textwidth]{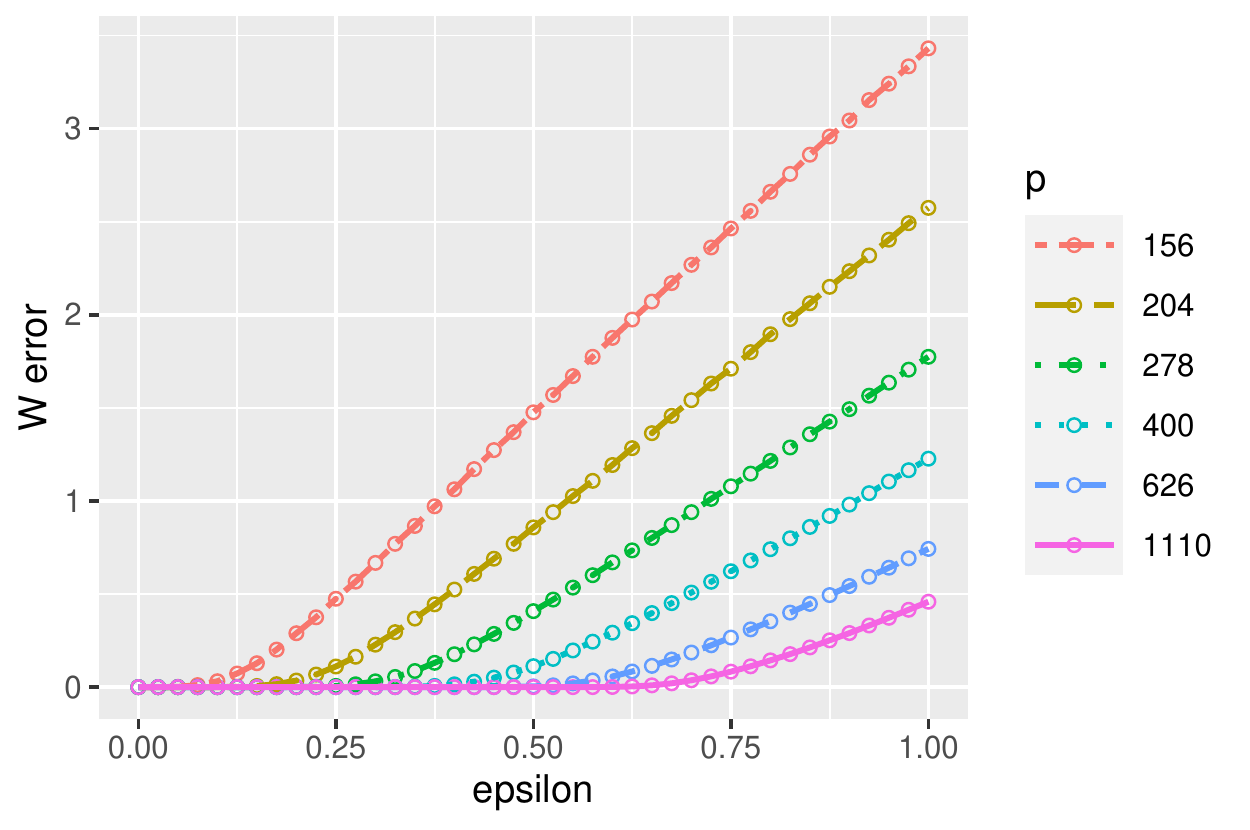} &
    \hskip-6pt\includegraphics[width=.32\textwidth]{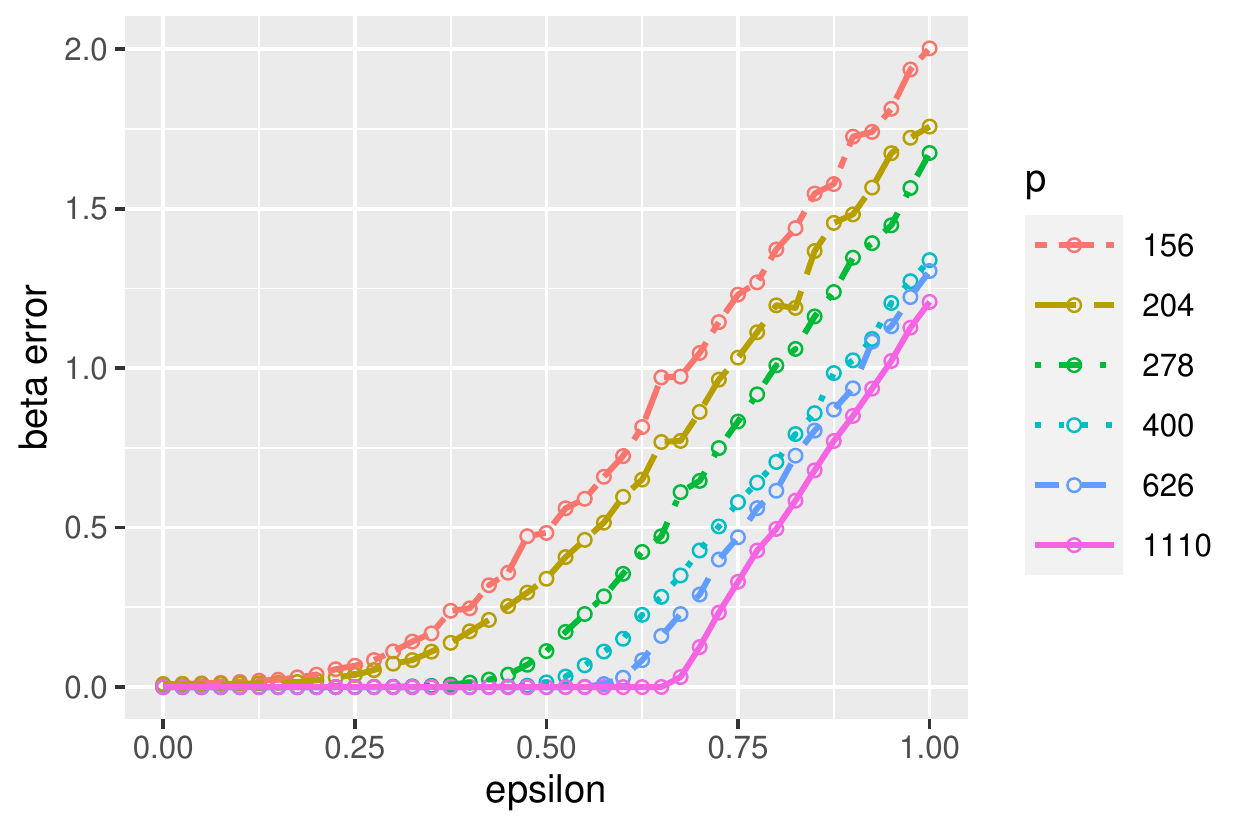}

    \end{tabular}
  \end{center}
\caption{Model repair for neural networks with a single hidden layer. The sample size is fixed at $n=50$, the number of hidden units is $p$, and the input dimension is $d=p/2$; the dimenson of the design is $X\in\reals^{n\times d}$, and the hidden layer is generated
as $\tanh(XW)$ where $W\in\reals^{d\times p}$. The predicted values are $\hat y = \tanh(XW)\beta$.
Left: The neural network is trained with gradient descent after initializing $W_{ij}$ as $N(0, 1/d)$
and $\beta \sim N(0, I_p)$. After training, a forward pass is made where $\hat W$ is fixed and the parameters
$\beta$ are retrained using gradient descent initialized at zero. This allows exact repair by running the linear program in stages, first repairing $\hat W$, and then repairing $\hat \beta$. Center and right: The neural network is trained with gradient descent using random initialization of $W$ and $\beta$; no forward pass is made after training. The weight matrix $W$ and weight vector $\beta$ are then not repaired exactly. The center and right plots show the average $L_2$ error in the estimated coefficients $W_{ij}$ and $\beta_j$ as a function of the corruption fraction $\epsilon$. In all plots, each point is the average over $100$ random trials. }
\label{fig:ann}
\end{figure}

\subsection{Neural networks}
\vskip10pt

In this final set of simulations we investigate repair algorithms for neural networks. We report results using a single hidden layer and the use of the hyperbolic tangent activation function. Results using the ReLU activation are similar as long as the features are centered.

As described in Section~\ref{sec:neural}, we consider two ways of training the models---with or without a forward pass to train the network parameters in stages. In the first approach, the network is trained using gradient descent, and the weights $\hat W\in\reals^{d\times p}$ are then fixed. Next, the weights $\beta\in\reals^p$ are initialized at zero, and gradient descent over $\beta$ is carried out using features
$\tilde X = \psi(X\hat W)$. This two-pass approach allows for exact repair, and only the initial weights $W(0)$ need to be accessed by the repair algorithm, using the seed value used in the pseudorandom number generator.

The left plot in Figure~\ref{fig:ann} shows the behavior of the linear program for exact repair when using this two-stage training algorithm.
The sample size is fixed at $n=50$, the number of ``neurons'' $W_j$ varies as $p$, and we take $d=p/2$. It can been seen that similar repair curves are obtained as for linear models, but the curves are shifted toward the left, indicating an overall smaller probability of successful repair. This is because successful repair requires that
$O(p^2)$ parameters are recovered, each of the $p$ columns $W_j\in\reals^d$ in addition to the vector $\beta\in\reals^p$.

In the second approach, the neural network is trained using standard gradient descent, without a final forward pass. As described in
Section~\ref{sec:neural}, when trained in this manner the column space of $\tilde X = \psi(X W)$ is continually changing. However,
the ``neural tangent kernel'' analysis ensures that the linear program will approximately recover the trained parameters after they are corrupted by additive noise. This is seen in the center and right plots of Figure~\ref{fig:ann}, which show the
squared errors $(\tilde W_{ij} - \hat W_{ij})^2$ and $(\tilde \beta_j - \hat\beta_j)^2$, averaged over $i$ and $j$. These results are consistent with our analysis, and suggest that the growth conditions on $d$ and $p$ in the results of our theorems are conservative.

% !TEX root = ./repair.tex

\def\ones{{\mathds{1}}}

\section{Discussion}
\label{sec:discuss}

In this paper we introduced the problem of model repair, related it to robust estimation, and established a series of results showing the theoretical performance of a repair algorithm that is based on median regression. The specific models treated include linear models and families of neural networks trained using gradient descent. The experimental results largely validate the theory, quantifying how model repair requires over-parameterization in the model and redundancy in the estimator.

This work suggests several directions to explore in future research.  A natural problem is to establish lower bounds for  model repair. In particular, our results show the level of over-parameterization sufficient for repair algorithms based on $\ell_1$ optimization. What level is required if the algorithm is not specified? Answering this question might exploit the rich literature on depth functions and multivariate generalizations of the median, together with minimax analysis for estimation and testing under the classical Huber model \citep{chen2018robust,diakonikolas,diakonikolas:2017}. In a different direction, \cite{gao2018robust} introduces a connection between these optimizations and certain learning algorithms for adversarial neural networks called $f$-GANs, giving a variational characterization of robust estimation that could lead to new algorithmic procedures for model repair.

The repair problem also could be formulated in other ways.  For example,
the corruption model could be modified, allowing a dependence between $z$ and $X$;
a simple form of this dependence would be $\eta_j \given X \sim (1-\epsilon) \delta_{\hat\theta_j} + \epsilon Q_j$.
What if the repair algorithm does not have access to the original training inputs $x_1,\ldots, x_n$?
If a new unlabeled dataset $x'_1,\ldots, x'_{m}$ is available for which $\text{span}(x_1,\ldots,x_n) \subset \text{span}(x'_1,\ldots,x'_{m})$, the results proven here will carry over. One could consider other formulations that make different assumptions on the information that is available.

It can be expected that the results for neural networks with a single layer established in the current paper
can be extended to multiple layers, based on results for multilayer networks that extend the analysis of gradient descent
of \cite{du2018gradient}, including \cite{zhuli} and \cite{dulee}. It would be interesting to
consider model repair for other architectures and estimation algorithms, including convolutional networks
and deep generative networks \citep{goodfellow2014generative,nvp,glow}.

Another natural direction to explore is repair for other families of statistical models, where over-parameterization and redundancy may take different forms. For instance, in classical Gaussian sequence models for orthonormal bases, additional coefficients could provide insurance against corrupted estimates.
%Concretely,
%consider $X_i \sim N(\theta_i, n^{-1})$ independently for all $i\in [n]$, where the signal
%$\theta$ belongs to a Sobolev ellipsoid of smoothness $\alpha$. A (non-adaptive) rate-optimal
%estimator is $\hat\theta_i = X_i$ for all $i\leq k_\alpha \equiv n^{\frac{1}{2\alpha+1}}$ and $\hat\theta_i = 0$ for $i$ %larger than $k_\alpha$.
%Suppose that the estimate is then corrupted by an adversary to
%$\eta = (\eta_1, \eta_2, \ldots)$, where up to $\epsilon n$ coefficients are changed according to an unknown distribution, %with $\eta_i \given \hat\theta_i \sim (1-\epsilon ) \delta_{\hat\theta_i} + \epsilon Q_i$
%independently for all $i$. Consider a hard thresholding estimator $\tilde\theta_i = \eta_i \ones(|\eta_i| \leq \lambda)$ for %$i \leq m$ and $\tilde \theta_i = 0$ for $i > m$.  By analyzing the bias, variance, and approximation error of this estimator %it can be shown that $m$ and $\lambda$ can be chosen to guarantee that
%$$ \E \| \tilde\theta - \theta\|^2 \lesssim \left(\frac{1}{n} + \epsilon \log (1/\epsilon) \right)^{\frac{1}{1+2\alpha}}.$$
%The question of optimality of this procedure could be studied.
It would also be interesting
to explore sequence models such as isotonic and shape-constrained regression. For example, consider the piecewise constant signals with $k$ pieces,
$$\Theta_k  = \{\theta : \mbox{$\theta_i = \mu_j$ for $i\in (a_{j-1}, a_j]$ for
some $0=a_0 \leq a_1\leq \cdots \leq a_k=n$}\}.$$
Adaptivity of the least--squares estimator to $k$ has been well-established
\citep{chatterjee2015risk,Chat14,bellec}; but
the redundancy in the sequence could also be exploited in model repair. If an initial estimator $\hat \theta$ is
corrupted to $\eta = \hat\theta + z$, a natural repair procedure is
$$\tilde \theta_i = \text{mode}(\{\eta_{i-h},\ldots,\eta_{i+h}\})$$
with $h$ acting as a bandwidth parameter. We conjecture that
$\inf_h \E\|\tilde\eta -\theta\|^2 = O(k \log (n/k))$. In this setting,
the piecewise constant signal acts as a simple repetition code, with majority vote
serving as a natural decoding procedure.

Returning to some of the motivation mentioned in the introduction, when training increasingly large neural networks
it becomes necessary to estimate the models in a distributed manner, and erasures and errors may occur when
communicating parameters across nodes, or after the trained model has been embedded in an application.
Instead of running the repair program on a central hub,
which would require sharing data and potentially compromising privacy, the linear program
might also be distributed \citep{hong12}. Finally, drawing an analogy to brain plasticity and repair after trauma, if a spatially localized part of a multilayer network is permanently corrupted, the repair problem needs to be reformulated to allow ``rewiring'' the parameters to obtain a model whose predictions are close to those of the original model, possibly through specialized training. With appropriate formalization, these and other extensions might permit statistical analysis.

% !TEX root = ./repair.tex

\section*{Acknowledgments}

Research of CG is supported in part by NSF grant DMS-1712957 and NSF CAREER award DMS-1847590. Research of JL is supported in part by NSF grant CCF-1839308.

\setlength{\bibsep}{8pt plus 0.3ex}
\bibliographystyle{apalike}
\bibliography{repair}

\clearpage
% !TEX root = ./repair.tex

\appendix

\section{Proofs}
\label{sec:proof}
\overfullrule=1mm

\subsection{Technical Lemmas}

We present a few technical lemmas that will be used in the proofs. The first lemma is Hoeffding's inequality.
\begin{lemma}[\cite{hoeffding1963probability}]\label{lem:hoeffding}
Consider independent random variables $X_1,...,X_n$ that satisfy $X_i\in[a_i,b_i]$ for all $i\in[n]$. Then, for any $t>0$,
$$\mathbb{P}\left(\left|\sum_{i=1}^n(X_i-\mathbb{E}X_i)\right|>t\right) \leq 2\exp\left(-\frac{2t^2}{\sum_{i=1}^n(b_i-a_i)^2}\right).$$
\end{lemma}

Next, we need a central limit theorem with an explicit third moment bound. The following lemma is Theorem 2.20 of \cite{ross2007second}.
\begin{lemma}\label{lem:stein}
If $Z\sim N(0,1)$ and $W=\sum_{i=1}^nX_i$ where $X_i$ are independent mean $0$ and $\Var(W)=1$, then
$$\sup_z\left|\mathbb{P}(W\leq z)-\mathbb{P}(Z\leq z)\right|\leq 2\sqrt{3\sum_{i=1}^n\mathbb{E}|X_i|^3}.$$
\end{lemma}

We also need a Talagrand Gaussian concentration inequality. The following version has explicit constants.
\begin{lemma}[\cite{cirel1976norms}]\label{lem:talagrand}
Let $f:\mathbb{R}^k\rightarrow\mathbb{R}$ be a Lipschitz function with constant $L>0$. That is, $|f(x)-f(y)|\leq L\|x-y\|$ for all $x,y\in\mathbb{R}^k$. Then, for any $t>0$,
$$\mathbb{P}\left(|f(Z)-\mathbb{E}f(Z)|>t\right)\leq 2\exp\left(-\frac{t^2}{2L^2}\right),$$
where $Z\sim N(0,I_k)$.
\end{lemma}

Finally, we present two lemmas on the concentration of norms and inner products of multivariate Gaussians.
\begin{lemma}[\cite{laurent2000adaptive}]\label{lem:chi-squared}
For any $t>0$, we have
\begin{eqnarray*}
\mathbb{P}\left(\chi_k^2\geq k+2\sqrt{tk}+2t\right) &\leq& e^{-t}, \\
\mathbb{P}\left(\chi_k^2\leq k-2\sqrt{tk}\right) &\leq& e^{-t}.
\end{eqnarray*}
\end{lemma}

\begin{lemma}\label{lem:inner-prod}
Consider independent $Y_1,Y_2\sim N(0,I_k)$. For any $t>0$, we have
\begin{eqnarray*}
\mathbb{P}\left(|\|Y_1\|\|Y_2\|-k|\geq 2\sqrt{tk}+2t\right) &\leq& 4e^{-t}, \\
\mathbb{P}\left(|Y_1^TY_2| \geq \sqrt{2kt}+2t\right) &\leq& 2e^{-t}.
\end{eqnarray*}
\end{lemma}
\begin{proof}
By Lemma \ref{lem:chi-squared}, we have
\begin{align*}
\mathbb{P}\bigl(\|Y_1\|\|Y_2\| -k \geq & 2\sqrt{tk}+2t\bigr) \\
&\leq \mathbb{P}\left(\|Y_1\|^2 \geq k +2\sqrt{tk}+2t\right) + \mathbb{P}\left(\|Y_2\|^2 \geq k +2\sqrt{tk}+2t\right) \\
&\leq 2e^{-t},
\end{align*}
and
\begin{align*}
\mathbb{P}\Bigl(\|Y_1\|\|Y_2\| -k \leq & -2\sqrt{tk}-2t\Bigr)\\
&\leq \mathbb{P}\left(\|Y_1\|^2 \leq k -2\sqrt{tk}\right) + \mathbb{P}\left(\|Y_2\|^2 \leq k -2\sqrt{tk}\right) \\
&\leq 2e^{-t}.
\end{align*}
Summing up the two bounds above, we obtain the first conclusion. For the second conclusion, note that
$$\mathbb{P}\left(Y_1^TY_2 \geq x\right)\leq e^{-\lambda x}\mathbb{E}e^{\lambda Y_1^TY_2}=\exp\left(-\lambda x-\frac{k}{2}\log(1-\lambda^2)\right)\leq \exp\left(-\lambda x+\frac{k}{2}\lambda^2\right),$$
for any $x>0$ and $\lambda\in (0,1)$. Optimize over $\lambda\in (0,1)$, and we obtain $\mathbb{P}\left(Y_1^TY_2>x\right)\leq e^{-\frac{1}{2}\left(\frac{x^2}{k}\wedge x\right)}$. Take $x=\sqrt{2kt}+2t$, and then we obtain the bound
$$\mathbb{P}\left(Y_1^TY_2 \geq \sqrt{2kt}+2t\right)\leq e^{-t},$$
which immediately implies the second conclusion.
\end{proof}

\subsection{Proofs of Theorem \ref{thm:robust-reg} and Theorem \ref{thm:robust-reg-b}} \label{sec:pf-robust-reg}

We first establish an empirical process result.
\begin{lemma}\label{lem:EP}
Consider independent random variables $z_1,...,z_m$. Assume $k/m\leq 1$. Then, for any $t\in (0,1/2)$ and any fixed $A^T=(a_1,...,a_m)^T\in\mathbb{R}^{m\times k}$ such that (\ref{eq:l2-upper-A}) holds, we have
$$\sup_{\|\Delta\|\leq t}\left|\frac{1}{m}\sum_{i=1}^m[(|a_i^T\Delta-z_i|-|z_i|)-\mathbb{E}(|a_i^T\Delta-z_i|-|z_i|)]\right|| \lesssim t\overline{\lambda}\sqrt{\frac{k}{m}\log\left(\frac{em}{k}\right)},$$
with high probability.
\end{lemma}
\begin{proof}
We use the notation $G_m(\Delta)=\frac{1}{m}\sum_{i=1}^m[(|a_i^T\Delta-z_i|-|z_i|)-\mathbb{E}(|a_i^T\Delta-z_i|-|z_i|)]$, and we apply a discretization argument. For the Euclidian ball $B_k(t)=\{\Delta\in\mathbb{R}^k:\|\Delta\|\leq t\}$, there exists a subset $\mathcal{N}_{t,\zeta}\subset B_k(t)$, such that for any $\Delta\in B_k(t)$, there exists a $\Delta'\in\mathcal{N}_{t,\zeta}$ that satisfies $\|\Delta-\Delta'\|\leq \zeta$, and we also have the bound $\log|\mathcal{N}_{t,\zeta}|\leq k\log(1+2t/\zeta)$ according to Lemma 5.2 of \cite{vershynin2010introduction}. For any $\Delta\in B_k(t)$ and the corresponding $\Delta'\in\mathcal{N}_{t,\zeta}$ that satisfies $\|\Delta-\Delta'\|\leq \zeta$, we have
\begin{eqnarray*}
|G_m(\Delta)-G_m(\Delta')| &\leq& 2\frac{1}{m}\sum_{i=1}^m|a_i^T(\Delta-\Delta')| \\
&\leq& 2\sqrt{\frac{1}{m}\sum_{i=1}^m|a_i^T(\Delta-\Delta')|^2} \leq 2\overline{\lambda}\zeta,
\end{eqnarray*}
where the last line is due to the condition (\ref{eq:l2-upper-A}). Thus,
$$\left|G_m(\Delta)\right| \leq \left|G_m(\Delta')\right| + 2\overline{\lambda}\zeta.$$
Taking the supremum over both sides of the inequality, we obtain
\begin{equation}
\sup_{\|\Delta\|\leq t}|G_m(\Delta)| \leq \max_{\Delta\in\mathcal{N}_{t,\zeta}}\left|G_m(\Delta)\right| + 2\overline{\lambda}\zeta. \label{eq:disc-not-good}
\end{equation}
For any $\Delta\in B_k(t)$, we have
$$\frac{1}{m}\sum_{i=1}^m\left(|a_i^T\Delta-z_i|-|z_i|\right)^2\leq \frac{1}{m}\sum_{i=1}^m|a_i^T\Delta|^2\leq \overline{\lambda}^2t^2.$$
By Lemma \ref{lem:hoeffding}, we have
$$\mathbb{P}\left(\left|G_m(\Delta)\right| > x\right) \leq 2\exp\left(-\frac{2mx^2}{\overline{\lambda}^2t^2}\right).$$
A union bound argument leads to
\begin{equation}
\mathbb{P}\left(\max_{\Delta\in\mathcal{N}_{t,\zeta}}\left|G_m(\Delta)\right| > x\right) \leq 2\exp\left(-\frac{2mx^2}{\overline{\lambda}^2t^2}+k\log\left(1+\frac{2t}{\zeta}\right)\right).\label{eq:double-ub}
\end{equation}
By choosing $x^2\asymp \frac{t^2\bar{\lambda}^2k\log(1+2t/\zeta)}{m}$, we have
$$\max_{\Delta\in\mathcal{N}_{t,\zeta}}\left|G_m(\Delta)\right| \lesssim t\bar{\lambda}\sqrt{\frac{k\log(1+2t/\zeta)}{m}},$$
with high probability. Together with the bound (\ref{eq:disc-not-good}), we have
$$\sup_{\|\Delta\|\leq t}|G_m(\Delta)| \lesssim t\bar{\lambda}\sqrt{\frac{k\log(1+2t/\zeta)}{m}} + \bar{\lambda}\zeta,$$
with high probability.
The choice $\zeta=t\sqrt{k/m}$ leads to the desired result.
\end{proof}

\begin{proof}[Proof of Theorem \ref{thm:robust-reg}]
Recall the definition of $L_m(u)$ in the proof of Theorem \ref{thm:main-improved}.
We introduce i.i.d. Rademacher random variables $\delta_1,...,\delta_m$. With the notation $\wt{a}_i=\delta_ia_i$, $\wt{b}_i=\delta_ib_i$ and $\wt{z}_i=\delta_iz_i$, we can write
$$
L_m(u) = \frac{1}{m}\sum_{i=1}^m\left(|\wt{a}_i^T(u^*-u)+\wt{z}_i|-|\wt{z}_i|\right).
$$
Let $\wt{A}\in\mathbb{R}^{m\times k}$ be the matrix whose $i$th row is $\wt{a}_i^T$. By the symmetry of $A$, we have $\mathbb{P}(\wt{A}\in U|\delta)=\mathbb{P}(\wt{A}\in U)=\mathbb{P}(A\in U)$ for any measurable set $U$. Therefore, for any measurable sets $U$ and $V$, we have
\begin{eqnarray*}
\mathbb{P}(\wt{A}\in U, \wt{z}\in V) &=& \mathbb{E}\mathbb{P}(\wt{A}\in U, \wt{z}\in V|\delta) \\
&=& \mathbb{E}\mathbb{P}(\wt{A}\in U|\delta)\mathbb{P}(\wt{z}\in V|\delta) \\
&=& \mathbb{E}\mathbb{P}(\wt{A}\in U)\mathbb{P}(\wt{z}\in V|\delta) \\
&=& \mathbb{P}(\wt{A}\in U)\mathbb{P}(\wt{z}\in V),
\end{eqnarray*}
and thus $\wt{A}$ ad $\wt{z}$ are independent. Define $L(u)=\mathbb{E}(L_m(u)|\wt{A})$. Suppose $\|\wh{u}-u^*\|\geq t$, we must have
$$\inf_{\|u-u^*\|\geq t}L_m(u) \leq L_m(u^*).$$
By the convexity of $L_m(u)$, we can replace $\|u-u^*\|\geq t$ by $\|u-u^*\| = t$ and the above inequality still holds, and therefore $\inf_{\|u-u^*\|= t}L_m(u)\leq 0$. This implies
\begin{equation}
\inf_{\|u-u^*\|=t}L(u) \leq \sup_{\|u-u^*\|= t}|L_m(u)-L(u)|. \label{eq:basic-L}
\end{equation}
Now we study $L(u)$. Introduce the function $f_i(x)=\mathbb{E}(|x+\wt{z}_i|-|\wt{z}_i|)$ so that we can write $L(u)=\frac{1}{m}\sum_{i=1}^mf_i(\wt{a}_i^T(u^*-u))$. For any $x\geq 0$,
\begin{eqnarray*}
f_i(x) &=& \mathbb{E}(|x+\wt{z}_i|-|\wt{z}_i|)\mathbb{I}\{\wt{z}_i<-x\} + \mathbb{E}(|x+\wt{z}_i|-|\wt{z}_i|)\mathbb{I}\{\wt{z}_i> 0\} \\
&& + \mathbb{E}(|x+\wt{z}_i|-|\wt{z}_i|)\mathbb{I}\{-x\leq \wt{z}_i< 0\} + x\mathbb{P}(\wt{z}_i=0) \\
&=& -x\mathbb{P}(\wt{z}_i<-x) + x\mathbb{P}(\wt{z}_i> 0) + \mathbb{E}(x+2\wt{z}_i)\mathbb{I}\{-x\leq \wt{z}_i<0\}   + x\mathbb{P}(\wt{z}_i=0) \\
&\geq& -x\mathbb{P}(\wt{z}_i<-x) + x\mathbb{P}(\wt{z}_i> 0) - x\mathbb{P}(-x\leq \wt{z}_i< 0)   + x\mathbb{P}(\wt{z}_i=0) \\
&=& -x\mathbb{P}(\wt{z}_i<-x) + x\mathbb{P}(\wt{z}_i< 0) - x\mathbb{P}(-x\leq \wt{z}_i< 0)  + x\mathbb{P}(\wt{z}_i=0) \\
&\geq& x\mathbb{P}(\wt{z}_i=0) \\
&\geq& (1-\epsilon)x.
\end{eqnarray*}
By the symmetry of $\wt{z}_i$, we also have
$$f_i(-x)=\mathbb{E}(|-x+\wt{z}_i|-|\wt{z}_i|)=\mathbb{E}(|x-\wt{z}_i|-|\wt{z}_i|)=\mathbb{E}(|x+\wt{z}_i|-|\wt{z}_i|)=f_i(x),$$
which implies $f_i(x)\geq (1-\epsilon)|x|$. Therefore, for any $u$ such that $\|u-u^*\|=t$, we have
\begin{eqnarray*}
L(u) &=& \frac{1}{m}\sum_{i=1}^mf_i(\wt{a}_i^T(u^*-u)) \\
&\geq& (1-\epsilon)\frac{1}{m}\sum_{i=1}^m|\wt{a}_i^T(u^*-u)| \\
&=& (1-\epsilon)\frac{1}{m}\sum_{i=1}^m|a_i^T(u^*-u)| \\
&\geq& \underline{\lambda}(1-\epsilon)t,
\end{eqnarray*}
where the last inequality is by (\ref{eq:l1-upper-A}). Together with (\ref{eq:basic-L}), we have
\begin{equation}
\mathbb{P}\left(\|\wh{u}-u\|\geq t\right) \leq \mathbb{P}\left(\sup_{\|u-u^*\|= t}|L_m(u)-L(u)| \geq \underline{\lambda}(1-\epsilon)t/2\right). \label{eq:larger-t-prob}
\end{equation}
Since the condition (\ref{eq:l2-upper-A}) continues to hold with $A$ replaced by $\wt{A}$, we can apply Lemma \ref{lem:EP} and obtain that
$$\sup_{\|u-u^*\|= t}|L_m(u)-L(u)| \lesssim t\overline{\lambda}\sqrt{\frac{k}{m}\log\left(\frac{em}{k}\right)},$$
with high probability. Under the conditions of the theorem, we know that $\frac{t\overline{\lambda}\sqrt{\frac{k}{m}\log\left(\frac{em}{k}\right)}}{\underline{\lambda}(1-\epsilon)t}$ is sufficiently small, and thus by (\ref{eq:larger-t-prob}), $\|\wh{u}-u^*\|<t$ with high probability. Since $t$ is arbitrary, we must have $\wh{u}=u^*$.
\end{proof}

\begin{proof}[Proof of Theorem \ref{thm:robust-reg-b}]
Recall the definitions of $L_m(u)$ and $L(u)$ in the proof of Theorem \ref{thm:main-improved}. Define
$$K_m(u)=\frac{1}{m}\sum_{i=1}^m\left(|b_i+a_i^T(u^*-u)+z_i|-|z_i|\right).$$
It is easy to see that
\begin{equation}
\sup_u|L_m(u)-K_m(u)|\leq \frac{1}{m}\sum_{i=1}^m|b_i|.\label{eq:K-L-b}
\end{equation}
Suppose $\|\wh{u}-u^*\|\geq t$, we must have $\inf_{\|u-u^*\|\geq t}K_m(u) \leq K_m(u^*)$.
By the convexity of $K_m(u)$, we can replace $\|u-u^*\|\geq t$ by $\|u-u^*\| = t$ and the inequality still holds. By (\ref{eq:K-L-b}), we have $K_m(u^*)\leq \frac{1}{m}\sum_{i=1}^m|b_i|$, and therefore $\inf_{\|u-u^*\|= t}K_m(u)\leq \frac{1}{m}\sum_{i=1}^m|b_i|$. Since
\begin{eqnarray*}
\inf_{\|u-u^*\|= t}K_m(u) &\geq& \inf_{\|u-u^*\|= t}L_m(u) - \frac{1}{m}\sum_{i=1}^m|b_i| \\
&\geq& \inf_{\|u-u^*\|=t}L(u) + \inf_{\|u-u^*\|= t}(L_m(u)-L(u)) - \frac{1}{m}\sum_{i=1}^m|b_i|,
\end{eqnarray*}
we then have
\begin{equation}
\inf_{\|u-u^*\|=t}L(u) \leq \sup_{\|u-u^*\|= t}|L_m(u)-L(u)| + 2\frac{1}{m}\sum_{i=1}^m|b_i|. %\label{eq:basic-L}
\end{equation}
With the lower bound for (\ref{eq:L-decomp}) and the upper bound (\ref{eq:upper-EP}), we obtain
$$\underline{\lambda}(1-\epsilon)t - \epsilon t\sigma\sqrt{\frac{k}{m}} - 2\frac{1}{m}\sum_{i=1}^m|b_i|\lesssim t\overline{\lambda}\sqrt{\frac{k}{m}\log\left(\frac{em}{k}\right)},$$
which is impossible with the choice $t=\frac{4\frac{1}{m}\sum_{i=1}^m|b_i|}{\underline{\lambda}(1-\epsilon)}$ when $\frac{\overline{\lambda}\sqrt{\frac{k}{m}\log\left(\frac{em}{k}\right)}+\epsilon\sigma\sqrt{\frac{k}{m}}}{\underline{\lambda}(1-\epsilon)}$ is sufficiently small. Thus, we obtain the desired conclusion.
\end{proof}

\subsection{Proofs of Lemma \ref{lem:design-linear}, Corollary \ref{cor:repair-linear}, Lemma \ref{lem:design-rf} and Corollary \ref{cor:repair-rf}}

\begin{proof}[Proof of Lemma \ref{lem:design-linear}]
Condition $A$ is obvious. For Condition $B$, we have
$$
\inf_{\|\Delta\|=1}\frac{1}{p}\sum_{j=1}^p|a_j^T\Delta| \geq \sqrt{\frac{2}{\pi}} - \sup_{\|\Delta\|=1}\left|\frac{1}{p}\sum_{j=1}^p|a_j^T\Delta| - \sqrt{\frac{2}{\pi}} \right|,
$$
and we will analyze the second term on the right hand side of the inequality above via a discretization argument for the Euclidean sphere $S^{n-1}=\{\Delta\in\mathbb{R}^n: \|\Delta\|=1\}$. There exists a subset $\mathcal{N}_{\zeta}\subset S^{n-1}$, such that for any $\Delta\in S^{n-1}$, there exists a $\Delta'\in\mathcal{N}_{\zeta}$ that satisfies $\|\Delta-\Delta'\|\leq\zeta$, and we also have the bound $\log|\mathcal{N}_{\zeta}|\leq n\log\left(1+2/\zeta\right)$ according to Lemma 5.2 of \cite{vershynin2010introduction}. For any $\Delta\in S^{n-1}$ and the corresponding $\Delta'\in\mathcal{N}_{\zeta}$ that satisfies $\|\Delta-\Delta'\|\leq \zeta$, we have
\begin{eqnarray*}
\left|\frac{1}{p}\sum_{j=1}^p|a_j^T\Delta| - \sqrt{\frac{2}{\pi}} \right| &\leq& \left|\frac{1}{p}\sum_{j=1}^p|a_j^T\Delta'| - \sqrt{\frac{2}{\pi}} \right|  + \zeta\sup_{\|\Delta\|=1}\frac{1}{p}\sum_{j=1}^p|a_j^T\Delta| \\
&\leq& \left|\frac{1}{p}\sum_{j=1}^p|a_j^T\Delta'| - \sqrt{\frac{2}{\pi}} \right| + \zeta\sup_{\|\Delta\|=1}\left|\frac{1}{p}\sum_{j=1}^p|a_j^T\Delta| - \sqrt{\frac{2}{\pi}} \right| + \zeta\sqrt{\frac{2}{\pi}}.
\end{eqnarray*}
Taking the supremum on both sides of the inequality, with some rearrangement, we obtain
$$\sup_{\|\Delta\|=1}\left|\frac{1}{p}\sum_{j=1}^p|a_j^T\Delta| - \sqrt{\frac{2}{\pi}} \right|\leq (1-\zeta)^{-1}\max_{\Delta\in\mathcal{N}_{\zeta}}\left|\frac{1}{p}\sum_{j=1}^p|a_j^T\Delta| - \sqrt{\frac{2}{\pi}} \right| + \frac{\zeta}{1-\zeta}\sqrt{\frac{2}{\pi}}.$$
Setting $\zeta=1/3$, we then have
$$\inf_{\|\Delta\|=1}\frac{1}{p}\sum_{j=1}^p|a_j^T\Delta| \geq (2\pi)^{-1} - \frac{3}{2}\max_{\Delta\in\mathcal{N}_{1/3}}\left|\frac{1}{p}\sum_{j=1}^p|a_j^T\Delta| - \sqrt{\frac{2}{\pi}} \right|.$$
Lemma \ref{lem:talagrand} together with a union bound argument leads to
$$\mathbb{P}\left(\max_{\Delta\in\mathcal{N}_{1/3}}\left|\frac{1}{p}\sum_{j=1}^p|a_j^T\Delta| - \sqrt{\frac{2}{\pi}} \right|>t\right)\leq 2\exp\left(n\log(7)-\frac{pt^2}{2}\right),$$
which implies $\max_{\Delta\in\mathcal{N}_{1/3}}\left|\frac{1}{p}\sum_{j=1}^p|a_j^T\Delta| - \sqrt{\frac{2}{\pi}} \right|\lesssim \sqrt{\frac{n}{p}}$ with high probability. Since $n/p$ is sufficiently small, we have $\inf_{\|\Delta\|=1}\frac{1}{p}\sum_{j=1}^p|a_j^T\Delta|\gtrsim 1$ with high probability as desired. The high probability bound $\sup_{\|\Delta\|=1}\frac{1}{p}\sum_{j=1}^p|a_j^T\Delta|^2=\opnorm{A}^2/p\lesssim 1+n/p$ is by \cite{davidson2001local}, and the proof is complete.
\end{proof}

\begin{proof}[Proof of Corollary \ref{cor:repair-linear}]
Since $\wh{\theta}$ belongs to the row space of $X$, there exists some $u^*\in\mathbb{R}^n$ such that $\wh{\theta}=X^Tu^*$.
By Theorem \ref{thm:main-improved} and Lemma \ref{lem:design-linear}, we know that $\wt{u}=u^*$ with high probability, and therefore $\wt{\theta}=X^T\wt{u}=X^Tu^*=\wh{\theta}$.
\end{proof}

Now we state the proof of Lemma \ref{lem:design-rf}. Note that Condition $A$ holds because
$$\sum_{i=1}^n\mathbb{E}\left(\frac{1}{p}\sum_{j=1}^pc_j\psi(W_j^Tx_i)\right)^2  \leq \sum_{i=1}^n\frac{1}{p^2}\sum_{j=1}^p\mathbb{E}|\psi(W_j^Tx_i)|^2\leq \frac{n}{p},$$
and we only need to prove Condition $B$. We present the proofs of (\ref{eq:l1-upper-A}) and (\ref{eq:l2-upper-A}) separately.
\begin{proof}[Proof of (\ref{eq:l1-upper-A}) of Lemma \ref{lem:design-rf}]
Let us adopt the notation that
$$f(W,X,\Delta)=\frac{1}{p}\sum_{j=1}^p\left|\sum_{i=1}^n\psi(W_j^Tx_i)\Delta_i\right|.$$
Define $g(X,\Delta)=\mathbb{E}(f(W,X,\Delta)|X)$.
We then have
\begin{eqnarray}
\nonumber \inf_{\|\Delta\|=1}f(W,X,\Delta) &\geq& \inf_{\|\Delta\|=1}\mathbb{E}f(W,X,\Delta) - \sup_{\|\Delta\|=1}\left|f(W,X,\Delta)-\mathbb{E}f(W,X,\Delta)\right| \\
\label{eq:exp-f-inf-relu} &\geq& \inf_{\|\Delta\|=1}\mathbb{E}f(W,X,\Delta) \\
\label{eq:ep-f-relu} && - \sup_{\|\Delta\|=1}\left|f(W,X,\Delta)-\mathbb{E}(f(W,X,\Delta)|X)\right| \\
\label{eq:ep-g-relu} && - \sup_{\|\Delta\|=1}\left|g(X,\Delta)-\mathbb{\mathbb{E}}g(X,\Delta)\right|.
\end{eqnarray}
We will analyze the three terms above separately.

\paragraph{Analysis of (\ref{eq:exp-f-inf-relu}).} For any $\Delta$ such that $\|\Delta\|=1$, we have
\begin{eqnarray*}
\mathbb{E}f(W,X,\Delta) &=& \mathbb{E}\left|\sum_{i=1}^n\psi(W^Tx_i)\Delta_i\right| \\
 &\geq& \mathbb{E}\left(\left|\sum_{i=1}^n\psi(W^Tx_i)\Delta_i\right|\mathbb{I}\left\{\left|\sum_{i=1}^n\psi(W^Tx_i)\Delta_i\right|\geq 1, 1/2\leq \|W\|^2\leq 2\right\}\right) \\
 &\geq& \mathbb{P}\left(\left|\sum_{i=1}^n\psi(W^Tx_i)\Delta_i\right|\geq 1, 1/2\leq \|W\|^2\leq 2\right) \\
 &=& \mathbb{P}\left(\left|\sum_{i=1}^n\psi(W^Tx_i)\Delta_i\right|\geq 1\Big|1/2\leq \|W\|^2\leq 2\right)\mathbb{P}\left(1/2\leq \|W\|^2\leq 2\right) \\
\nonumber &\geq& \mathbb{P}\left(\left|\sum_{i=1}^n\psi(W^Tx_i)\Delta_i\right|\geq 1\Big|1/2\leq \|W\|^2\leq 2\right)\left(1-2\exp(-d/16)\right),
\end{eqnarray*}
where the last inequality is by Lemma \ref{lem:chi-squared}. It is easy to see that $$\Var\left(\psi(W^Tx)|W\right)\leq \mathbb{E}(|\psi(W^Tx)|^2|W)\leq 1.$$
Moreover, for any $W$ such that $1/2\leq \|W\|^2\leq 2$,
$$\Var\left(\psi(W^Tx)|W\right)= \mathbb{E}(|\psi(W^Tx)|^2|W) \geq \frac{1}{5}\mathbb{P}\left(|W^Tx|>1/2|W\right)\geq \frac{1}{5}\mathbb{P}(|N(0,1)|\geq 1/\sqrt{2}),$$
which is at least $1/20$. In summary, we have
$$1/20 \leq \Var\left(\psi(W^Tx)|W\right) \leq 1,$$
for any $W$ such that $1/2\leq \|W\|^2\leq 2$.
By Lemma \ref{lem:stein}, we have
\begin{eqnarray*}
 && \mathbb{P}\left(\left|\sum_{i=1}^n\psi(W^Tx_i)\Delta_i\right|\geq 1\Big|1/2\leq \|W\|^2\leq 2\right) \\
 &\geq& \mathbb{P}\left(\frac{\left|\sum_{i=1}^n\psi(W^Tx_i)\Delta_i\right|}{\sqrt{\Var\left(\psi(W^Tx)|W\right)}}\geq \sqrt{20}\Bigg|1/2\leq \|W\|^2\leq 2\right) \\
 &\geq& \mathbb{P}\left(N(0,1)>\sqrt{20}\right) - \sup_{1/2\leq \|W\|^2\leq 2} 2\sqrt{3\sum_{i=1}^n|\Delta_i|^3\frac{\mathbb{E}\left(|\psi(W^Tx_i)|^3|W\right)}{\left(\Var\left(\psi(W^Tx)|W\right)\right)^{3/2}}} \\
 &\geq& \mathbb{P}\left(N(0,1)>\sqrt{20}\right) - 35\sqrt{\sum_{i=1}^n|\Delta_i|^3} \\
 &\geq& \mathbb{P}\left(N(0,1)>\sqrt{20}\right) - 35\max_{1\leq i\leq n}|\Delta_i|^{3/2}.
\end{eqnarray*}
Hence, when $\max_{1\leq i\leq n}|\Delta_i|^{3/2}\leq \delta_0^{3/2}:=\mathbb{P}\left(N(0,1)>\sqrt{20}\right)/70$, we can lower bound the expectation $\mathbb{E}f(W,X,\Delta)$ by an absolute constant, and we conclude that
\begin{equation}
\inf_{\|\Delta\|=1, \max_{1\leq i\leq n}|\Delta_i|\leq\delta_0}\mathbb{E}f(W,X,\Delta) \gtrsim 1.\label{appeq:l1-1-2}
\end{equation}

We also need to consider the case when $\max_{1\leq i\leq n}|\Delta_i|> \delta_0$. Without loss of generality, we can assume $\Delta_1>\delta_0$.
We then lower bound $\mathbb{E}f(W,X,\Delta)$ by
\begin{eqnarray*}
&& \mathbb{E}\left(\left|\sum_{i=1}^n\psi(W^Tx_i)\Delta_i\right|\mathbb{I}\left\{\sum_{i=1}^n\psi(W^Tx_i)\Delta_i \geq \delta_0/2, 1/2\leq \|W\|^2\leq 2\right\}\right) \\
&\geq& \frac{\delta_0}{2} \mathbb{P}\left(\sum_{i=1}^n\psi(W^Tx_i)\Delta_i \geq \delta_0/2\Big| 1/2\leq \|W\|^2\leq 2\right)\mathbb{P}\left(1/2\leq \|W\|^2\leq 2\right) \\
&\geq& \frac{\delta_0}{2} \mathbb{P}\left(\psi(W^Tx_1)\Delta_1 \geq \delta_0/2\Big|1/2\leq \|W\|^2\leq 2\right) \\
&& \times \mathbb{P}\left(\sum_{i=2}^n\psi(W^Tx_i)\Delta_i \geq 0\Big|1/2\leq \|W\|^2\leq 2\right)\left(1-2\exp(-d/16)\right) \\
&=& \frac{\delta_0}{4} \mathbb{P}\left(\psi(W^Tx_1)\Delta_1 \geq \delta_0/2\Big|1/2\leq \|W\|^2\leq 2\right)\left(1-2\exp(-d/16)\right).
\end{eqnarray*}
For any $W$ that satisfies $1/2\leq \|W\|^2\leq 2$, we have
\begin{eqnarray*}
\mathbb{P}\left(\psi(W^Tx_1)\Delta_1 \geq \delta_0/2\Big|W\right) &\geq& \mathbb{P}\left(\psi(W^Tx_1)\geq 1/2\Big|W\right) \\
&\geq&  \mathbb{P}\left(W^Tx_1\geq 1\Big|W\right) \\
&\geq& \mathbb{P}\left(N(0,1)\geq \sqrt{2}\right),
\end{eqnarray*}
which is a constant.
Therefore, we have
$$\mathbb{E}f(W,X,\Delta)\geq \frac{\delta_0}{4}\left(1-2\exp(-d/16)\right)\mathbb{P}\left(N(0,1)\geq \sqrt{2}\right)\gtrsim 1,$$
and we can conclude that
\begin{equation}
\inf_{\|\Delta\|=1, \max_{1\leq i\leq n}|\Delta_i|\geq\delta_0}\mathbb{E}f(W,X,\Delta) \gtrsim 1.\label{appeq:l1-1-3}
\end{equation}
Combining the two cases (\ref{appeq:l1-1-2}) and (\ref{appeq:l1-1-3}),  we obtain the conclusion that
$$\inf_{\|\Delta\|=1}\mathbb{E}f(W,X,\Delta)\gtrsim 1.$$

\paragraph{Analysis of (\ref{eq:ep-f-relu}).} We now denote the conditional expectation operator $\mathbb{E}(\cdot|X)$ by $\mathbb{E}^X$. Letting $\wt{W}$ be an independent copy of $W$, we first bound the moment generating function via a standard symmetrization argument. For any $\lambda>0$,
\begin{eqnarray}
\nonumber && \mathbb{E}^X\exp\left(\lambda \sup_{\|\Delta\|=1}\left|f(W,X,\Delta)-\mathbb{E}^Xf(W,X,\Delta)\right|\right) \\
\nonumber &\leq& \mathbb{E}^X\exp\left(\lambda \mathbb{E}^{X,W}\sup_{\|\Delta\|=1}\left|f(W,X,\Delta)-f(\wt{W},X,\Delta)\right|\right) \\
\nonumber &\leq& \mathbb{E}^X\exp\left(\lambda \sup_{\|\Delta\|=1}\left|f(W,X,\Delta)-f(\wt{W},X,\Delta)\right|\right) \\
\nonumber &=& \mathbb{E}^X\exp\left(\lambda\sup_{\|\Delta\|=1}\left|\frac{1}{p}\sum_{j=1}^p\epsilon_j\left(\left|\sum_{i=1}^n\psi(W_j^Tx_i)\Delta_i\right|-\left|\sum_{i=1}^n\psi(\wt{W}_j^Tx_i)\Delta_i\right|\right)\right|\right) \\
\label{eq:mgf-b-relu} &\leq& \mathbb{E}^X\exp\left(2\lambda\sup_{\|\Delta\|=1}\left|\frac{1}{p}\sum_{j=1}^p\epsilon_j\left|\sum_{i=1}^n\psi(W_j^Tx_i)\Delta_i\right|\right|\right),
\end{eqnarray}
where $\epsilon_1,...,\epsilon_p$ are independent Rademacher random variables. Let us adopt the notation
$$F(\epsilon,W,X,\Delta)=\frac{1}{p}\sum_{j=1}^p\epsilon_j\left|\sum_{i=1}^n\psi(W_j^Tx_i)\Delta_i\right|.$$
We use a discretization argument. For the Euclidean sphere $S^{n-1}=\{\Delta\in\mathbb{R}^n: \|\Delta\|=1\}$, there exists a subset $\mathcal{N}\subset S^{n-1}$, such that for any $\Delta\in S^{n-1}$, there exists a $\Delta'\in\mathcal{N}$ that satisfies $\|\Delta-\Delta'\|\leq 1/2$, and we also have the bound $\log|\mathcal{N}|\leq 2n$. See, for example, Lemma 5.2 of \cite{vershynin2010introduction}.
For any $\Delta\in S^{n-1}$ and the corresponding $\Delta'\in\mathcal{N}$ that satisfies $\|\Delta-\Delta'\|\leq 1/2$, we have
\begin{eqnarray*}
|F(\epsilon,W,X,\Delta)| &\leq& |F(\epsilon,W,X,\Delta')| + |F(\epsilon,W,X,\Delta-\Delta')| \\
&\leq& |F(\epsilon,W,X,\Delta')| + \frac{1}{2}\sup_{\|\Delta\|=1}|F(\epsilon,W,X,\Delta)|,
\end{eqnarray*}
which, by taking supremum over both sides, implies
$$\sup_{\|\Delta\|=1}|F(\epsilon,W,X,\Delta)|\leq 2\max_{\Delta\in\mathcal{N}}|F(\epsilon,W,X,\Delta)|.$$
Define $\bar{F}(\epsilon,X,\Delta)=\mathbb{E}^{\epsilon,X}F(\epsilon,W,X,\Delta)$, and then
$$\max_{\Delta\in\mathcal{N}}|F(\epsilon,W,X,\Delta)|\leq \max_{\Delta\in\mathcal{N}}|F(\epsilon,W,X,\Delta)-\bar{F}(\epsilon,X,\Delta)|+\max_{\Delta\in\mathcal{N}}|\bar{F}(\epsilon,X,\Delta)|.$$
In view of (\ref{eq:mgf-b-relu}), we obtain the bound
\begin{align}
\nonumber \mathbb{E}^X\exp\Bigl(\lambda & \sup_{\|\Delta\|=1}\Bigl|f(W,X,\Delta)-\mathbb{E}^Xf(W,X,\Delta)\Bigr|\Bigr) \\
\nonumber &\leq \mathbb{E}^X\exp\left(4\lambda \max_{\Delta\in\mathcal{N}}|F(\epsilon,W,X,\Delta)-\bar{F}(\epsilon,X,\Delta)|+4\lambda \max_{\Delta\in\mathcal{N}}|\bar{F}(\epsilon,X,\Delta)|\right) \\
\label{eq:tala-mgf1-relu} &\leq \frac{1}{2}\sum_{\Delta\in\mathcal{N}} \mathbb{E}^{X}\exp\left(4\lambda|F(\epsilon,W,X,\Delta)-\bar{F}(\epsilon,X,\Delta)|\right) \\
\label{eq:tala-mgf2-relu} & \qquad + \frac{1}{2}\sum_{\Delta\in\mathcal{N}} \mathbb{E}^{X}\exp\left(4\lambda |\bar{F}(\epsilon,X,\Delta)|\right).
\end{align}
We will bound the two terms above on the event $E=\left\{\sum_{i=1}^n\|x_i\|^2\leq 3nd\right\}$. For any $W,\wt{W}$, we have
\begin{align*}
\left|F(\epsilon,W,X,\Delta)-F(\epsilon,\wt{W},X,\Delta)\right| &\leq \frac{1}{p}\sum_{j=1}^p\sum_{i=1}^n\left|({\psi}(W_j^Tx_i)-{\psi}(\wt{W}_j^Tx_i))\Delta_i\right| \\
&\leq \frac{1}{p}\sum_{j=1}^p\sum_{i=1}^n|(W_j-\wt{W}_j)^Tx_i||\Delta_i| \\
&\leq \frac{1}{p}\sum_{j=1}^p\sum_{i=1}^n\|W_j-\wt{W}_j\|\|x_i\||\Delta_i| \\
&\leq \frac{1}{\sqrt{p}}\sqrt{\sum_{j=1}^p\|W_j-\wt{W}_j\|^2}\sqrt{\sum_{i=1}^n\|x_i\|^2} \\
&\leq \sqrt{\frac{3n}{p}}\sqrt{\sum_{j=1}^p\|\sqrt{d}W_j-\sqrt{d}\wt{W}_j\|^2},
\end{align*}
where the last inequality holds under the event $E$. By Lemma \ref{lem:talagrand}, we have for any $X$ such that $E$ holds,
$$\mathbb{P}\left(|F(\epsilon,W,X,\Delta)-\bar{F}(\epsilon,X,\Delta)|>t\big|X\right)\leq 2\exp\left(-\frac{pt^2}{6n}\right),$$
for any $t>0$. The sub-Gaussian tail implies a bound for the moment generating function. By Lemma 5.5 of \cite{vershynin2010introduction}, we have
$$\mathbb{E}^{X}\exp\left(4\lambda|F(\epsilon,W,X,\Delta)-\bar{F}(\epsilon,X,\Delta)|\right) \leq \exp\left(C_1\frac{n}{p}\lambda^2\right),$$
for some constant $C_1>0$. To bound the moment generating function of $\bar{F}(\epsilon,X,\Delta)$, we note that
\begin{align*}
|\bar{F}(\epsilon,X,\Delta)|\ &\leq \left|\frac{1}{p}\sum_{j=1}^p\epsilon_j\right|\mathbb{E}^X\left|\sum_{i=1}^n\psi(W^Tx_i)\Delta_i\right| \\
&\leq \left|\frac{1}{p}\sum_{j=1}^p\epsilon_j\right|\sqrt{\sum_{i=1}^n\mathbb{E}^X|\psi(W^Tx_i)|^2} \\
&\leq \left|\frac{1}{p}\sum_{j=1}^p\epsilon_j\right|\sqrt{\sum_{i=1}^n\|x_i\|^2/d} \leq \sqrt{3n}\left|\frac{1}{p}\sum_{j=1}^p\epsilon_j\right|,
\end{align*}
where the last inequality holds under the event $E$. With an application of Hoeffding-type inequality (Lemma 5.9 of \cite{vershynin2010introduction}), we have
$$\mathbb{E}^{X}\exp\left(4\lambda |\bar{F}(\epsilon,X,\Delta)|\right)\leq \mathbb{E}\exp\left(4\lambda\sqrt{3n}\left|\frac{1}{p}\sum_{j=1}^p\epsilon_j\right|\right)\leq \exp\left(C_1\frac{n}{p}\lambda^2\right).$$
Note that we can use the same constant $C_1$ by making its value sufficiently large. Plug the two moment generating function bounds into (\ref{eq:tala-mgf1-relu}) and (\ref{eq:tala-mgf2-relu}), and we obtain the bound
$$\mathbb{E}^X\exp\left(\lambda \sup_{\|\Delta\|=1}\left|f(W,X,\Delta)-\mathbb{E}^Xf(W,X,\Delta)\right|\right)\leq \exp\left(C_1\frac{n}{p}\lambda^2+2n\right),$$
for any $X$ such that $E$ holds. To bound (\ref{eq:ep-f}), we apply Chernoff bound, and then
$$\mathbb{P}\left(\sup_{\|\Delta\|=1}\left|f(W,X,\Delta)-\mathbb{E}(f(W,X,\Delta)|X)\right| > t\right) \leq \exp\left(-\lambda t + C_1\frac{n}{p}\lambda^2+2n\right).$$
Optimize over $\lambda$, set $t\asymp \sqrt{\frac{n^2}{p}}$, and we have
$$\sup_{\|\Delta\|=1}\left|f(W,X,\Delta)-\mathbb{E}(f(W,X,\Delta)|X)\right|\lesssim \sqrt{\frac{n^2}{p}},$$
with high probability.

\paragraph{Analysis of (\ref{eq:ep-g-relu}).} We use a discretization argument. There exists a subset $\mathcal{N}_{\zeta}\subset S^{n-1}$, such that for any $\Delta\in S^{n-1}$, there exists a $\Delta'\in\mathcal{N}_{\zeta}$ that satisfies $\|\Delta-\Delta'\|\leq\zeta$, and we also have the bound $\log|\mathcal{N}|\leq n\log\left(1+2/\zeta\right)$ according to Lemma 5.2 of \cite{vershynin2010introduction}. For any $\Delta\in S^{n-1}$ and the corresponding $\Delta'\in\mathcal{N}_{\zeta}$ that satisfies $\|\Delta-\Delta'\|\leq \zeta$, we have
\begin{eqnarray*}
|g(X,\Delta) - \mathbb{E}g(X,\Delta)| &\leq& |g(X,\Delta')-\mathbb{E}g(X,\Delta')| \\
&& + |g(X,\Delta-\Delta')-\mathbb{E}g(X,\Delta-\Delta')|  \\
&& + 2\mathbb{E}g(X,\Delta-\Delta') \\
&\leq& |g(X,\Delta')-\mathbb{E}g(X,\Delta')| \\
&& + \zeta\sup_{\|\Delta\|=1}|g(X,\Delta)-\mathbb{E}g(X,\Delta)|  \\
&& + 2\zeta\sup_{\|\Delta\|=1}\mathbb{E}g(X,\Delta).
\end{eqnarray*}
Take supremum over both sides, arrange the inequality, and we obtain the bound
\begin{eqnarray}
\label{eq:union-tala-g-relu} \sup_{\|\Delta\|=1}|g(X,\Delta) - \mathbb{E}g(X,\Delta)| &\leq& (1-\zeta)^{-1}\max_{\Delta\in\mathcal{N}_{\zeta}}|g(X,\Delta) - \mathbb{E}g(X,\Delta)| \\
\label{eq:small-exp-relu} && 2\zeta(1-\zeta)^{-1}\mathbb{E}g(X,\Delta).
\end{eqnarray}
To bound (\ref{eq:union-tala-g-relu}), we will use Lemma \ref{lem:talagrand} together with a union bound argument. For any $X,\wt{X}$, we have
\begin{eqnarray*}
|g(X,\Delta)-g(\wt{X},\Delta)| &\leq& \mathbb{E}^X\left|\sum_{i=1}^n(\psi(W_j^Tx_i)-\psi(W_j^T\wt{x}_j))\Delta_i\right| \\
&\leq& \mathbb{E}^X\sqrt{\sum_{i=1}^n\left(\psi(W_j^Tx_i)-\psi(W_j^T\wt{x}_j)\right)^2} \\
&\leq& \sqrt{\sum_{i=1}^n\mathbb{E}^X\left(W_j^T(x_i-\wt{x}_i)\right)^2} \\
&=& \frac{1}{\sqrt{d}}\sqrt{\sum_{i=1}^n\|x_i-\wt{x}_i\|^2}.
\end{eqnarray*}
Therefore, by Lemma \ref{lem:talagrand},
$$\mathbb{P}\left(|g(X,\Delta)-g(\wt{X},\Delta)|>t\right) \leq 2\exp\left(-\frac{dt^2}{2}\right),$$
for any $t>0$. A union bound argument leads to
$$\mathbb{P}\left(\max_{\Delta\in\mathcal{N}_{\zeta}}|g(X,\Delta) - \mathbb{E}g(X,\Delta)|>t\right)\leq 2\exp\left(-\frac{dt^2}{2}+n\log\left(1+\frac{2}{\zeta}\right)\right),$$
which implies that
$$\max_{\Delta\in\mathcal{N}_{\zeta}}|g(X,\Delta) - \mathbb{E}g(X,\Delta)|\lesssim \sqrt{\frac{n\log(1+2/\zeta)}{d}},$$
with high probability. For (\ref{eq:small-exp-relu}), we have
$$\mathbb{E}g(X,\Delta)\leq \sqrt{\mathbb{E}\Var\left(\sum_{i=1}^n\psi(W^Tx_i)\Delta_i\Big|W\right)}\leq \sqrt{\mathbb{E}|\psi(W^Tx)|^2}\leq 1.$$
Combining the bounds for (\ref{eq:union-tala-g-relu}) and (\ref{eq:small-exp-relu}), we have
$$\sup_{\|\Delta\|=1}|g(X,\Delta) - \mathbb{E}g(X,\Delta)|\lesssim \sqrt{\frac{n\log(1+2/\zeta)}{d}} + \zeta,$$
with high probability as long as $\zeta\leq 1/2$. We choose $\zeta=\sqrt{n/d}$, and thus the bound is sufficiently small as long as $n/d$ is sufficiently small.

Finally, combine results for (\ref{eq:exp-f-inf-relu}), (\ref{eq:ep-f-relu}) and (\ref{eq:ep-g-relu}), and we obtain the desired conclusion as long as $n^2/p$ and $n/d$ are sufficiently small.
\end{proof}

To prove (\ref{eq:l2-upper-A}) of Lemma \ref{lem:design-rf}, we establish the following stronger result.
\begin{lemma}\label{lem:lim-G}
Consider independent $W_1,\ldots,W_p\sim N(0,d^{-1}I_d)$ and $x_1,\ldots,x_n\sim N(0,I_d)$. We define the matrices $G,\bar{G}\in\mathbb{R}^{n\times n}$ by
\begin{eqnarray*}
G_{il} &=& \frac{1}{p}\sum_{j=1}^p\psi(W^T_jx_i)\psi(W_j^Tx_l), \\
\bar{G}_{il} &=& |\mathbb{E}\psi'(Z)|^2\frac{x_i^Tx_l}{\|x_i\|\|x_l\|} + \left(\mathbb{E}|\psi(Z)|^2-|\mathbb{E}\psi'(Z)|^2\right)\mathbb{I}\{i=l\},
\end{eqnarray*}
where $Z\sim N(0,1)$.
Assume $d/\log n$ is sufficiently large, and then
$$\opnorm{G-\bar{G}}^2\lesssim \frac{n^2}{p} + \frac{\log n}{d} + \frac{n^2}{d^2},$$
with high probability. Therefore, if we further assume $n^2/p$ and $n/d$ are sufficiently small, we also have
\begin{equation}
1\lesssim\lambda_{\min}(G)\leq \lambda_{\max}(G)\lesssim 1, \label{eq:spectrum-G-bound}
\end{equation}
with high probability.
\end{lemma}
\begin{proof}
Define $\wt{G}\in\mathbb{R}^{n\times n}$ with entries $\wt{G}_{il}=\mathbb{E}\left(\psi(W^Tx_i)\psi(W^Tx_l)|X\right)$, and we first bound the difference between $G$ and $\wt{G}$. Note that
$$\mathbb{E}(G_{il}-\wt{G}_{il})^2 = \mathbb{E}\Var(G_{il}|X) \leq \frac{1}{p}\mathbb{E}|\psi(W^Tx_i)\psi(W^Tx_l)|^2 \leq p^{-1}.$$
We then have
$$
\mathbb{E}\opnorm{G-\wt{G}}^2 \leq \mathbb{E}\fnorm{G-\wt{G}}^2 \leq \frac{n^2}{p}.
$$
By Markov's inequality,
\begin{equation}
\opnorm{G-\wt{G}}^2 \lesssim \frac{n^2}{p}, \label{eq:G-G-tilde}
\end{equation}
with high probability.

Next, we study the diagonal entries of $\wt{G}$. For any $i\in[n]$,
$$\wt{G}_{ii}=\mathbb{E}(|\psi(W^Tx_i)|^2|X)=\mathbb{E}_{U\sim N(0,\|x_i\|^2/d)}|\psi(U)|^2.$$
Therefore,
$$\max_{1\leq i\leq n}|\wt{G}_{ii}-\bar{G}_{ii}|\leq \max_{1\leq i\leq n}\TV\left(N(0,\|x_i\|^2/d), N(0,1)\right)\leq\frac{3}{2}\max_{1\leq i\leq n}\left|\frac{\|x_i\|^2}{d}-1\right|.$$
By Lemma \ref{lem:chi-squared} and a union bound argument, we have
\begin{equation}
\max_{1\leq i\leq n}|\wt{G}_{ii}-\bar{G}_{ii}|\lesssim \sqrt{\frac{\log n}{d}}, \label{eq:G-diag}
\end{equation}
with high probability.

Now we analyze the off-diagonal entries. We use the notation $\bar{x}_i=\frac{\sqrt{d}}{\|x_i\|}x_i$. For any $i\neq l$, we have
\begin{eqnarray}
\label{eq:G-tilde-1} \wt{G}_{il} &=& \mathbb{E}\left(\psi(W^T\bar{x}_i)\psi(W^T\bar{x}_l)|X\right) \\
\label{eq:G-tilde-2} && + \mathbb{E}\left((\psi(W^Tx_i)-\psi(W^T\bar{x}_i))\psi(W^T\bar{x}_l)|X\right) \\
\label{eq:G-tilde-3} && + \mathbb{E}\left(\psi(W^T\bar{x}_i)(\psi(W^Tx_l)-\psi(W^T\bar{x}_l))|X\right) \\
\label{eq:G-tilde-4} && + \mathbb{E}\left((\psi(W^Tx_i)-\psi(W^T\bar{x}_i))(\psi(W^Tx_l)-\psi(W^T\bar{x}_l))|X\right).
\end{eqnarray}
For first term on the right hand side of (\ref{eq:G-tilde-1}), we observe that $\mathbb{E}\left(\psi(W^T\bar{x}_i)\psi(W^T\bar{x}_l)|X\right)$ is a function of $\frac{\bar{x}_i^T\bar{x}_l}{d}$, and thus we can write
$$\mathbb{E}\left(\psi(W^T\bar{x}_i)\psi(W^T\bar{x}_l)|X\right)=f\left(\frac{\bar{x}_i^T\bar{x}_l}{d}\right),$$
where
$$f(\rho) = \begin{cases}
\mathbb{E}\psi(\sqrt{1-\rho}U+\sqrt{\rho}Z)\psi(\sqrt{1-\rho}V+\sqrt{\rho}Z), & \rho \geq 0, \\
\mathbb{E}\psi(\sqrt{1+\rho}U-\sqrt{-\rho}Z)\psi(\sqrt{1+\rho}V+\sqrt{-\rho}Z), & \rho < 0,
\end{cases}$$
with $U,V,Z\stackrel{iid}{\sim} N(0,1)$. By some direct calculations, we have $f(0)=0$, $f'(0)=(\mathbb{E}\psi'(Z))^2$, and $\sup_{|\rho|\leq 0.2}|f''(\rho)|\lesssim 1$. Therefore, as long as $|\bar{x}_i^T\bar{x}_l|/d\leq 1/5$,
$$\left|f\left(\frac{\bar{x}_i^T\bar{x}_l}{d}\right)-(\mathbb{E}\psi'(Z))^2\frac{\bar{x}_i^T\bar{x}_l}{d}\right|\leq C_1\left|\frac{\bar{x}_i^T\bar{x}_l}{d}\right|^2,$$
for some constant $C_1>0$. By Lemma \ref{lem:inner-prod}, we know that $\max_{i\neq l}|\bar{x}_i^T\bar{x}_l|/d\lesssim \sqrt{\frac{\log n}{d}}\leq 1/5$ with high probability, which then implies
\begin{equation}
\sum_{i\neq l}\left(\mathbb{E}\left(\psi(W^T\bar{x}_i)\psi(W^T\bar{x}_l)|X\right)-\bar{G}_{il}\right)^2 \leq C_1\sum_{i\neq l}\left|\frac{\bar{x}_i^T\bar{x}_l}{d}\right|^4. \label{eq:G-H}
\end{equation}
The term on the right hand side can be bounded by
$$\sum_{i\neq l}\left|\frac{\bar{x}_i^T\bar{x}_l}{d}\right|^4\leq \frac{d}{\min_{1\leq i\leq n}\|x_i\|^2}\sum_{i\neq l}\left|\frac{x_i^Tx_l}{d}\right|^4.$$
By Lemma \ref{lem:chi-squared}, $\frac{d}{\min_{1\leq i\leq n}\|x_i\|^2}\lesssim 1$ with high probability. By integrating out the probability tail bound of $|x_i^Tx_l|$ given in Lemma \ref{lem:inner-prod}, we have $\sum_{i\neq l}\mathbb{E}\left|\frac{x_i^Tx_l}{d}\right|^4\lesssim \frac{n^2}{d^2}$, and by Markov's inequality, we have $\sum_{i\neq l}\left|\frac{x_i^Tx_l}{d}\right|^4\lesssim \frac{n^2}{d^2}$ with high probability.

We also need to analyze the contributions of (\ref{eq:G-tilde-2}) and (\ref{eq:G-tilde-3}). We can write (\ref{eq:G-tilde-2}) as
\begin{eqnarray}
\label{eq:second-order-G-1} && \mathbb{E}\left[\psi(W^T\bar{x}_l)\psi'(W^T\bar{x}_i)W^T(x_i-\bar{x}_i)|X\right] \\
\label{eq:second-order-G-2} && + \frac{1}{2}\mathbb{E}\left[\psi(W^T\bar{x}_i)\psi''(t_i)|W^T(x_i-\bar{x}_i)|^2|X\right],
\end{eqnarray}
where $t_i$ is some random variable between $W^Tx_i$ and $W^T\bar{x}_i$. The first term (\ref{eq:second-order-G-1}) can be expressed as
$$\left(\frac{\|x_i\|}{\sqrt{d}}-1\right)\mathbb{E}\left[\psi(W^T\bar{x}_l)\psi'(W^T\bar{x}_i)W^T\bar{x}_i|X\right]=\left(\frac{\|x_i\|}{\sqrt{d}}-1\right)g\left(\frac{\bar{x}_i^T\bar{x}_l}{d}\right),$$
where the function $g$ satisfies $g(0)=0$ and $\sup_{|\rho|\leq 0.2}|g'(\rho)|\lesssim 1$, and thus
$$\left|g\left(\frac{\bar{x}_i^T\bar{x}_l}{d}\right)\right|\lesssim \left|\frac{\bar{x}_i^T\bar{x}_l}{d}\right|\lesssim \left|\frac{x_i^Tx_l}{d}\right|,$$
because of the high probability bound $\max_{i\neq l}|\bar{x}_i^T\bar{x}_l|/d\lesssim \sqrt{\frac{\log n}{d}}\leq 1/5$.
Therefore,
\begin{eqnarray}
\nonumber && \sum_{i\neq l}\left(\mathbb{E}\left[\psi(W^T\bar{x}_l)\psi'(W^T\bar{x}_i)W^T(x_i-\bar{x}_i)|X\right]\right)^2 \\
\nonumber &\lesssim& \sum_{i\neq l}\left|\frac{\|x_i\|}{\sqrt{d}}-1\right|^2\left|\frac{{x}_i^T{x}_l}{d}\right|^2 \\
\label{eq:G-H-mixed} &\lesssim& n\sum_{i=1}^n \left|\frac{\|x_i\|}{\sqrt{d}}-1\right|^4 + \sum_{i\neq l}\left|\frac{{x}_i^T{x}_l}{d}\right|^4.
\end{eqnarray}
By integrating out the probability tail bound of Lemma \ref{lem:chi-squared}, we have $\mathbb{E}\left|\frac{\|x_i\|}{\sqrt{d}}-1\right|^4\lesssim d^{-2}$. We also have $\mathbb{E}\left|\frac{{x}_i^T{x}_l}{d}\right|^4\lesssim d^{-2}$. Hence, $\sum_{i\neq l}\left(\mathbb{E}\left[\psi(W^T\bar{x}_l)\psi'(W^T\bar{x}_i)W^T(x_i-\bar{x}_i)|X\right]\right)^2\lesssim \frac{n^2}{d^2}$ with high probability. To bound (\ref{eq:second-order-G-2}), we observe that
$$\frac{1}{2}\mathbb{E}\left[\psi(W^T\bar{x}_i)\psi''(t_i)|W^T(x_i-\bar{x}_i)|^2|X\right]\leq \mathbb{E}(|W^T(x_i-\bar{x}_i)|^2|X)=\left|\frac{\|x_i\|}{\sqrt{d}}-1\right|^2,$$
where the inequality above is by $\sup_x|\psi(x)|\leq 1$ and $\sup_x|\psi''(x)|\leq 2$. Since $\mathbb{E}\left|\frac{\|x_i\|}{\sqrt{d}}-1\right|^4\lesssim d^{-2}$, we then have
$$\sum_{i\neq l}\left(\frac{1}{2}\mathbb{E}\left[\psi(W^T\bar{x}_i)\psi''(t_i)|W^T(x_i-\bar{x}_i)|^2|X\right]\right)^2\lesssim \frac{n^2}{d^2},$$
with high probability. With a similar analysis of (\ref{eq:G-tilde-3}), we conclude that the contributions of (\ref{eq:G-tilde-2}) and (\ref{eq:G-tilde-3}) is at most at the order of $\frac{n^2}{d^2}$ with respect to the squared Frobenius norm.

Finally, we show that the contribution of (\ref{eq:G-tilde-4}) is negligible. Note that
\begin{eqnarray*}
&& \left|\mathbb{E}\left((\psi(W^Tx_i)-\psi(W^T\bar{x}_i))(\psi(W^Tx_l)-\psi(W^T\bar{x}_l))|X\right)\right| \\
&\leq& \left|\frac{\|x_i\|}{\sqrt{d}}-1\right|\left|\frac{\|x_l\|}{\sqrt{d}}-1\right|\mathbb{E}\left(|W^T\bar{x}_i||W^T\bar{x}_l||X\right) \\
&\leq& \left|\frac{\|x_i\|}{\sqrt{d}}-1\right|\left|\frac{\|x_l\|}{\sqrt{d}}-1\right|,
\end{eqnarray*}
where the last inequality is by $\mathbb{E}\left(|W^T\bar{x}_i||W^T\bar{x}_l||X\right)\leq \frac{1}{2}\mathbb{E}(|W^T\bar{x}_i|^2+|W^T\bar{x}_l|^2|X)=1$.
Since
$$\sum_{i\neq l}\mathbb{E}\left(\frac{\|x_i\|}{\sqrt{d}}-1\right)^2\mathbb{E}\left(\frac{\|x_l\|}{\sqrt{d}}-1\right)^2\lesssim \frac{n^2}{d^2},$$
we can conclude that (\ref{eq:G-tilde-4}) is bounded by $O\left(\frac{n^2}{d^2}\right)$ with high probability by Markov's inequality.

Combining the analyses of (\ref{eq:G-tilde-1}), (\ref{eq:G-tilde-2}), (\ref{eq:G-tilde-3}) and (\ref{eq:G-tilde-4}), we conclude that $\sum_{i\neq l}(\wt{G}_{il}-\bar{G}_{il})^2\lesssim \frac{n^2}{d^2}$ with high probability. Together with (\ref{eq:G-G-tilde}) and (\ref{eq:G-diag}), we obtain the desired bound for $\opnorm{G-\bar{G}}$.
For the last conclusion, since $\opnorm{G-\bar{G}}$ is sufficiently small, it is sufficient to show $1\lesssim\lambda_{\min}(\bar{G})\leq\lambda_{\max}(\bar{G})\lesssim 1$. The bound $1\lesssim\lambda_{\min}(\bar{G})$ is a direct consequence of the definition of $\bar{G}$. To upper bound $\lambda_{\max}(\bar{G})$, we have
\begin{eqnarray*}
\lambda_{\max}(\bar{G}) &\lesssim& 1 + \max_{\|v\|=1}\sum_{i=1}^n\sum_{l=1}^nv_iv_l\frac{x_i^Tx_l}{\|x_i\|\|x_l\|} \\
&\lesssim& 1 + \max_{\|v\|=1}\sum_{i=1}^n\sum_{l=1}^nv_iv_l\frac{x_i^Tx_l}{d} \\
&=& 1 + \opnorm{X}^2/d \\
&\lesssim& 1 + \frac{n}{d},
\end{eqnarray*}
with high probability, where the last inequality is by \cite{davidson2001local}. The proof is complete.
\end{proof}

\begin{proof}[Proof of Corollary \ref{cor:repair-rf}]
Since $\wh{\theta}$ belongs to the row space of $\wt{X}$, there exists some $u^*\in\mathbb{R}^n$ such that $\wh{\theta}=\wt{X}^Tu^*$.
By Theorem \ref{thm:main-improved} and Lemma \ref{lem:design-rf}, we know that $\wt{u}=u^*$ with high probability, and therefore $\wt{\theta}=\wt{X}^T\wt{u}=\wt{X}^Tu^*=\wh{\theta}$.
\end{proof}

\subsection{Proof of Theorem \ref{thm:nn-grad}}

To prove Theorem \ref{thm:nn-grad}, we need the following kernel random matrix result.
\begin{lemma}\label{lem:lim-H}
Consider independent $W_1,\ldots,W_p\sim N(0,d^{-1}I_d)$, $x_1,\ldots,x_n\sim N(0,I_d)$, and parameters $\beta_1,\ldots,\beta_p\sim N(0,1)$. We define the matrices $H, \bar{H}\in\mathbb{R}^{n\times n}$ by
\begin{eqnarray*}
H_{il} &=& \frac{x_i^Tx_l}{d}\frac{1}{p}\sum_{j=1}^p\beta_j^2\psi'(W^T_jx_i)\psi'(W_j^Tx_l), \\
\bar{H}_{il} &=& |\mathbb{E}\psi'(Z)|^2\frac{x_i^Tx_l}{\|x_i\|\|x_l\|} + \left(\mathbb{E}|\psi'(Z)|^2-|\mathbb{E}\psi'(Z)|^2\right)\mathbb{I}\{i=l\},
\end{eqnarray*}
where $Z\sim N(0,1)$.
Assume $d/\log n$ is sufficiently large, and then
$$\opnorm{H-\bar{H}}^2 \lesssim \frac{n^2}{pd} + \frac{n}{p} + \frac{\log n}{d} + \frac{n^2}{d^2},$$
with high probability. If we further assume that $d/n$ and $p/n$ are sufficiently large, we will also have
\begin{equation}
0.09\leq\lambda_{\min}(H)\leq \lambda_{\max}(H)\lesssim 1, \label{eq:spectrum-H-bound}
\end{equation}
with high probability.
\end{lemma}
\begin{proof}
Define $\wt{H}\in\mathbb{R}^{n\times n}$ with entries $\wt{H}_{il}=\frac{x_i^Tx_l}{d}\mathbb{E}\left(\psi'(W^Tx_i)\psi'(W^Tx_l)\big|X\right)$, and we first bound the difference between $H$ and $\wt{H}$. Note that
$$\mathbb{E}(H_{il}-\wt{H}_{il})^2=\mathbb{E}\Var(H_{il}|X)\leq \frac{1}{p}\mathbb{E}\left(\frac{|x_i^Tx_l|^2}{d^2}\beta^4\right)\leq\begin{cases}
\frac{3}{pd}, & i\neq l, \\
9p^{-1}, & i=l.
\end{cases}$$
We then have
$$\mathbb{E}\opnorm{H-\wt{H}}^2 \leq \mathbb{E}\fnorm{H-\wt{H}}^2 \leq \frac{3n^2}{pd} + \frac{9n}{p}.$$
By Markov's inequality,
\begin{equation}
\opnorm{H-\wt{H}}^2 \lesssim \frac{n^2}{pd} + \frac{n}{p}, \label{eq:H-H-tilde}
\end{equation}
with high probability.

Next, we study the diagonal entries of $\wt{H}$. For any $i\in[n]$,
$$\wt{H}_{ii}=\frac{\|x_i\|^2}{d}\mathbb{E}(|\psi'(W^Tx_i)|^2|X)=\frac{\|x_i\|^2}{d}\mathbb{E}_{U\sim N(0,\|x_i\|^2/d)}|\psi'(U)|^2.$$
Since $\sup_x|\psi'(x)|\leq 1$ and $\sup_x|\psi''(x)|\leq 2$, we have
\begin{eqnarray*}
|\wt{H}_{ii} - \bar{H}_{ii}| &\leq& \left|\frac{\|x_i\|^2}{d}-1\right| + \left|\mathbb{E}_{U\sim N(0,\|x_i\|^2/d)}|\psi'(U)|^2 - \mathbb{E}_{U\sim N(0,1)}|\psi'(U)|^2\right| \\
&\leq& \left|\frac{\|x_i\|^2}{d}-1\right| + 2\TV\left(N(0,\|x_i\|^2/d), N(0,1)\right) \\
&\leq& 4\left|\frac{\|x_i\|^2}{d}-1\right|
\end{eqnarray*}
Similar to (\ref{eq:G-diag}), Lemma \ref{lem:chi-squared} and a union bound argument imply
\begin{equation}
\max_{1\leq i\leq n}|\wt{H}_{ii}-\bar{H}_{ii}|\lesssim \sqrt{\frac{\log n}{d}}, \label{eq:H-diag}
\end{equation}
with high probability.

Now we analyze the off-diagonal entries. Recall the notation $\bar{x}_i=\frac{\sqrt{d}}{\|x_i\|}x_i$. For any $i\neq l$, we have
\begin{eqnarray}
\label{eq:H-tilde-1} \wt{H}_{il} &=& \frac{\bar{x}_i^T\bar{x}_l}{d}\mathbb{E}\left(\psi'(W^T\bar{x}_i)\psi'(W^T\bar{x}_l)\big|X\right) \\
\label{eq:H-tilde-2} && + \frac{x_i^Tx_l}{d}\mathbb{E}\left(\psi'(W^T{x}_i)\psi'(W^T{x}_l)-\psi'(W^T\bar{x}_i)\psi'(W^T\bar{x}_l)\big|X\right) \\
\label{eq:H-tilde-3} && + \left(\frac{\|x_i\|\|x_l\|}{d}-1\right)\frac{\bar{x}_i^T\bar{x}_l}{d}\mathbb{E}\left(\psi'(W^T\bar{x}_i)\psi'(W^T\bar{x}_l)\big|X\right).
\end{eqnarray}
For the first term on the right hand side of (\ref{eq:H-tilde-1}), we observe that $\frac{\bar{x}_i^T\bar{x}_l}{d}\mathbb{E}\left(\psi'(W^T\bar{x}_i)\psi'(W^T\bar{x}_l)\big|X\right)$ is a function of $\frac{\bar{x}_i^T\bar{x}_l}{d}$, and thus we can write
$$\frac{\bar{x}_i^T\bar{x}_l}{d}\mathbb{E}\left(\psi'(W^T\bar{x}_i)\psi'(W^T\bar{x}_l)\big|X\right)=f\left(\frac{\bar{x}_i^T\bar{x}_l}{d}\right),$$
where
$$f(\rho)=\begin{cases}
\rho\mathbb{E}\psi'(\sqrt{1-\rho}U+\sqrt{\rho}Z)\psi'(\sqrt{1-\rho}V+\sqrt{\rho}Z), & \rho \geq 0, \\
\rho\mathbb{E}\psi'(\sqrt{1+\rho}U-\sqrt{-\rho}Z)\psi'(\sqrt{1+\rho}V+\sqrt{-\rho}Z), & \rho < 0,
\end{cases}$$
with $U,V,Z\stackrel{iid}{\sim} N(0,1)$. By some direct calculations, we have $f(0)=0$, $f'(0)=(\mathbb{E}\psi'(Z))^2$, and $\sup_{|\rho|\leq 0.2}|f''(\rho)|\lesssim 1$. Therefore, using the same analysis that leads to the bound for (\ref{eq:G-H}), we have
$$\sum_{i\neq l}\left(\frac{\bar{x}_i^T\bar{x}_l}{d}\mathbb{E}\left(\psi'(W^T\bar{x}_i)\psi'(W^T\bar{x}_l)\big|X\right)-\bar{H}_{il}\right)^2 \lesssim \sum_{i\neq l}\left|\frac{\bar{x}_i^T\bar{x}_l}{d}\right|^4\lesssim \frac{n^2}{d^2},$$
with high probability.

For (\ref{eq:H-tilde-2}), we note that
\begin{align*}
 \mathbb{E}\Bigl(\psi'(W^T{x}_i)\psi'(W^T{x}_l)&-\psi'(W^T\bar{x}_i)\psi'(W^T\bar{x}_l)\big|X\Bigr) \\
&\leq \mathbb{E}\left(|\psi'(W^T{x}_i)-\psi'(W^T\bar{x}_i)|\big|X\right) + \mathbb{E}\left(|\psi'(W^T{x}_l)-\psi'(W^T\bar{x}_l)|\big|X\right) \\
&\leq 2\mathbb{E}\left(|W^T(x_i-\bar{x}_i)|\big|X\right) + 2\mathbb{E}\left(|W^T(x_l-\bar{x}_l)|\big|X\right) \\
&= 2\left|\frac{\|x_i\|}{\sqrt{d}}-1\right| + 2\left|\frac{\|x_l\|}{\sqrt{d}}-1\right|,
\end{align*}
where we have used $\sup_x|\psi'(x)|\leq 1$ and $\sup_x|\psi''(x)|\leq 2$ in the above inequalities. Therefore, the contribution of (\ref{eq:H-tilde-2}) in terms of squared Frobenius norm is bounded by
\begin{eqnarray*}
&& \sum_{i\neq l}\left|\frac{x_i^Tx_l}{d}\right|^2\left(2\left|\frac{\|x_i\|}{\sqrt{d}}-1\right| + 2\left|\frac{\|x_l\|}{\sqrt{d}}-1\right|\right)^2 \\
&\lesssim& \sum_{i\neq l}\left|\frac{x_i^Tx_l}{d}\right|^4 + n\sum_{i=1}^n\left|\frac{\|x_l\|}{\sqrt{d}}-1\right|^4 \\
&\lesssim& \frac{n^2}{d^2},
\end{eqnarray*}
with high probability, and the last inequality above uses the same analysis that bounds (\ref{eq:G-H-mixed}).

Finally, since (\ref{eq:H-tilde-3}) can be bounded by $\left|\frac{\|x_i\|\|x_l\|}{d}-1\right|\left|\frac{\bar{x}_i^T\bar{x}_l}{d}\right|$, its contribution in terms of squared Frobenius norm is bounded by
\begin{eqnarray*}
&& \sum_{i\neq l}\left|\frac{\|x_i\|\|x_l\|}{d}-1\right|^2\left|\frac{\bar{x}_i^T\bar{x}_l}{d}\right|^2 \\
&\lesssim& \sum_{i\neq l}\left|\frac{\|x_i\|\|x_l\|}{d}-1\right|^4 + \sum_{i\neq l}\left|\frac{\bar{x}_i^T\bar{x}_l}{d}\right|^4.
\end{eqnarray*}
We have already shown that $\sum_{i\neq l}\left|\frac{\bar{x}_i^T\bar{x}_l}{d}\right|^4\lesssim \frac{n^2}{d^2}$ in the analysis of (\ref{eq:G-H}). For the first term on the right hand side of the above inequality, we use Lemma \ref{lem:inner-prod} and obtain a probability tail bound for $|\|x_i\|\|x_l\|-d|$. By integrating out this tail bound, we have
$$\sum_{i\neq l}\mathbb{E}\left(\frac{\|x_i\|\|x_l\|}{d}-1\right)^4\lesssim \frac{n^2}{d^2},$$
which, by Markov's inequality, implies $\sum_{i\neq l}\left(\frac{\|x_i\|\|x_l\|}{d}-1\right)^4\lesssim \frac{n^2}{d^2}$ with high probability.

Combining the analyses of (\ref{eq:H-tilde-1}), (\ref{eq:H-tilde-2}), and (\ref{eq:H-tilde-3}), we conclude that $\sum_{i\neq l}(\wt{H}_{il}-\bar{H}_{il})^2\lesssim \frac{n^2}{d^2}$ with high probability. Together with (\ref{eq:H-H-tilde}) and (\ref{eq:H-diag}), we obtain the desired bound for $\opnorm{H-\bar{H}}$.
The last conclusion (\ref{eq:spectrum-H-bound}) follows a similar argument used in the proof of Lemma \ref{lem:lim-G}.
\end{proof}

Now we are ready to prove Theorem \ref{thm:nn-grad}.
\begin{proof}[Proof of Theorem \ref{thm:nn-grad}]
We first establish some high probability events:
\begin{eqnarray}
\label{eq:r1e1} \max_{1\leq j\leq p}|\beta_j(0)| &\leq& 2\sqrt{\log p}, \\
\label{eq:r1e2} \max_{k\in\{1,2,3\}}\frac{1}{p}\sum_{j=1}^p|\beta_j(0)|^k &\lesssim& 1, \\
\label{eq:r1e3} \sum_{i=1}^n\|x_i\|^2 &\leq& 2nd, \\
\label{eq:r1e4} \max_{1\leq i\leq n}\|x_i\| &\lesssim& \sqrt{d}, \\
\label{eq:r1e5} \max_{1\leq i\neq l\leq n}\left|\frac{x_i^Tx_l}{d}\right| &\lesssim& d^{-1/2}, \\
\label{eq:r1e6}\max_{1\leq l\leq n}\sum_{i=1}^n\left|\frac{x_i^Tx_l}{d}\right| &\lesssim& 1+\frac{n}{\sqrt{d}}, \\
 \label{eq:r1e7}\|u(0)\| &\leq& \sqrt{n}(\log p)^{1/4}, \\
 \label{eq:r1e8}\max_{1\leq j\leq p}\sum_{i=1}^n|W_j(0)^Tx_i|^2 &\leq& 6n+18\log p, \\
 \label{eq:r1e9}\max_{1\leq i\leq n}\frac{1}{p}\sum_{j=1}^p\mathbb{I}\{|W_j(0)^Tx_i|\leq R_1\|x_i\|\} &\lesssim& \sqrt{d}R_1 + \sqrt{\frac{\log n}{p}}.
\end{eqnarray}
The bound (\ref{eq:r1e1}) is a consequence of a standard Gaussian tail inequality and a union bound argument. The second bound (\ref{eq:r1e2}) is by Markov's inequality and the fact that $\frac{1}{p}\sum_{j=1}^p\mathbb{E}|\beta_j(0)|^k\lesssim 1$. Then, we have (\ref{eq:r1e3}), (\ref{eq:r1e4}) and (\ref{eq:r1e8}) derived from Lemma \ref{lem:chi-squared} and a union bound. Similarly, (\ref{eq:r1e5}) is by Lemma \ref{lem:inner-prod} and a union bound. The bound (\ref{eq:r1e6}) is a direct consequence of (\ref{eq:r1e4}) and (\ref{eq:r1e5}). To obtain (\ref{eq:r1e7}), we note that $\mathbb{E}|u_i(0)|^2 = \mathbb{E}\Var(u_i(0)|X) \lesssim 1$, which then implies (\ref{eq:r1e7}) by Markov's inequality. Finally, for (\ref{eq:r1e9}), we have
\begin{eqnarray*}
&&\max_{1\leq i\leq n}\frac{1}{p}\sum_{j=1}^p\mathbb{I}\{|W_j(0)^Tx_i|\leq R_1\|x_i\|\} \leq \mathbb{P}\left(|N(0,1)|\leq \sqrt{d}R_1\right) \\
&&+ \max_{1\leq i\leq n}\frac{1}{p}\sum_{j=1}^p\left(\mathbb{I}\{|W_j(0)^Tx_i|\leq R_1\|x_i\|\}-\mathbb{P}\left(|N(0,1)|\leq \sqrt{d}R_1\right)\right),
\end{eqnarray*}
where the first term $\mathbb{P}\left(|N(0,1)|\leq \sqrt{d}R_1\right)$ can be bounded by $O(\sqrt{d}R_1)$, and the second term can be bounded by $\sqrt{\frac{\log n}{p}}$ according to Lemma \ref{lem:hoeffding} and a union bound.

Now we are ready to prove the main result. We introduce the function
$$v_i(t)=\frac{1}{\sqrt{p}}\sum_{j=1}^p\beta_j(t)\psi(W_j(t-1)^Tx_i).$$
Besides (\ref{eq:iter-parameter}), (\ref{eq:iter-parameter-beta}) and (\ref{eq:iter-function}), we will also establish
\begin{equation}
\|y-v(t)\|^2 \leq \left(1-\frac{\gamma}{8}\right)^t\|y-v(0)\|^2. \label{eq:v-seq}
\end{equation}
It suffices to show the following to claims are true.
\begin{thm1}
With high probability, for any integer $k\geq 1$, as long as (\ref{eq:v-seq}), (\ref{eq:iter-function}), (\ref{eq:iter-parameter}) and (\ref{eq:iter-parameter-beta}) hold for all $t\leq k$, then (\ref{eq:iter-parameter-beta}) holds for $t=k+1$.
\end{thm1}
\begin{thm2}
With high probability, for any integer $k\geq 1$, as long as (\ref{eq:v-seq}), (\ref{eq:iter-function}) and (\ref{eq:iter-parameter}) hold for all $t\leq k$, and (\ref{eq:iter-parameter-beta}) holds for all $t\leq k+1$, then (\ref{eq:v-seq}) holds for $t=k+1$.
\end{thm2}
\begin{thm3}
With high probability, for any integer $k\geq 1$, as long as (\ref{eq:iter-function}) and (\ref{eq:iter-parameter}) hold for all $t\leq k$, and (\ref{eq:iter-parameter-beta}) and (\ref{eq:v-seq}) hold for all $t\leq k+1$, then (\ref{eq:iter-parameter}) holds for $t=k+1$.
\end{thm3}
\begin{thm4}
With high probability, for any integer $k\geq 1$, as long as (\ref{eq:iter-function}) holds for all $t\leq k$, and (\ref{eq:v-seq}), (\ref{eq:iter-parameter}) and (\ref{eq:iter-parameter-beta}) hold for all $t\leq k+1$, then (\ref{eq:iter-function}) holds for $t=k+1$.
\end{thm4}
\noindent With all the claims above being true, we can then deduce (\ref{eq:iter-parameter}), (\ref{eq:iter-parameter-beta}), (\ref{eq:iter-function}) and (\ref{eq:v-seq}) for all $t\geq 1$ by mathematical induction.

\paragraph{Proof of Claim A.}
By triangle inequality and the gradient formula,
\begin{eqnarray*}
|\beta_j(k+1)-\beta_j(0)| &\leq& \sum_{t=0}^k|\beta_j(t+1)-\beta_j(t)| \\
&\leq& \frac{\gamma}{\sqrt{p}}\sum_{t=0}^k\left|\sum_{i=1}^n(u_i(t)-y_i)\psi(W_j(t)^Tx_i)\right| \\
&\leq& \frac{\gamma}{\sqrt{p}}\sum_{t=0}^k\sum_{i=1}^n|y_i-u_i(t)||W_j(t)^Tx_i| \\
&\leq& \frac{\gamma}{\sqrt{p}}\sum_{t=0}^k\|y-u(t)\|\sqrt{\sum_{i=1}^n|W_j(t)^Tx_i|^2} \\
&\leq& \frac{\gamma}{\sqrt{p}}\sum_{t=0}^k\|y-u(t)\|\left(R_1\sqrt{\sum_{i=1}^n\|x_i\|^2}+\sqrt{\sum_{i=1}^n|W_j(0)^Tx_i|^2}\right) \\
&\leq& \gamma\sqrt{\frac{7n+18\log p}{p}}\sum_{t=0}^k\|y-u(t)\| \\
&\leq& 16\sqrt{\frac{7n+18\log p}{p}}\|y-u(0)\| \\
&\leq& 32\sqrt{\frac{n^2\log p}{p}} = R_2,
\end{eqnarray*}
where we have used (\ref{eq:r1e4}), (\ref{eq:r1e7}) and (\ref{eq:r1e8}).
Hence, (\ref{eq:iter-parameter-beta}) holds for $t=k+1$, and Claim A is true.

\paragraph{Proof of Claim B.} We omit this step, because the analysis uses the same argument as that of the proof of Claim D.

\paragraph{Proof of Claim C.} We bound $\|W_j(k+1)-W_j(0)\|$ by $\sum_{t=0}^k\|W_j(t+1)-W_j(t)\|$. Then by the gradient descent formula, we have
\begin{eqnarray*}
\|W_j(k+1)-W_j(0)\| &\leq& \frac{\gamma}{d\sqrt{p}}\sum_{t=0}^k\left\|\beta_j(t+1)\sum_{i=1}^n(v_i(t+1)-y_i)\psi'(W_j(t)^Tx_i)x_i\right\| \\
&\leq& \frac{\gamma}{d\sqrt{p}}\sum_{t=0}^k|\beta_j(t+1)|\sum_{i=1}^n|y_i-v_i(t+1)|\|x_i\| \\
&\leq& \frac{\gamma}{d\sqrt{p}}(|\beta_j(0)|+R_2)\sqrt{\sum_{i=1}^n\|x_i\|^2}\sum_{t=0}^k\|y-v(t+1)\| \\
&\leq& \frac{16}{d\sqrt{p}}(|\beta_j(0)|+R_2)\sqrt{\sum_{i=1}^n\|x_i\|^2}\|y-v(0)\| \\
&\leq& \frac{100n\log p}{\sqrt{pd}} = R_1,
\end{eqnarray*}
where we have used (\ref{eq:r1e1}), (\ref{eq:r1e3}) and (\ref{eq:r1e7}) in the above inequalities. Thus, Claim C is true.

\paragraph{Proof of Claim D.}
We first analyze $u(k+1)-u(k)$. For each $i\in[n]$, we have
\begin{eqnarray*}
&& u_i(k+1) - u_i(k) \\
&=& \frac{1}{\sqrt{p}}\sum_{j=1}^p\beta_j(k+1)\left(\psi(W_j(k+1)^Tx_i)-\psi(W_j(k)^Tx_i)\right) \\
&& + \frac{1}{\sqrt{p}}\sum_{j=1}^p(\beta_j(k+1)-\beta_j(k))\psi(W_j(k)^Tx_i) \\
&=& \frac{1}{\sqrt{p}}\sum_{j=1}^p\beta_j(k+1)(W_j(k+1)-W_j(k))^Tx_i\psi'(W_j(k)^Tx_i) \\
&& + \frac{1}{\sqrt{p}}\sum_{j=1}^p(\beta_j(k+1)-\beta_j(k))\psi(W_j(k)^Tx_i) + r_i(k) \\
&=& \gamma\sum_{l=1}^n(H_{il}(k)+G_{il}(k))(y_l-u_l(k)) + r_i(k),
\end{eqnarray*}
where
\begin{eqnarray*}
G_{il}(k) &=& \frac{1}{p}\sum_{j=1}^p\psi(W_j(k)^Tx_l)\psi(W_j(k)^Tx_i), \\
H_{il}(k) &=& \frac{x_i^Tx_l}{d}\frac{1}{p}\sum_{j=1}^p\beta_j(k+1)^2\psi'(W_j(k)^Tx_i)\psi'(W_j(k)^Tx_l),
\end{eqnarray*}
and
$$r_i(k)=\frac{1}{2\sqrt{p}}\sum_{j=1}^p\beta_j(k+1)|(W_j(k+1)-W_j(k))^Tx_i|^2\psi''(\xi_{ijk}).$$
Note that $\xi_{ijk}$ is some random variable whose value is between $W_j(k)^Tx_i$ and $W_j(k+1)^Tx_i$.
The above iteration formula can be summarized in a vector form as
\begin{equation}
u(k+1)-u(k)=\gamma(H(k)+G(k))(y-u(k))+r(k). \label{eq:iter-u-arctan}
\end{equation}
We need to understand the eigenvalues of $G(k)$ and $H(k)$, and bound the absolute value of $r_i(k)$.

To analyze $G(k)$, we first control the difference between $G(k)$ and $G(0)$. Since
\begin{eqnarray*}
|G_{il}(k) - G_{il}(0)| &\leq& \frac{1}{p}\sum_{j=1}^p|\psi(W_j(k)^Tx_l) - \psi(W_j(0)^Tx_l)| \\
&& + \frac{1}{p}\sum_{j=1}^p|\psi(W_j(k)^Tx_i) - \psi(W_j(0)^Tx_i)| \\
&\leq& \frac{1}{p}\sum_{j=1}^p|(W_j(k)-W_j(0))^Tx_l| + \frac{1}{p}\sum_{j=1}^p|(W_j(k)-W_j(0))^Tx_i| \\
&\leq& R_1\left(\|x_l\| + \|x_i\|\right),
\end{eqnarray*}
then, by (\ref{eq:r1e4}),
\begin{equation}
\opnorm{G(k)-G(0)}\leq\max_{1\leq l\leq n}\sum_{i=1}^n|G_{il}(k) - G_{il}(0)|\leq 2R_1n\max_{1\leq i\leq n}\|x_i\|\lesssim \frac{n^2\log p}{\sqrt{p}}. \label{eq:x-japan}
\end{equation}
By Lemma \ref{lem:lim-G}, we have
\begin{equation}
0 \leq \lambda_{\min}(G(k)) \leq \lambda_{\max}(G(k)) \lesssim 1+\frac{n^2\log p}{\sqrt{p}}. \label{eq:Gk-spec-arctan}
\end{equation}
For the matrix $H(k)$, we show its eigenvalues can be controlled by those of $H(0)$. We have
\begin{eqnarray*}
|H_{il}(k)-H_{il}(0)| &\leq& \left|\frac{x_i^Tx_l}{d}\right|\frac{1}{p}\sum_{j=1}^p|\beta_j(k+1)^2-\beta_j^2(0)| \\
&& + \left|\frac{x_i^Tx_l}{d}\right|\frac{1}{p}\sum_{j=1}^p\beta_j^2(0)|\psi'(W_j(k)^Tx_i) - \psi'(W_j(0)^Tx_i)| \\
&& + \left|\frac{x_i^Tx_l}{d}\right|\frac{1}{p}\sum_{j=1}^p\beta_j^2(0)|\psi'(W_j(k)^Tx_l) - \psi'(W_j(0)^Tx_l)| \\
&\leq& \left|\frac{x_i^Tx_l}{d}\right|\frac{1}{p}\sum_{j=1}^pR_2(R_2+2|\beta_j(0)|) \\
&& + 2R_1\left(\|x_l\| + \|x_i\|\right)\left|\frac{x_i^Tx_l}{d}\right|\frac{1}{p}\sum_{j=1}^p\beta_j^2(0).
\end{eqnarray*}
Thus, by (\ref{eq:r1e2}) and (\ref{eq:r1e6}),
$$\max_{1\leq l\leq n}\sum_{i=1}^n|H_{il}(k) - H_{il}(0)|\lesssim \max_{1\leq l\leq n}\sum_{i=1}^n(R_2+R_1\sqrt{d})\left|\frac{x_i^Tx_l}{d}\right|\lesssim \frac{n\log p}{\sqrt{p}}\left(1+\frac{n}{\sqrt{d}}\right).$$
Then, we have
$$
\opnorm{H(k)-H(0)}\leq \max_{1\leq l\leq n}\sum_{i=1}^n|H_{il}(k) - H_{il}(0)|\lesssim \frac{n\log p}{\sqrt{p}}\left(1+\frac{n}{\sqrt{d}}\right).
$$
Together with Lemma \ref{lem:lim-H}, we obtain
\begin{equation}
0.089 \leq \lambda_{\min}(H(k)) \leq \lambda_{\max}(H(k)) \lesssim 1. \label{eq:H-time-stable}
\end{equation}
Next, we give a bound for $r_i(k)$. By $\sup_x|\psi''(x)|\leq 2$ and $\sup_x|\psi'(x)|\leq 1$, we have
\begin{eqnarray*}
|r_i(k)| &\leq& \frac{1}{\sqrt{p}}\sum_{j=1}^p|\beta_j(k+1)||(W_j(k+1)-W_j(k))^Tx_i|^2 \\
&\leq& \frac{\|x_i\|^2}{\sqrt{p}}\sum_{j=1}^p|\beta_j(k+1)|\|W_j(k+1)-W_j(k)\|^2 \\
&\leq& \frac{\gamma^2}{pd^2}\frac{\|x_i\|^2}{\sqrt{p}}\sum_{j=1}^p|\beta_j(k+1)||\beta_j(k)|^2\left(\sum_{l=1}^n|y_l-u_l(k)|\|x_l\|\right)^2 \\
&\leq& \frac{\gamma^2}{pd^2}\frac{\|x_i\|^2\sum_{l=1}^n\|x_l\|^2}{\sqrt{p}}\|y-u(k)\|^2\sum_{j=1}^p|\beta_j(k+1)||\beta_j(k)|^2 \\
&\lesssim& \frac{\gamma^2n}{\sqrt{p}}\|y-u(k)\|^2 \\
&\lesssim& \frac{\gamma^2n\sqrt{n\log p}}{\sqrt{p}}\|y-u(k)\|,
\end{eqnarray*}
where we have used (\ref{eq:r1e2}), (\ref{eq:r1e3}), (\ref{eq:r1e4}) and (\ref{eq:r1e7}) in the above inequalities.
This leads to the bound
\begin{equation}
\|r(k)\|=\sqrt{\sum_{i=1}^n|r_i(k)|^2}\lesssim \frac{\gamma^2n^2\sqrt{\log p}}{\sqrt{p}}\|y-u(k)\|.\label{eq:bound-res-k}
\end{equation}
By (\ref{eq:iter-u-arctan}), we have
\begin{eqnarray*}
\|y-u(k+1)\|^2
&=& \|y-u(k)\|^2 - 2\gamma(y-u(k))^T(H(k)+G(k))(y-u(k)) \\
&& - 2\iprod{y-u(k)}{r(k)} + \|u(k)-u(k+1)\|^2.
\end{eqnarray*}
The bounds (\ref{eq:Gk-spec-arctan}) and (\ref{eq:H-time-stable}) imply
\begin{equation}
- 2\gamma(y-u(k))^T(H(k)+G(k))(y-u(k)) \leq -\frac{\gamma}{6}\|y-u(k)\|^2. \label{eq:main-inner}
\end{equation}
The bound (\ref{eq:bound-res-k}) implies
$$- 2\iprod{y-u(k)}{r(k)}\leq 2\|y-u(k)\|\|r(k)\|\lesssim \frac{\gamma^2n^2\sqrt{\log p}}{\sqrt{p}}\|y-u(k)\|^2.$$
Using (\ref{eq:Gk-spec-arctan}), (\ref{eq:H-time-stable}) and (\ref{eq:bound-res-k}), we have
\begin{eqnarray*}
\|u(k)-u(k+1)\|^2 &\leq& 2\gamma^2\|(H(k)+G(k))(y-u(k))\|^2 + 2\|r(k)\|^2 \\
&\lesssim& \gamma^2\left(1+\frac{n^4(\log p)^2}{p}\right)\|y-u(k)\|^2 + \frac{\gamma^4n^4\log p}{p}\|y-u(k)\|^2 .
\end{eqnarray*}
Therefore, as long as $\gamma\frac{n^4(\log p)^2}{p}$ is sufficiently small, we have
$$- 2\iprod{y-u(k)}{r(k)} + \|u(k)-u(k+1)\|^2 \leq \frac{\gamma}{24}\|y-u(k)\|^2.$$
Together with the bound (\ref{eq:main-inner}), we have
$$\|y-u(k+1)\|^2 \leq \left(1-\frac{\gamma}{8}\right)\|y-u(k)\|^2\leq \left(1-\frac{\gamma}{8}\right)^{k+1}\|y-u(0)\|^2,$$
and thus Claim D is true. The proof is complete.
\end{proof}

% !TEX root = ./repair.tex

\section{Results with ReLU activation}

\subsection{Repair of random feature model and neural nets} \label{app:results}

In this section, we present analogous results of Sections \ref{sec:linear} and \ref{sec:neural} with ReLU activation. First, consider the random feature model with design $\wt{X}=\psi(XW)=\{\psi(W_j^Tx_i)\}_{i\in[n],j\in[p]}$, where $\psi(t)=\max(0,t)$. Recall that $x_i\sim N(0,I_d)$ and $W_j\sim N(0,d^{-1}I_d)$ independently for all $i\in[n]$ and $j\in[p]$. The random matrix $\wt{X}$ has good properties, which is given by the following lemma.
\begin{lemma}\label{lem:design-rf-relu}
Assume $n/p^2$ and $n\log n/d$ are sufficiently small. Then, Condition $A$ and Condition $B$ hold for $A=\wt{X}^T$, $m=p$ and $k=n$ with some $\sigma^2\asymp p$, $\overline{\lambda}^2\asymp n$ and $\underline{\lambda}\asymp 1$.
\end{lemma}

Now consider a model $\wh{\theta}$ that lies in the row space of $\wt{X}$. For example, $\wh{\theta}$ can be computed from a gradient-based algorithm initialized at $0$. We observe a contaminated version $\eta=\wh{\theta}+z$.
We can then compute the procedure $\wt{u}=\argmin_{u\in\mathbb{R}^n}\|\eta-\wt{X}^Tu\|_1$ and use $\wt{\theta}=\wt{X}^T\wt{u}$ for model repair.

\begin{corollary}\label{cor:repair-rf-relu}
Assume $\epsilon\sqrt{n}$, $n/p^2$ and $n\log n/d$ are sufficiently small. We then have $\wt{\theta}=\wh{\theta}$ with high probability.
\end{corollary}

Next, we study the repair of neural network $f(x)=\frac{1}{\sqrt{p}}\sum_{j=1}^p\beta_j\psi(W_j^Tx)$ with ReLU activation $\psi(t)=\max(0,t)$. The gradient descent algorithm (Algorithm \ref{alg:GD}) enjoys the following property. Recall the notation that $u_i(t)=\frac{1}{\sqrt{p}}\sum_{j=1}^p\beta_j(t)\psi(W_j(t)^Tx_i)$. We assume $x_i$ is i.i.d. $N(0,I_d)$ and $|y_i|\leq 1$ for all $i\in[n]$.
\begin{thm}\label{thm:nn-grad-relu}
Assume $\frac{n\log n}{d}$, $\frac{n^3(\log p)^4}{p}$ and $\gamma n$ are all sufficiently small. Then we have
\begin{eqnarray}
\label{eq:iter-parameter-relu} \max_{1\leq j\leq p}\|W_j(t)-W_j(0)\| &\leq& R_1, \\
\label{eq:iter-parameter-beta-relu} \max_{1\leq j\leq p}|\beta_j(t)-\beta_j(0)| &\leq& R_2,
\end{eqnarray}
and
\begin{equation}
\|y-u(t)\|^2 \leq \left(1-\frac{\gamma}{8}\right)^t\|y-u(0)\|^2, \label{eq:iter-function-relu}
\end{equation}
for all $t\geq 1$ with high probability, where $R_1=\frac{100n\log p}{\sqrt{pd}}$ and $R_2=32\sqrt{\frac{n^2\log p}{p}}$.
\end{thm}

Consider the contaminated model $\eta=\wh{\beta}+z$ and $\Theta_j=\wh{W}_j+Z_j$, where each entry of $z$ and $Z_j$ is zero with probability $1-\epsilon$ and follows an arbitrary distribution with the complementary probability $\epsilon$. We apply Algorithm \ref{alg:MR} to repair the neural net model. We study two situations. In the first situation, $\wh{\beta}=\beta(t_{\max})$ and $\wh{W}=W(t_{\max})$ are the direct output of Algorithm \ref{alg:GD}.

\begin{thm}\label{thm:repair-nn-1-relu}
Under the conditions of Theorem \ref{thm:nn-grad-relu}, additionally assume that $\frac{\log p}{d}$ and $\epsilon\sqrt{n}$ are sufficiently small. We then have $\wt{W}=\wh{W}$ and $\frac{1}{p}\|\wt{\beta}-\wh{\beta}\|^2 \lesssim \frac{n^3\log p}{p}$ with high probability.
\end{thm}

In the second situation, we have $\wh{W}=W(t_{\max})$ and then $\wh{\beta}$ is obtained by carrying out gradient descent over $\beta$ using features $\wt{X}=\psi(X\wh{W})$. Since the gradient descent over $\beta$ is initialized at $0$, we shall replace the $\beta(0)$ by $0$ in Algorithm \ref{alg:MR} as well.

\begin{thm}\label{thm:repair-nn-2-relu}
Under the conditions of Theorem \ref{thm:nn-grad-relu}, additionally assume that $\frac{\log p}{d}$ and $\epsilon\sqrt{n}$ are sufficiently small. We then have $\wt{W}=\wh{W}$ and $\wt{\beta}=\wh{\beta}$ with high probability.
\end{thm}

\begin{remark}
When $\epsilon\sqrt{n}$ is sufficiently small, the conditions of Theorem \ref{thm:repair-nn-1-relu} and Theorem \ref{thm:repair-nn-2-relu} can be simplified as $p \gg n^3$ and $d\gg n$ by ignoring the logarithmic factors. The more stringent requirement on $\epsilon$ is due to the fact that the design matrix $\psi(XW)$ does not have approximate zero mean with the ReLU activation. This results in a large $\sigma^2$ in Condition $A$. In contrast, the hyperbolic tangent activation is an odd function, a property of symmetry that leads to Condition $A$ with a constant $\sigma^2$.
\end{remark}

\subsection{Proofs of Lemma \ref{lem:design-rf-relu} and Corollary \ref{cor:repair-rf-relu}}

We first state the proof of Lemma \ref{lem:design-rf-relu}.
The conclusion of Condition $A$ is obvious by
$$\sum_{i=1}^n\mathbb{E}\left(\frac{1}{p}\sum_{j=1}^pc_j\psi(W_j^Tx_i)\right)^2\leq \sum_{i=1}^n\frac{1}{p}\sum_{j=1}^p\mathbb{E}|W_j^Tx_i|^2= n,$$
and Markov's inequality. To check Condition $B$, we prove (\ref{eq:l1-upper-A}) and (\ref{eq:l2-upper-A}) separately.
\begin{proof}[Proof of (\ref{eq:l1-upper-A}) of Lemma \ref{lem:design-rf-relu}]
We adopt a similar strategy to the proof of Lemma \ref{lem:design-rf}. Define
$$f(W,X,\Delta)=\frac{1}{p}\sum_{j=1}^p\left|\sum_{i=1}^n\psi(W_j^Tx_i)\Delta_i\right|,$$
and $g(X,\Delta)=\mathbb{E}(f(W,X,\Delta)|X)$.
We then have
\begin{eqnarray}
\nonumber \inf_{\|\Delta\|=1}f(W,X,\Delta) &\geq& \inf_{\|\Delta\|=1}\mathbb{E}f(W,X,\Delta) - \sup_{\|\Delta\|=1}\left|f(W,X,\Delta)-\mathbb{E}f(W,X,\Delta)\right| \\
\label{eq:exp-f-inf} &\geq& \inf_{\|\Delta\|=1}\mathbb{E}f(W,X,\Delta) \\
\label{eq:ep-f} && - \sup_{\|\Delta\|=1}\left|f(W,X,\Delta)-\mathbb{E}(f(W,X,\Delta)|X)\right| \\
\label{eq:ep-g} && - \sup_{\|\Delta\|=1}\left|g(X,\Delta)-\mathbb{\mathbb{E}}g(X,\Delta)\right|.
\end{eqnarray}
We will analyze the three terms above separately.

\paragraph{Analysis of (\ref{eq:exp-f-inf}).} Define $h(W_j)=\mathbb{E}(\psi(W_j^Tx_i)|W_j)$ and $\bar{\psi}(W_j^Tx_i)=\psi(W_j^Tx_i)-h(W_j)$. We then have
\begin{equation}
\mathbb{E}f(W,X,\Delta)=\mathbb{E}\left|\sum_{i=1}^n\bar{\psi}(W^Tx_i)\Delta_i+h(W)\sum_{i=1}^n\Delta_i\right|. \label{eq:will-be-split}
\end{equation}
A lower bound of (\ref{eq:will-be-split}) is
$$\mathbb{E}f(W,X,\Delta)\geq \left|\sum_{i=1}^n\Delta_i\right|\left|\mathbb{E}h(W)\right|-\mathbb{E}\left|\sum_{i=1}^n\bar{\psi}(W^Tx_i)\Delta_i\right|,$$
where the second term can be bounded by
\begin{eqnarray*}
\mathbb{E}\left|\sum_{i=1}^n\bar{\psi}(W^Tx_i)\Delta_i\right| &\leq& \sqrt{\mathbb{E}\left|\sum_{i=1}^n\bar{\psi}(W^Tx_i)\Delta_i\right|^2} \\
&=& \sqrt{\mathbb{E}\Var\left(\left|\sum_{i=1}^n\psi(W^Tx_i)\Delta_i\right|\Big|W\right)} \\
&=& \sqrt{\mathbb{E}\sum_{i=1}^n\Delta_i^2\Var(\psi(W^Tx_i)|W)} \\
&=& \sqrt{\mathbb{E}\sum_{i=1}^n\Delta_i^2\mathbb{E}(|\psi(W^Tx_i)|^2|W)} \\
&=& \sqrt{\mathbb{E}|\psi(W^Tx)|^2} \leq \sqrt{\mathbb{E}|W^Tx|^2} = 1.
\end{eqnarray*}
Since
$$\mathbb{E}h(W)=\frac{1}{\sqrt{2\pi}}\mathbb{E}\|W\|=\frac{1}{\sqrt{\pi}}\frac{\Gamma((d+1)/2)}{\sqrt{d}\Gamma(d/2)}\geq \frac{1}{\sqrt{2\pi}}\sqrt{\frac{d-1}{d}}.$$
Therefore, as long as $d\geq 3$ and $\left|\sum_{i=1}^n\Delta_i\right|\geq 7$, we have $\mathbb{E}f(W,X,\Delta)\geq 1$, and we thus can conclude that
\begin{equation}
\inf_{\|\Delta\|=1,|\sum_{i=1}^n\Delta_i|\geq 7}\mathbb{E}f(W,X,\Delta) \gtrsim 1.\label{eq:l1-1-1}
\end{equation}

Now we consider the case $\left|\sum_{i=1}^n\Delta_i\right|< 7$. A lower bound for $\left|\sum_{i=1}^n\psi(W^Tx_i)\Delta_i\right|$ is
\begin{equation}
\left|\sum_{i=1}^n\psi(W^Tx_i)\Delta_i\right| \geq \left|\sum_{i=1}^n\bar{\psi}(W^Tx_i)\Delta_i\right| - 7h(W) = \left|\sum_{i=1}^n\bar{\psi}(W^Tx_i)\Delta_i\right| - \frac{7}{\sqrt{2\pi}}\|W\|. \label{eq:seven}
\end{equation}
Thus,
\begin{eqnarray*}
 \mathbb{E}f(W,X,\Delta) &\geq& \mathbb{E}\left(\left|\sum_{i=1}^n\psi(W^Tx_i)\Delta_i\right|\mathbb{I}\left\{\left|\sum_{i=1}^n\bar{\psi}(W^Tx_i)\Delta_i\right|\geq 6, 1/2\leq \|W\|^2\leq 2\right\}\right) \\
 &\geq& \mathbb{P}\left(\left|\sum_{i=1}^n\bar{\psi}(W^Tx_i)\Delta_i\right|\geq 6, 1/2\leq \|W\|^2\leq 2\right) \\
 &=& \mathbb{P}\left(\left|\sum_{i=1}^n\bar{\psi}(W^Tx_i)\Delta_i\right|\geq 6\Big|1/2\leq \|W\|^2\leq 2\right)\mathbb{P}\left(1/2\leq \|W\|^2\leq 2\right) \\
 &\geq& \mathbb{P}\left(\left|\sum_{i=1}^n\bar{\psi}(W^Tx_i)\Delta_i\right|\geq 6\Big|1/2\leq \|W\|^2\leq 2\right)\left(1-2\exp(-d/16)\right),
\end{eqnarray*}
where the last inequality is by Lemma \ref{lem:chi-squared}. By direct calculation, we have
\begin{equation}
\Var\left(\bar{\psi}(W^Tx)|W\right)=\|W\|^2\Var(\max(0,W^Tx/\|W\|)|W)=\|W\|^2\frac{1-\pi^{-1}}{2}, \label{eq:cond-var-X-W}
\end{equation}
and
$$\mathbb{E}\left(|\bar{\psi}(W^Tx)|^3|W\right) \leq 3\mathbb{E}\left(|\psi(W^Tx)|^3|W\right)+3|h(W)|^3 \leq \frac{3}{2}\|W\|^3.$$
Therefore, by Lemma \ref{lem:stein}, we have
\begin{eqnarray*}
 && \mathbb{P}\left(\left|\sum_{i=1}^n\bar{\psi}(W^Tx_i)\Delta_i\right|\geq 6\Big|1/2\leq \|W\|^2\leq 2\right) \\
 &\geq& \mathbb{P}\left(\frac{\left|\sum_{i=1}^n\bar{\psi}(W^Tx_i)\Delta_i\right|}{\|W\|\sqrt{\frac{1-\pi^{-1}}{2}}}\geq 21\Bigg|1/2\leq \|W\|^2\leq 2\right) \\
 &\geq& \mathbb{P}\left(N(0,1)>21\right) - \sup_{1/2\leq \|W\|^2\leq 2} 2\sqrt{3\sum_{i=1}^n|\Delta_i|^3\frac{\mathbb{E}\left(|\bar{\psi}(W^Tx_i)|^3|W\right)}{\|W\|^3\left(\frac{1-\pi^{-1}}{2}\right)^{3/2}}} \\
 &\geq& \mathbb{P}\left(N(0,1)>21\right) - 10\sqrt{\sum_{i=1}^n|\Delta_i|^3} \\
 &\geq& \mathbb{P}\left(N(0,1)>21\right) - 10\max_{1\leq i\leq n}|\Delta_i|^{3/2}.
\end{eqnarray*}
Hence, when $\max_{1\leq i\leq n}|\Delta_i|^{3/2}\leq \delta_0^{3/2}:=\mathbb{P}\left(N(0,1)>21\right)/20$ and $\left|\sum_{i=1}^n\Delta_i\right|< 7$, we can lower bound $\mathbb{E}f(W,X,\Delta)$ by an absolute constant, and we conclude that
\begin{equation}
\inf_{\|\Delta\|=1,|\sum_{i=1}^n\Delta_i|\leq 7, \max_{1\leq i\leq n}|\Delta_i|\leq\delta_0}\mathbb{E}f(W,X,\Delta) \gtrsim 1.\label{eq:l1-1-2-relu}
\end{equation}

Finally, we consider the case when $\max_{1\leq i\leq n}|\Delta_i|> \delta_0$ and $\left|\sum_{i=1}^n\Delta_i\right|< 7$. Without loss of generality, we can assume $\Delta_1>\delta_0$. Note that the lower bound (\ref{eq:seven}) still holds, and thus we have
$$\left|\sum_{i=1}^n\psi(W^Tx_i)\Delta_i\right|\geq  \bar{\psi}(W^Tx_1)\Delta_1 - \left|\sum_{i=2}^n\bar{\psi}(W^Tx_i)\Delta_i\right| - \frac{7}{\sqrt{2\pi}}\|W\|.$$
We then lower bound $\mathbb{E}f(W,X,\Delta)$ by
\begin{eqnarray*}
&& \mathbb{E}\left(\left|\sum_{i=1}^n\psi(W^Tx_i)\Delta_i\right|\mathbb{I}\left\{ \bar{\psi}(W^Tx_1)\Delta_1 \geq 8, \left|\sum_{i=2}^n\bar{\psi}(W^Tx_i)\Delta_i\right|\leq 2, 1/2\leq \|W\|^2\leq 2\right\}\right) \\
&\geq& \mathbb{P}\left(\bar{\psi}(W^Tx_1)\Delta_1 \geq 8, \left|\sum_{i=2}^n\bar{\psi}(W^Tx_i)\Delta_i\right|\leq 2\Big|1/2\leq \|W\|^2\leq 2\right)\mathbb{P}\left(1/2\leq \|W\|^2\leq 2\right) \\
&\geq& \mathbb{P}\left(\bar{\psi}(W^Tx_1)\Delta_1 \geq 8\Big|1/2\leq \|W\|^2\leq 2\right) \\
&& \times \mathbb{P}\left(\left|\sum_{i=2}^n\bar{\psi}(W^Tx_i)\Delta_i\right|\leq 2\Big|1/2\leq \|W\|^2\leq 2\right)\left(1-2\exp(-d/16)\right).
\end{eqnarray*}
For any $W$ that satisfies $1/2\leq \|W\|^2\leq 2$, we have
\begin{eqnarray*}
\mathbb{P}\left(\bar{\psi}(W^Tx_1)\Delta_1 \geq 8\Big|W\right) &\geq& \mathbb{P}\left(\bar{\psi}(W^Tx_1)\geq 8/\delta_0\Big|W\right) \\
&\geq&  \mathbb{P}\left(\psi(W^Tx_1)\geq 8/\delta_0+1/\sqrt{\pi}\Big|W\right) \\
&\geq& \mathbb{P}\left(W^Tx_1\geq 8/\delta_0+1/\sqrt{\pi}\Big|W\right) \\
&\geq& \mathbb{P}\left(N(0,1)\geq \sqrt{2}8/\delta_0 + \sqrt{2/\pi}\right),
\end{eqnarray*}
which is a constant. We also have
\begin{eqnarray*}
&& \mathbb{P}\left(\left|\sum_{i=2}^n\bar{\psi}(W^Tx_i)\Delta_i\right|\leq 2\Big|1/2\leq \|W\|^2\leq 2\right) \\
&\geq& 1 - \frac{1}{4}\Var\left(\sum_{i=2}^n\bar{\psi}(W^Tx_i)\Delta_i\Big|W\right) \\
&\geq& \frac{1}{2},
\end{eqnarray*}
where the last inequality is by (\ref{eq:cond-var-X-W}). Therefore, we have
$$\mathbb{E}f(W,X,\Delta)\geq \frac{1}{2}\left(1-2\exp(-d/16)\right)\mathbb{P}\left(N(0,1)\geq \sqrt{2}8/\delta_0 + \sqrt{2/\pi}\right)\gtrsim 1,$$
and we can conclude that
\begin{equation}
\inf_{\|\Delta\|=1,|\sum_{i=1}^n\Delta_i|\leq 7, \max_{1\leq i\leq n}|\Delta_i|\geq\delta_0}\mathbb{E}f(W,X,\Delta) \gtrsim 1.\label{eq:l1-1-3-relu}
\end{equation}

In the end, we combine the three cases (\ref{eq:l1-1-1}), (\ref{eq:l1-1-2-relu}), and (\ref{eq:l1-1-3-relu}),  and we obtain the conclusion that $\inf_{\|\Delta\|=1}\mathbb{E}f(W,X,\Delta)\gtrsim 1$.

\paragraph{Analysis of (\ref{eq:ep-f}).} This step follows the same analysis of (\ref{eq:ep-f-relu}) in the proof of Lemma \ref{lem:design-rf}, and we have
$$\sup_{\|\Delta\|=1}\left|f(W,X,\Delta)-\mathbb{E}(f(W,X,\Delta)|X)\right|\lesssim \sqrt{\frac{n^2}{p}},$$
with high probability.

\paragraph{Analysis of (\ref{eq:ep-g}).} This step follows a similar analysis of (\ref{eq:ep-g-relu}) in the proof of Lemma \ref{lem:design-rf}. The only difference is that the bound $\mathbb{E}g(X,\Delta)\leq 1$ there can be replaced by $\mathbb{E}g(X,\Delta)\leq \sqrt{n}$, because
$$\mathbb{E}g(X,\Delta)\leq \mathbb{E}\sqrt{\sum_{i=1}^n|\psi(W^Tx_i)|^2}\leq \sqrt{\sum_{i=1}^n\mathbb{E}|\psi(W^Tx_i)|^2}\leq \sqrt{n}.$$
Therefore,
$$\sup_{\|\Delta\|=1}|g(X,\Delta) - \mathbb{E}g(X,\Delta)|\lesssim \sqrt{\frac{n\log(1+2/\zeta)}{d}} + \sqrt{n}\zeta,$$
with high probability as long as $\zeta\leq 1/2$. We choose $\zeta=\frac{c}{\sqrt{n}}$ with a sufficiently small constant $c>0$, and thus the bound is sufficiently small as long as $\frac{n\log n}{d}$ is sufficiently small.

Finally, combine results for (\ref{eq:exp-f-inf}), (\ref{eq:ep-f}) and (\ref{eq:ep-g}), and we obtain the desired conclusion as long as $n^2/p$ and $n\log n/d$ are sufficiently small.
\end{proof}

To prove (\ref{eq:l2-upper-A}) of Lemma \ref{lem:design-rf-relu}, we establish the following stronger result.
\begin{lemma}\label{lem:lim-G-relu}
Consider independent $W_1,...,W_p\sim N(0,d^{-1}I_d)$ and $x_1,...,x_n\sim N(0,I_d)$. We define the matrices $G,\bar{G}\in\mathbb{R}^{n\times n}$ by
$$
G_{il}=\frac{1}{p}\sum_{j=1}^p\psi(W^T_jx_i)\psi(W_j^Tx_l),
$$
and
$$\bar{G}_{il}=\begin{cases}
\frac{1}{2}, & i=l, \\
\frac{1}{2\pi}+\frac{1}{4}\frac{\bar{x}_i^T\bar{x}_l}{d} + \frac{1}{2\pi}\left(\frac{\|x_i\|}{\sqrt{d}}-1+\frac{\|x_l\|}{\sqrt{d}}-1\right), & i\neq l.
\end{cases}$$
Assume $d/\log n$ is sufficiently large, and then
$$\opnorm{G-\bar{G}}^2\lesssim \frac{n^2}{p} + \frac{\log n}{d} + \frac{n^2}{d^2},$$
with high probability. Moreover, we also have $\opnorm{G}\lesssim n$ with high probability.
\end{lemma}
\begin{proof}
Define $\wt{G}\in\mathbb{R}^{n\times n}$ with entries $\wt{G}_{il}=\mathbb{E}\left(\psi(W^Tx_i)\psi(W^Tx_l)|X\right)$, and we first bound the difference between $G$ and $\wt{G}$. Note that
$$\mathbb{E}(G_{il}-\wt{G}_{il})^2 = \mathbb{E}\Var(G_{il}|X) \leq \frac{1}{p}\mathbb{E}|\psi(W^Tx_i)\psi(W^Tx_l)|^2=\frac{3}{2p}\mathbb{E}\|W\|^4\leq 5p^{-1}.$$
We then have
$$
\mathbb{E}\opnorm{G-\wt{G}}^2 \leq \mathbb{E}\fnorm{G-\wt{G}}^2 \leq \frac{5n^2}{p}.
$$
By Markov's inequality,
\begin{equation}
\opnorm{G-\wt{G}}^2 \lesssim \frac{n^2}{p}, \label{eq:G-G-tilde-relu}
\end{equation}
with high probability.

Next, we study the diagonal entries of $\wt{G}$. For any $i\in[n]$, $\wt{G}_{ii}=\mathbb{E}(|\psi(W^Tx_i)|^2|X)=\frac{\|x_i\|^2}{2d}$. By Lemma \ref{lem:chi-squared} and a union bound argument, we have
\begin{equation}
\max_{1\leq i\leq n}|\wt{G}_{ii}-\bar{G}_{ii}|\lesssim \sqrt{\frac{\log n}{d}}, \label{eq:G-diag-relu}
\end{equation}
with high probability.

Now we analyze the off-diagonal entries. We use the notation $\bar{x}_i=\frac{\sqrt{d}}{\|x_i\|}x_i$. For any $i\neq l$, we have
\begin{eqnarray}
\label{eq:G-tilde-1-relu} \wt{G}_{il} &=& \mathbb{E}\left(\psi(W^T\bar{x}_i)\psi(W^T\bar{x}_l)|X\right) \\
\label{eq:G-tilde-2-relu} && + \mathbb{E}\left((\psi(W^Tx_i)-\psi(W^T\bar{x}_i))\psi(W^T\bar{x}_l)|X\right) \\
\label{eq:G-tilde-3-relu} && + \mathbb{E}\left(\psi(W^T\bar{x}_i)(\psi(W^Tx_l)-\psi(W^T\bar{x}_l))|X\right) \\
\label{eq:G-tilde-4-relu} && + \mathbb{E}\left((\psi(W^Tx_i)-\psi(W^T\bar{x}_i))(\psi(W^Tx_l)-\psi(W^T\bar{x}_l))|X\right).
\end{eqnarray}
For the first term on the right hand side of (\ref{eq:G-tilde-1-relu}), we observe that $\mathbb{E}\left(\psi(W^T\bar{x}_i)\psi(W^T\bar{x}_l)|X\right)$ is a function of $\frac{\bar{x}_i^T\bar{x}_l}{d}$, and thus we can write
$$\mathbb{E}\left(\psi(W^T\bar{x}_i)\psi(W^T\bar{x}_l)|X\right)=f\left(\frac{\bar{x}_i^T\bar{x}_l}{d}\right),$$
where
$$f(\rho) = \begin{cases}
\mathbb{E}\psi(\sqrt{1-\rho}U+\sqrt{\rho}Z)\psi(\sqrt{1-\rho}V+\sqrt{\rho}Z), & \rho \geq 0, \\
\mathbb{E}\psi(\sqrt{1+\rho}U-\sqrt{-\rho}Z)\psi(\sqrt{1+\rho}V+\sqrt{-\rho}Z), & \rho < 0,
\end{cases}$$
with $U,V,Z\stackrel{iid}{\sim} N(0,1)$. By some direct calculations, we have $f(0)=\frac{1}{2\pi}$, $f'(0)=\frac{1}{4}$, and $\sup_{|\rho|\leq 0.2}\frac{|f'(\rho)-f'(0)|}{|\rho|}\lesssim 1$. Therefore, as long as $|\bar{x}_i^T\bar{x}_l|/d\leq 1/5$,
$$\left|f\left(\frac{\bar{x}_i^T\bar{x}_l}{d}\right)-\frac{1}{2\pi}-\frac{1}{4}\frac{\bar{x}_i^T\bar{x}_l}{d}\right|\leq C_1\left|\frac{\bar{x}_i^T\bar{x}_l}{d}\right|^2,$$
for some constant $C_1>0$. By Lemma \ref{lem:inner-prod}, we know that $\max_{i\neq l}|\bar{x}_i^T\bar{x}_l|/d\lesssim \sqrt{\frac{\log n}{d}}\leq 1/5$ with high probability, which then implies
\begin{equation}
\sum_{i\neq l}\left(\mathbb{E}\left(\psi(W^T\bar{x}_i)\psi(W^T\bar{x}_l)|X\right)-\bar{G}_{il}\right)^2 \leq C_1\sum_{i\neq l}\left|\frac{\bar{x}_i^T\bar{x}_l}{d}\right|^4.\label{eq:4th-bd-later}
\end{equation}
The term on the right hand side has been analyzed in (\ref{eq:G-H}), and we have $\sum_{i\neq l}\left|\frac{x_i^Tx_l}{d}\right|^4\lesssim \frac{n^2}{d^2}$ with high probability.

We also need to analyze the contributions of (\ref{eq:G-tilde-2-relu}) and (\ref{eq:G-tilde-3-relu}). Observe the fact that $\mathbb{I}\{W^Tx_i\geq 0\}=\mathbb{I}\{W^T\bar{x}_i\geq 0\}$, which implies
\begin{eqnarray}
\nonumber \psi(W^Tx_i)-\psi(W^T\bar{x}_i) &=& W^T(x_i-\bar{x}_i)\mathbb{I}\{W^T\bar{x}_i\geq 0\}\psi(W^T\bar{x}_l) \\
\label{eq:interesting-rep} &=& \left(\frac{\|x_i\|}{\sqrt{d}}-1\right)\psi(W^T\bar{x}_i)\psi(W^T\bar{x}_l).
\end{eqnarray}
Then, the sum of (\ref{eq:G-tilde-2-relu}) and (\ref{eq:G-tilde-3-relu}) can be written as
$$\left(\frac{\|x_i\|}{\sqrt{d}}-1+\frac{\|x_l\|}{\sqrt{d}}-1\right)f\left(\frac{\bar{x}_i^T\bar{x}_l}{d}\right).$$
Note that
\begin{eqnarray*}
&& \sum_{i\neq l} \left(\frac{\|x_i\|}{\sqrt{d}}-1+\frac{\|x_l\|}{\sqrt{d}}-1\right)^2\left[f\left(\frac{\bar{x}_i^T\bar{x}_l}{d}\right)-\frac{1}{2\pi}\right]^2 \\
&\lesssim& \sum_{i\neq l} \left(\frac{\|x_i\|}{\sqrt{d}}-1+\frac{\|x_l\|}{\sqrt{d}}-1\right)^4 + \sum_{i\neq l}\left|\frac{\bar{x}_i^T\bar{x}_l}{d}\right|^4.
\end{eqnarray*}
We have already shown that $\sum_{i\neq l}\left|\frac{\bar{x}_i^T\bar{x}_l}{d}\right|^4\lesssim \frac{n^2}{d^2}$ with high probability. By integrating out the probability tail bound of Lemma \ref{lem:chi-squared}, we have $\mathbb{E}\left(\frac{\|x_i\|}{\sqrt{d}}-1\right)^4\lesssim d^{-2}$, which then implies
$$\sum_{i\neq l} \mathbb{E}\left(\frac{\|x_i\|}{\sqrt{d}}-1+\frac{\|x_l\|}{\sqrt{d}}-1\right)^4\lesssim \frac{n^2}{d^2}$$
and the corresponding high-probability bound by Markov's inequality.

Finally, we show that the contribution of (\ref{eq:G-tilde-4-relu}) is negligible. By (\ref{eq:interesting-rep}), we can write (\ref{eq:G-tilde-4-relu}) as
$$\left(\frac{\|x_i\|}{\sqrt{d}}-1\right)\left(\frac{\|x_l\|}{\sqrt{d}}-1\right)\mathbb{E}\left(\psi(W^T\bar{x}_i)^2\psi(W^T\bar{x}_l)^2\Big|X\right),$$
whose absolute value can be bounded by $\frac{3}{2}\left|\frac{\|x_i\|}{\sqrt{d}}-1\right|\left|\frac{\|x_l\|}{\sqrt{d}}-1\right|$. Since
$$\sum_{i\neq l}\mathbb{E}\left(\frac{\|x_i\|}{\sqrt{d}}-1\right)^2\mathbb{E}\left(\frac{\|x_l\|}{\sqrt{d}}-1\right)^2\lesssim \frac{n^2}{d^2},$$
we can conclude that (\ref{eq:G-tilde-4-relu}) is bounded by $O\left(\frac{n^2}{d^2}\right)$ with high probability by Markov's inequality.

Combining the analyses of (\ref{eq:G-tilde-1-relu}), (\ref{eq:G-tilde-2-relu}), (\ref{eq:G-tilde-3-relu}) and (\ref{eq:G-tilde-4-relu}), we conclude that $\sum_{i\neq l}(\wt{G}_{il}-\bar{G}_{il})^2\lesssim \frac{n^2}{d^2}$ with high probability. Together with (\ref{eq:G-G-tilde-relu}) and (\ref{eq:G-diag-relu}), we obtain the desired bound for $\opnorm{G-\bar{G}}$.

To prove the last conclusion $\opnorm{\bar{G}}\lesssim n$, it suffices to analyze $\lambda_{\max}(\bar{G})$. We bound this quantity by $\mathbb{E}\lambda_{\max}(\bar{G})^2\leq \mathbb{E}\fnorm{\bar{G}}^2\lesssim n^2$, which leads to the desired conclusion.
\end{proof}

\begin{proof}[Proof of Corollary \ref{cor:repair-rf-relu}]
Since $\wh{\theta}$ belongs to the row space of $\wt{X}$, there exists some $u^*\in\mathbb{R}^n$ such that $\wh{\theta}=\wt{X}^Tu^*$.
By Theorem \ref{thm:main-improved} and Lemma \ref{lem:design-rf-relu}, we know that $\wt{u}=u^*$ with high probability, and therefore $\wt{\theta}=\wt{X}^T\wt{u}=\wt{X}^Tu^*=\wh{\theta}$.
\end{proof}

\subsection{Proof of Theorem \ref{thm:nn-grad-relu}}

To prove Theorem \ref{thm:nn-grad-relu}, we need the following kernel random matrix result.
\begin{lemma}\label{lem:lim-H-relu}
Consider independent $W_1,\ldots,W_p\sim N(0,d^{-1}I_d)$, $x_1,\ldots,x_n\sim N(0,I_d)$, and parameters $\beta_1,\ldots,\beta_p\sim N(0,1)$. We define the matrices $H, \bar{H}\in\mathbb{R}^{n\times n}$ by
\begin{eqnarray*}
H_{il} &=& \frac{x_i^Tx_l}{d}\frac{1}{p}\sum_{j=1}^p\beta_j^2\mathbb{I}\{W_j^Tx_i\geq 0, W_j^Tx_l\geq 0\}, \\
\bar{H}_{il} &=& \frac{1}{4}\frac{x_i^Tx_l}{\|x_i\|\|x_l\|} + \frac{1}{4}\mathbb{I}\{i=l\}.
\end{eqnarray*}
Assume $d/\log n$ is sufficiently large, and then
$$\opnorm{H-\bar{H}}^2 \lesssim \frac{n^2}{pd} + \frac{n}{p} + \frac{\log n}{d} + \frac{n^2}{d^2},$$
with high probability. If we additionally assume that $d/n$ and $p/n$ are sufficiently large, we will also have
\begin{equation}
\frac{1}{5}\leq\lambda_{\min}(H)\leq\lambda_{\max}(H)\lesssim 1,\label{eq:last-added-ref}
\end{equation}
with high probability.
%If we assume that $d/n$ and $p/n$ are sufficiently large, we will also have
%$$\lambda_{\min}(H)\geq \frac{1}{8},$$
%with high probability.
\end{lemma}
\begin{proof}
Define $\wt{H}\in\mathbb{R}^{n\times n}$ with entries $\wt{H}_{il}=\frac{x_i^Tx_l}{d}\mathbb{E}\left(\beta^2\mathbb{I}\{W^Tx_i\geq 0, W^Tx_l\geq 0\}\big|X\right)$, and we first bound the difference between $H$ and $\wt{H}$. Note that
$$\mathbb{E}(H_{il}-\wt{H}_{il})^2=\mathbb{E}\Var(H_{il}|X)\leq \frac{1}{p}\mathbb{E}\left(\frac{|x_i^Tx_l|^2}{d^2}\beta^4\right)\leq\begin{cases}
\frac{3}{pd}, & i\neq l, \\
9p^{-1}, & i=l.
\end{cases}$$
We then have
$$\mathbb{E}\opnorm{H-\wt{H}}^2 \leq \mathbb{E}\fnorm{H-\wt{H}}^2 \leq \frac{3n^2}{pd} + \frac{9n}{p}.$$
By Markov's inequality,
\begin{equation}
\opnorm{H-\wt{H}}^2 \lesssim \frac{n^2}{pd} + \frac{n}{p}, \label{eq:H-H-tilde-relu}
\end{equation}
with high probability.

Next, we study the diagonal entries of $\wt{H}$. For any $i\in[n]$, $\wt{H}_{ii}=\frac{\|x_i\|^2}{d}\mathbb{E}(\beta^2\mathbb{I}\{W^Tx_i\geq 0\}|X)=\frac{\|x_i\|^2}{2d}$. The same analysis that leads to the bound (\ref{eq:G-diag-relu}) also implies that
\begin{equation}
\max_{1\leq i\leq n}|\wt{H}_{ii}-\bar{H}_{ii}|\lesssim \sqrt{\frac{\log n}{d}}, \label{eq:H-diag-relu}
\end{equation}
with high probability.

Now we analyze the off-diagonal entries. Recall the notation $\bar{x}_i=\frac{\sqrt{d}}{\|x_i\|}x_i$. For any $i\neq l$, we have
\begin{eqnarray}
\nonumber \wt{H}_{il} &=& \frac{\|x_i\|\|x_l\|}{d}\frac{\bar{x}_i^T\bar{x}_l}{d}\mathbb{P}\left(W^T\bar{x}_i\geq 0, W^T\bar{x}_l\geq 0|X\right) \\
\label{eq:H-tilde-il-relu} &=& \frac{\bar{x}_i^T\bar{x}_l}{d}\mathbb{P}\left(W^T\bar{x}_i\geq 0, W^T\bar{x}_l\geq 0|X\right) \\
\nonumber && + \left(\frac{\|x_i\|\|x_l\|}{d}-1\right)\frac{\bar{x}_i^T\bar{x}_l}{d}\mathbb{P}\left(W^T\bar{x}_i\geq 0, W^T\bar{x}_l\geq 0|X\right).
\end{eqnarray}
Since $\mathbb{P}\left(W^T\bar{x}_i\geq 0, W^T\bar{x}_l\geq 0|X\right)$ is a function of $\frac{\bar{x}_i^T\bar{x}_l}{d}$, we can write
\begin{equation}
\frac{\bar{x}_i^T\bar{x}_l}{d}\mathbb{P}\left(W^T\bar{x}_i\geq 0, W^T\bar{x}_l\geq 0|X\right)=f\left(\frac{\bar{x}_i^T\bar{x}_l}{d}\right), \label{eq:f-H-tilde-relu}
\end{equation}
where for $\rho>0$,
\begin{eqnarray*}
f(\rho) &=& \rho\mathbb{P}\left(\sqrt{1-\rho}U+\sqrt{\rho}Z\geq 0, \sqrt{1-\rho}V+\sqrt{\rho}Z\geq0\right) \\
&=& \rho\mathbb{E}\mathbb{P}\left(\sqrt{1-\rho}U+\sqrt{\rho}Z\geq 0, \sqrt{1-\rho}V+\sqrt{\rho}Z\geq 0|Z\right) \\
&=& \rho\mathbb{E}\Phi\left(\sqrt{\frac{\rho}{1-\rho}}Z\right)^2,
\end{eqnarray*}
with $U,V,Z\stackrel{iid}{\sim} N(0,1)$ and $\Phi(\cdot)$ being the cumulative distribution function of $N(0,1)$. Similarly, for $\rho<0$,
$$f(\rho) = \rho\mathbb{E}\left[\Phi\left(\sqrt{\frac{-\rho}{1+\rho}}Z\right)\left(1-\Phi\left(\sqrt{\frac{-\rho}{1+\rho}}Z\right)\right)\right].$$
By some direct calculations, we have $f(0)=0$, $f'(0)=\frac{1}{4}$, and
$$\sup_{|\rho|\leq 1/5}|f''(\rho)|\lesssim \sup_{|t|\leq 1/2}\left|\mathbb{E}\phi(tZ)\Phi(tZ)Z/t\right| + \sup_{|t|\leq 1/2}\left|\mathbb{E}\phi(tZ)Z/t\right|,$$
where $\phi(x)=(2\pi)^{-1/2}e^{-x^2/2}$. For any $|t|\leq 1/2$,
$$
\left|\mathbb{E}\phi(tZ)Z/t\right| = \left|\mathbb{E}\frac{\phi(tZ)-\phi(0)}{tZ}Z^2\right| = \left|\mathbb{E}\xi\phi(\xi)Z^2\right| \leq \frac{|t|}{\sqrt{2\pi}}\mathbb{E}|Z|^3\lesssim 1,
$$
where $\xi$ is a scalar between $0$ and $tZ$ so that $|\xi|\leq |tZ|$. By a similar argument, we also have $\sup_{|t|\leq 1/2}\left|\mathbb{E}\phi(tZ)\Phi(tZ)Z/t\right|\lesssim 1$ so that $\sup_{|\rho|\leq 1/5}|f''(\rho)|\lesssim 1$. Therefore, as long as $|\bar{x}_i^T\bar{x}_l|/d\leq 1/5$,
$$\left|f\left(\frac{\bar{x}_i^T\bar{x}_l}{d}\right)-\frac{1}{4}\frac{\bar{x}_i^T\bar{x}_l}{d}\right|\leq C_1\left|\frac{\bar{x}_i^T\bar{x}_l}{d}\right|^2,$$
for some constant $C_1>0$. By Lemma \ref{lem:inner-prod}, we know that $\max_{i\neq l}|\bar{x}_i^T\bar{x}_l|/d\lesssim \sqrt{\frac{\log n}{d}}\leq 1/5$ with high probability. In view of the identities (\ref{eq:H-tilde-il-relu}) and (\ref{eq:f-H-tilde-relu}), we then have the high probability bound,
\begin{eqnarray}
\nonumber \sum_{i\neq l}\left(\wt{H}_{il}-\frac{1}{4}\frac{\bar{x}_i^T\bar{x}_l}{d}\right)^2 &\leq& 2\sum_{i\neq l}\left(\frac{\|x_i\|\|x_l\|}{d}-1\right)^2\left|\frac{\bar{x}_i^T\bar{x}_l}{d}\right|^2  + 2C_1\sum_{i\neq l}\left|\frac{\bar{x}_i^T\bar{x}_l}{d}\right|^4 \\
 \label{eq:high-prob-off-diag} &\leq& \sum_{i\neq l}\left(\frac{\|x_i\|\|x_l\|}{d}-1\right)^4 + (2C_1+1)\sum_{i\neq l}\left|\frac{\bar{x}_i^T\bar{x}_l}{d}\right|^4.
\end{eqnarray}
For the first term on the right hand side of (\ref{eq:high-prob-off-diag}), we use Lemma \ref{lem:inner-prod} and obtain a probability tail bound for $|\|x_i\|\|x_l\|-d|$. By integrating out this tail bound, we have
$$\sum_{i\neq l}\mathbb{E}\left(\frac{\|x_i\|\|x_l\|}{d}-1\right)^4\lesssim \frac{n^2}{d^2},$$
which, by Markov's inequality, implies $\sum_{i\neq l}\left(\frac{\|x_i\|\|x_l\|}{d}-1\right)^4\lesssim \frac{n^2}{d^2}$ with high probability.
Using the same argument in the proof of Lemma \ref{lem:lim-G-relu}, we have $\sum_{i\neq l}\left|\frac{x_i^Tx_l}{d}\right|^4\lesssim \frac{n^2}{d^2}$ with high probability.
Finally, combining (\ref{eq:H-H-tilde-relu}), (\ref{eq:H-diag-relu}), and the bound for (\ref{eq:high-prob-off-diag}), we obtain the desired bound for $\opnorm{H-\bar{H}}$.
%The lower bound for $\lambda_{\min}(H)$ is a direct application of Weyl's inequality.
The last conclusion (\ref{eq:last-added-ref}) follows a similar argument in the proof of Lemma \ref{lem:lim-G}. The proof is complete.
\end{proof}

Now we are ready to prove Theorem \ref{thm:nn-grad-relu}.
\begin{proof}[Proof of Theorem \ref{thm:nn-grad-relu}]
The proof is similar to that of Theorem \ref{thm:nn-grad}, and we will omit repeated arguments. We will use the high-probability inequalities (\ref{eq:r1e1})-(\ref{eq:r1e9}). Then, it suffices to establish Claims A, B, C and D in the proof of Theorem \ref{thm:nn-grad}. Since Claims A and C follow the same argument, we only need to check Claims B and D. Given the similarity of Claims B and D, we only present the proof of Claim D.
We have
\begin{equation}
u(k+1)-u(k)=\gamma(H(k)+G(k))(y-u(k))+r(k), \label{eq:iter-u-relu}
\end{equation}
where
\begin{eqnarray*}
G_{il}(k) &=& \frac{1}{p}\sum_{j=1}^p\psi(W_j(k)^Tx_l)\psi(W_j(k)^Tx_i), \\
H_{il}(k) &=& \frac{x_i^Tx_l}{d}\frac{1}{p}\sum_{j=1}^p\beta_j(k+1)^2\psi'(W_j(k)^Tx_i)\psi'(W_j(k)^Tx_l),
\end{eqnarray*}
and
\begin{eqnarray*}
r_i(k) &=& \frac{1}{\sqrt{p}}\sum_{j=1}^p\beta_j(k+1)\left(\psi(W_j(k+1)^Tx_i)-\psi(W_j(k)^Tx_i)\right) \\
&& - \frac{1}{\sqrt{p}}\sum_{j=1}^p\beta_j(k+1)(W_j(k+1)-W_j(k))^Tx_i\psi'(W_j(k)^Tx_i).
\end{eqnarray*}

With the same argument, the bound (\ref{eq:x-japan}) still holds.
By Lemma \ref{lem:lim-G-relu} and the fact that $G(k)$ is positive semi-definite, we have
\begin{equation}
0 \leq \lambda_{\min}(G(k)) \leq \lambda_{\max}(G(k)) \lesssim n. \label{eq:Gk-spec}
\end{equation}

We also need to control the difference between $H(k)$ and $H(0)$. By the definition, we have
\begin{eqnarray}
\label{eq:r-H-d-1-relu} |H_{il}(k)-H_{il}(0)| &\leq& \left|\frac{x_i^Tx_l}{d}\right|\frac{1}{p}\sum_{j=1}^p|\beta_j(k+1)^2-\beta_j^2(0)| \\
\label{eq:r-H-d-2-relu} && + \left|\frac{x_i^Tx_l}{d}\right|\frac{1}{p}\sum_{j=1}^p\beta_j^2(0)|\psi'(W_j(k)^Tx_i) - \psi'(W_j(0)^Tx_i)| \\
\label{eq:r-H-d-3-relu} && + \left|\frac{x_i^Tx_l}{d}\right|\frac{1}{p}\sum_{j=1}^p\beta_j^2(0)|\psi'(W_j(k)^Tx_l) - \psi'(W_j(0)^Tx_l)|.
\end{eqnarray}
We can bound (\ref{eq:r-H-d-1-relu}) by $\left|\frac{x_i^Tx_l}{d}\right|\frac{1}{p}\sum_{j=1}^pR_2(R_2+2|\beta_j(0)|)$. To bound (\ref{eq:r-H-d-2-relu}), we note that
\begin{eqnarray}
\nonumber |\psi'(W_j(k)^Tx_i) - \psi'(W_j(0)^Tx_i)| &\leq& \mathbb{I}\{|W_j(0)^Tx_i|\leq |(W_j(k)-W_j(0))^Tx_i|\} \\
\label{eq:simon-bound} &\leq& \mathbb{I}\{|W_j(0)^Tx_i|\leq R_1\|x_i\|\},
\end{eqnarray}
which implies
\begin{eqnarray*}
&& \left|\frac{x_i^Tx_l}{d}\right|\frac{1}{p}\sum_{j=1}^p\beta_j^2(0)|\psi'(W_j(k)^Tx_i) - \psi'(W_j(0)^Tx_i)| \\
&\leq& \left|\frac{x_i^Tx_l}{d}\right|\frac{1}{p}\sum_{j=1}^p\beta_j^2(0)\mathbb{I}\{|W_j(0)^Tx_i|\leq R_1\|x_i\|\},
\end{eqnarray*}
and a similar bound holds for (\ref{eq:r-H-d-3-relu}). Then,
\begin{eqnarray*}
\opnorm{H(k)-H(0)} &\leq& \max_{1\leq i\leq n}|H_{ii}(k)-H_{ii}(0)| + \max_{1\leq l\leq n}\sum_{i\in[n]\backslash\{l\}}|H_{il}(k)-H_{il}(0)| \\
&\lesssim& \max_{1\leq i\leq n} \frac{1}{p}\sum_{j=1}^p\beta_j^2(0)\mathbb{I}\{|W_j^T(0)^Tx_i|\leq R_1\|x_i\|\} \\
&& + d^{-1/2}n\max_{1\leq i\leq n}\frac{1}{p}\sum_{j=1}^p\beta_j^2(0)\mathbb{I}\{|W_j^T(0)^Tx_i|\leq R_1\|x_i\|\} \\
&& + \max_{1\leq l\leq n}\sum_{i=1}^n\left|\frac{x_i^Tx_l}{d}\right|R_2\frac{1}{p}\sum_{j=1}^p(R_2+2|\beta_j(0)|) \\
&\lesssim& \left(1+\frac{n}{\sqrt{d}}\right)\left(\sqrt{d}R_1\log p+\frac{\sqrt{\log n}\log p}{\sqrt{p}} + R_2^2 + R_2\sqrt{\log p}\right) \\
&\lesssim& \left(1+\frac{n}{\sqrt{d}}\right)\frac{n(\log p)^2}{\sqrt{p}},
\end{eqnarray*}
where we have used (\ref{eq:r1e1}), (\ref{eq:r1e4}), (\ref{eq:r1e5}), (\ref{eq:r1e6}) and (\ref{eq:r1e9}). In view of Lemma \ref{lem:lim-H-relu}, we then have
\begin{equation}
\frac{1}{6} \leq \lambda_{\min}(H(k)) \leq \lambda_{\max}(H(k)) \lesssim 1, \label{eq:Hk-spec}
\end{equation}
under the conditions of $d,p$ and $n$.

Next, we give a bound for $r_i(k)$. Observe that
$$\psi(W_j(k+1)^Tx_i)-\psi(W_j(k)^Tx_i)=(W_j(k+1)-W_j(k))^Tx_i\psi'(W_j(k)^Tx_i),$$
when $\mathbb{I}\{W_j(k+1)^Tx_i>0\}=\mathbb{I}\{W_j(k)^Tx_i>0\}$. Thus, we only need to sum over those $j\in[p]$ that $\mathbb{I}\{W_j(k+1)^Tx_i>0\}\neq \mathbb{I}\{W_j(k)^Tx_i>0\}$. By (\ref{eq:simon-bound}), we have
\begin{eqnarray*}
&& \left|\mathbb{I}\{W_j(k+1)^Tx_i>0\}-\mathbb{I}\{W_j(k)^Tx_i>0\}\right| \\
&\leq& \left|\mathbb{I}\{W_j(k+1)^Tx_i>0\}-\mathbb{I}\{W_j(0)^Tx_i>0\}\right|+ \left|\mathbb{I}\{W_j(k)^Tx_i>0\}-\mathbb{I}\{W_j(0)^Tx_i>0\}\right| \\
&\leq& 2\mathbb{I}\{|W_j^T(0)^Tx_i|\leq R_1\|x_i\|\}.
\end{eqnarray*}
Therefore,
\begin{eqnarray*}
&& \left|\psi(W_j(k+1)^Tx_i)-\psi(W_j(k)^Tx_i)-(W_j(k+1)-W_j(k))^Tx_i\psi'(W_j(k)^Tx_i)\right| \\
&\leq& 4|(W_j(k+1)-W_j(k))^Tx_i|\mathbb{I}\{|W_j^T(0)^Tx_i|\leq R_1\|x_i\|\} \\
&\leq& \frac{4\gamma}{d\sqrt{p}}|\beta_j(k+1)|\|y-u(k)\|\|x_i\|\sqrt{\sum_{l=1}^n\|x_l\|^2}\mathbb{I}\{|W_j^T(0)^Tx_i|\leq R_1\|x_i\|\},
\end{eqnarray*}
which implies
\begin{eqnarray*}
|r_i(k)| &\leq& \frac{4\gamma}{dp}\sum_{j=1}^p|\beta_j(k+1)|^2\|y-u(k)\|\|x_i\|\sqrt{\sum_{l=1}^n\|x_l\|^2}\mathbb{I}\{|W_j^T(0)^Tx_i|\leq R_1\|x_i\|\} \\
&\lesssim& \sqrt{n}\|y-u(k)\|\gamma\frac{1}{p}\sum_{j=1}^p(\beta_j(0)^2+R_2^2)\mathbb{I}\{|W_j^T(0)^Tx_i|\leq R_1\|x_i\|\} \\
&\lesssim& \gamma\sqrt{n}\log p\left(R_1+\sqrt{\frac{\log n}{p}}\right)\|y-u(k)\|.
\end{eqnarray*}
This leads to the bound
\begin{equation}
\|r(k)\|=\sqrt{\sum_{i=1}^n|r_i(k)|^2}\lesssim \gamma n\log p\left(R_1+\sqrt{\frac{\log n}{p}}\right)\|y-u(k)\|.\label{eq:bound-res-k-relu}
\end{equation}

Now we are ready to analyze $\|y-u(k+1)\|^2$. Given the relation (\ref{eq:iter-u-relu}), we have
\begin{eqnarray*}
\|y-u(k+1)\|^2 &=& \|y-u(k)\|^2 - 2\iprod{y-u(k)}{u(k+1)-u(k)} + \|u(k)-u(k+1)\|^2 \\
&=& \|y-u(k)\|^2 - 2\gamma(y-u(k))^T(H(k)+G(k))(y-u(k)) \\
&& - 2\iprod{y-u(k)}{r(k)} + \|u(k)-u(k+1)\|^2.
\end{eqnarray*}
By (\ref{eq:Gk-spec}) and (\ref{eq:Hk-spec}), we have
\begin{equation}
- 2\gamma(y-u(k))^T(H(k)+G(k))(y-u(k)) \leq -\frac{\gamma}{6}\|y-u(k)\|^2. \label{eq:main-inner-relu}
\end{equation}
The bound (\ref{eq:bound-res-k-relu}) implies
$$- 2\iprod{y-u(k)}{r(k)}\leq 2\|y-u(k)\|\|r(k)\|\lesssim \gamma n\log p\left(R_1+\sqrt{\frac{\log n}{p}}\right)\|y-u(k)\|^2.$$
By (\ref{eq:Gk-spec}), (\ref{eq:Hk-spec}) and (\ref{eq:bound-res-k-relu}), we also have
\begin{eqnarray*}
\|u(k)-u(k+1)\|^2 &\leq& 2\gamma^2\|(H(k)+G(k))(y-u(k))\|^2 + 2\|r(k)\|^2 \\
&\lesssim& \gamma^2n\|y-u(k)\|^2 + (\gamma n\log p)^2\left(R_1+\sqrt{\frac{\log n}{p}}\right)^2\|y-u(k)\|^2 .
\end{eqnarray*}
Therefore, as long as $\frac{n\log n}{d}$, $\frac{n^3(\log p)^4}{p}$ and $\gamma n$ are all sufficiently small, we have
$$- 2\iprod{y-u(k)}{r(k)} + \|u(k)-u(k+1)\|^2 \leq \frac{\gamma}{24}\|y-u(k)\|^2.$$
Together with the bound (\ref{eq:main-inner-relu}), we have
$$\|y-u(k+1)\|^2 \leq \left(1-\frac{\gamma}{8}\right)\|y-u(k)\|^2\leq \left(1-\frac{\gamma}{8}\right)^{k+1}\|y-u(0)\|^2,$$
and thus Claim D is true. The proof is complete.
\end{proof}

\subsection{Proofs of Theorem \ref{thm:repair-nn-1-relu} and Theorem \ref{thm:repair-nn-2-relu}}

\begin{proof}[Proof of Theorem \ref{thm:repair-nn-1-relu}]
The proof is the same as that of Theorem \ref{thm:repair-nn-1}. The only exception here is that we apply Lemma \ref{lem:design-rf-relu} and Lemma \ref{lem:lim-G-relu} instead of Lemma \ref{lem:design-rf} and Lemma \ref{lem:lim-G}.
\end{proof}

\begin{proof}[Proof of Theorem \ref{thm:repair-nn-2-relu}]
The analysis of $\wh{v}_1,...,\wh{v}_p$ is the same as that in the proof of Theorem \ref{thm:repair-nn-1}, and we have $\wt{W}_j=\wh{W}_j$ for all $j\in[p]$ with high probability.

To analyze $\wh{u}$, we apply Theorem \ref{thm:main-improved}. It suffices to check Condition $A$ and Condition $B$ for the design matrix $\psi(X^T\wt{W}^T)=\psi(X^T\wh{W}^T)$. Since
$$\sum_{i=1}^n\mathbb{E}\left(\frac{1}{p}\sum_{j=1}^pc_j\psi(\wh{W}_j^Tx_i)\right)^2\leq \sum_{i=1}^n\frac{1}{p}\sum_{j=1}^p\mathbb{E}\psi(\wh{W}^Tx_i)^2,$$
and $\mathbb{E}\psi(\wh{W}^Tx_i)^2\leq \mathbb{E}|\wh{W}_j^Tx_i|^2\lesssim 1 + R_1d\lesssim 1$, Condition $A$ holds with $\sigma^2\asymp p$.
We also need to check Condition $B$. By Theorem \ref{thm:nn-grad-relu}, we have
\begin{eqnarray*}
&& \left|\frac{1}{p}\sum_{j=1}^p\left|\sum_{i=1}^n\psi(\wh{W}_j^Tx_i)\Delta_i\right| - \frac{1}{p}\sum_{j=1}^p\left|\sum_{i=1}^n\psi(W_j(0)^Tx_i)\Delta_i\right|\right| \\
&\leq& \frac{1}{p}\sum_{j=1}^p\sum_{i=1}^n|\wh{W}_j^Tx_i-W_j(0)^Tx_i||\Delta_i| \\
&\leq& R_1\sum_{i=1}^n\|x_i\||\Delta_i| \\
&\leq& R_1\sqrt{\sum_{i=1}^n\|x_i\|^2} \\
&\lesssim& \frac{n^{3/2}\log p}{\sqrt{p}},
\end{eqnarray*}
where $\sum_{i=1}^n\|x_i\|^2\lesssim nd$ is by Lemma \ref{lem:chi-squared}. By Lemma \ref{lem:design-rf-relu}, we can deduce that
$$\inf_{\|\Delta\|=1}\frac{1}{p}\sum_{j=1}^p\left|\sum_{i=1}^n\psi(\wh{W}_j^Tx_i)\Delta_i\right|\gtrsim 1,$$
as long as $\frac{n^{3/2}\log p}{\sqrt{p}}$ is sufficiently small. By (\ref{eq:Gk-spec}), we also have
$$\sup_{\|\Delta\|=1}\frac{1}{p}\sum_{j=1}^p\left|\sum_{i=1}^n\psi(\wh{W}_j^Tx_i)\Delta_i\right|^2\lesssim n.$$
Therefore, Condition $B$ holds with $\overline{\lambda}^2\asymp n$ and $\underline{\lambda}\asymp 1$. Applying Theorem \ref{thm:main-improved}, we have $\wt{\beta}=\wh{\beta}$ with high probability, as desired.
\end{proof}

%\end{document}
\newpage

\end{document}